\documentclass[final,3p,times]{elsarticle}    

\usepackage{amssymb}
\usepackage{amsmath}
\usepackage{amsthm}
\usepackage{xcolor}
\usepackage{dirtytalk}
\usepackage{tabularx}
\usepackage{url}
\usepackage{booktabs} 
\usepackage{hyperref}
\usepackage{pgfplots}
\usepackage{multicol}
\usepackage{multirow}
\usepackage{caption}
\usepackage{subcaption}
\usepackage{natbib}
\usepackage{algorithm}
\usepackage{algpseudocode}
\usepackage{float}
\usepackage{enumitem}

\newtheorem{theorem}{Theorem}
\newtheorem{lemma}{Lemma}

\newtheorem{remark}{Remark}

\makeatletter
\renewcommand{\theALG@line}{\arabic{ALG@line}}
\makeatother

\journal{Computer Methods in Applied Mechanics and Engineering}

\begin{document}

\begin{frontmatter}

\title{Weak and entropy physics-informed neural networks for conservation laws}

\author[a,b]{Ismail Oubarka\corref{cor1}}
\ead{ismail.oubarka@um6p.ma, oubarka.ismail5@gmail.com}
\cortext[cor1]{Corresponding author}

\author[a]{Imad Kissami}
\author[b]{Mohamed Boubekeur}
\author[b]{Fayssal Benkhaldoun}
\author[c]{Zakaria Saâdi}
\author[d]{Aziz Madrane}

\affiliation[a]{organization={Mohammed VI Polytechnic University, College of Computing},
            addressline={Lot 660}, 
            city={Ben Guerir},
            postcode={43150}, 
            country={Morocco}}
            
\affiliation[b]{organization={LAGA, CNRS, UMR 7539, Université Sorbonne Paris Nord},
            city={Villetaneuse},
            postcode={93430}, 
            country={France}}   
\affiliation[c]{organization={Autorité de Sûreté Nucléaire et de Radioprotection (ASNR), PSE-ENV/SPDR/UEMIS, F-92260},
               city={Fontenay-aux-Roses},
               country={France}}
\affiliation[d]{organization={Department of Mechanical Engineering, McGill University}, 
city={Montreal, QC H3A 2T8},
country={Canada}}

\begin{abstract}
The Weak and Entropy PINNs (WE-PINNs) framework is proposed for the approximation of entropy solutions to nonlinear hyperbolic conservation laws. Standard physics-informed neural networks enforce governing equations in strong differential form, an approach that becomes structurally inconsistent in the presence of discontinuities due to the divergence of strong-form residuals near shocks. The proposed method replaces pointwise residual minimization with a space--time weak formulation derived from the divergence theorem. Conservation is enforced through boundary flux integrals over dynamically sampled space--time control volumes, yielding a mesh-free control-volume framework that remains well-defined for discontinuous solutions. Entropy admissibility is incorporated in integral form through the full Kru\v{z}kov family $\eta_\kappa(u) = |u-\kappa|$, with the entropy index $\kappa$ sampled stochastically during training, enabling the selection of the Oleinik-admissible solution for non-convex fluxes such as Buckley--Leverett. The resulting loss functional combines space--time flux balance and integral entropy inequalities, without resorting to dual-norm saddle-point formulations, auxiliary potential networks, or fixed discretization meshes. The conservation and entropy residuals reduce to boundary integrals and require no automatic differentiation of the neural approximation; a total-variation term is added separately as a soft regularization to support the convergence analysis. The framework is straightforward to implement and shows that an adequate loss function gives excellent results with a simple multi-layer perceptron. We establish a conditional convergence analysis linking the network's discrete losses and explicit sampling-consistency residuals to the $L^1$ error towards the entropy solution, providing the first such $L^1$ estimate for a mesh-free space--time control-volume PINNs formulation with integral Kru\v{z}kov entropy enforcement via the Bouchut--Perthame framework for scalar conservation laws. Numerical experiments on the Burgers equation, the Buckley--Leverett equation, and the compressible Euler equations in one and two spatial dimensions demonstrate accurate shock resolution and robust performance in both smooth and shock-dominated regimes.
\end{abstract}

\begin{keyword} 
PINNs\sep
Hyperbolic conservation laws\sep 
Weak entropy formulation\sep 
Space--time control volumes\sep 
$L^1$ convergence\sep 
TVD regularization
\end{keyword}
\end{frontmatter}

\section{Introduction}\label{sec:introduction}
Nonlinear hyperbolic conservation laws arise in a wide range of scientific and engineering applications, including compressible fluid dynamics, shallow water flows, gas dynamics, and related transport phenomena \cite{Lax1973, LeVeque2002, Dafermos2016, toro2009riemann, aboussi2023highly, ziggaf2020fvc}. A distinctive feature of these models is the spontaneous formation of shock discontinuities, even when the initial conditions are smooth \cite{Dafermos2016, Lax1957, Glimm1965, bressan2000hyperbolic, crandall1980monotone}. Once shocks appear, classical differentiable solutions cease to exist, and the governing equations must be interpreted in the weak sense \cite{Dafermos2016, Kruzkov1970}. Since weak solutions are generally not unique, the physically relevant solution must be selected through admissibility criteria, most notably entropy conditions \cite{Lax1973, Dafermos2016, Kruzkov1970, tadmor1987entropy}. Classical shock-capturing methods, including Godunov-type schemes, ENO and WENO methods, discontinuous Galerkin schemes, and entropy-stable finite-volume formulations, have been developed precisely to approximate such weak entropy solutions while preserving conservation and controlling nonphysical oscillations \cite{LeVeque2002, toro2009riemann, tadmor1987entropy}.

Physics-Informed Neural Networks (PINNs) \cite{Raissi2019} have emerged as a flexible framework for solving partial differential equations by embedding the governing equations directly into the training loss through automatic differentiation \cite{cai2021physics, karniadakis2021physics, cuomo2022scientific}. The approach has been extended to fractional operators \cite{pang2019fpinns}, variational formulations \cite{kharazmi2021hp}, and operator-learning settings through DeepONets \cite{lu2021learning} and Fourier neural operators \cite{li2021fno}. A substantial body of work has also investigated the training pathologies of PINNs through gradient-flow analysis \cite{wang2021understanding}, neural tangent kernel perspectives \cite{wang2022and}, causal training schedules \cite{wang2024causal}, self-adaptive weighting \cite{mcclenny2023self}, and adaptive sampling strategies \cite{wu2023comprehensive}. Rigorous error estimates have also been established for PINNs in several smooth regimes, under suitable regularity and stability assumptions, including representative elliptic/parabolic settings and incompressible Navier--Stokes equations \cite{mishra2023estimates, de2024error}.

However, the direct use of standard PINNs for nonlinear hyperbolic conservation laws remains problematic. In their classical form, PINNs minimize a strong-form residual such as
\[
    \|\partial_t u + \nabla\!\cdot F(u)\|_{L^2}^2,
\]
which assumes the existence of pointwise derivatives. This assumption is incompatible with discontinuous entropy solutions. In shock-dominated regimes, several studies have shown that standard PINNs tend to oversmooth discontinuities, underestimate shock speeds, distort plateau states, or fail to converge to the correct entropy solution on canonical benchmarks such as Burgers, Buckley--Leverett, and Euler equations \cite{Mao2020, jagtap2020conservative, fuks2020limitations, krishnapriyan2021characterizing}.

Recent theoretical analyses have clarified that this failure is not only a matter of architecture or hyperparameter tuning. De Ryck, Mishra and Molinaro~\cite{de2024wpinns} showed that the strong-form residual becomes unbounded as a smooth approximation sharpens toward a discontinuous entropy solution, whereas suitable weak or dual residual norms remain meaningful in the presence of shocks. Wang and Yang~\cite{CIPINN2024} further interpreted this phenomenon as an optimization barrier: the entropy solution is not naturally favored by the strong-form objective, and the loss landscape can push the optimizer toward smoother nonphysical profiles. These results motivate replacing pointwise differential residuals by weak or integral formulations that remain meaningful in the presence of shocks.

Several strategies have been proposed to improve the behavior of PINNs for discontinuous solutions. A first line of work retains the strong-form residual but introduces stabilization mechanisms, such as adaptive residual sampling \cite{Mao2020}, gradient reweighting \cite{wang2021understanding}, adaptive localized artificial viscosity \cite{coutinho2023physics}, gradient annihilation \cite{ferrer2024gradient}, lift-and-embed formulations \cite{sun2026lift}, and explicit enforcement of Rankine--Hugoniot jump conditions \cite{liu2024discontinuity}. Conservative PINNs (cPINNs) \cite{jagtap2020conservative} enforce flux continuity across subdomain interfaces in a domain-decomposition setting. Constraint-aware neural solvers \cite{magiera2020constraint} and RoeNet-type architectures \cite{tong2024roenet} have also been introduced to learn discontinuity-resolving operators or Riemann-solver structures. While these approaches improve empirical robustness, the governing equations are still often enforced through pointwise differential information, and shock resolution can remain sensitive to hyperparameters, sampling strategies, or prior knowledge of discontinuity locations.

A second line of work incorporates finite-volume ideas directly into the learning objective. Control-volume PINNs (cvPINNs) \cite{Patel2022} replace the strong-form residual by least-squares residuals of space--time flux integrals over a fixed partition, supplemented by an artificial viscosity flux proportional to $\Delta x^{2}$ and a discrete cell-based total-variation penalty. The resulting scheme is a hybrid finite-volume solver in which a neural network parameterizes the cell-centered solution, with both artificial viscosity and total-variation control explicitly tied to the mesh size. Closely related Godunov-type losses \cite{cassia2025godunov, patsatzis2025gorinns} and the SFVnet finite-volume-informed U-net \cite{SFVnet2026} embed discrete update rules, approximate Riemann solvers, or numerical fluxes into the training objective, inheriting problem-dependent discretization choices. Two structural limitations remain across these formulations. First, a space--time partition is fixed before training, and the stabilization mechanisms, namely artificial viscosity proportional to the mesh size $\Delta x$ and a cell-based total-variation penalty operating on cell-centered values, are intrinsically tied to this partition and cannot be decoupled from it. Second, entropy enforcement is performed through a single, problem-specific entropy pair, which is insufficient to select the Oleinik solution for non-convex fluxes such as the Buckley--Leverett model. 

A third line replaces the strong $L^2$ residual with weaker objectives that remain bounded in the presence of discontinuities. Weak PINNs (wPINNs) \cite{de2024wpinns} target dual $W^{-1,p}$ norms of the residual through saddle-point optimization with auxiliary test-function networks; this approach was subsequently refined by Chaumet and Giesselmann \cite{Chaumet2024} for improved dual-norm computation, boundary handling, and extension to systems. WF-PINNs \cite{wang2025wf} combine a weak-form integral loss based on smooth compactly supported test functions with an entropy condition for the scalar Burgers equation, representing an important step toward entropy-aware weak training. The Coupled Integral PINNs (CI-PINNs) framework \cite{CIPINN2024} follows another integral route by introducing an auxiliary potential network $\tilde{S}$ such that $\nabla\!\cdot\!\tilde{S} \approx \tilde{q}$ and $\partial_t \tilde{S}_j + F_j(\tilde{q}) \approx 0$, thereby reducing the reliance on direct differentiation of discontinuous primitive states in the main conservation loss. While mathematically appealing, such formulations introduce additional test-function networks, auxiliary potentials, and dual-network coupling
constraints, or consistency losses that increase algorithmic complexity.

Taken together, the existing weak, integral, and finite-volume-informed PINNs formulations do not simultaneously provide (i) a purely integral, mesh-free space--time formulation in which the conservation and entropy residuals reduce to boundary flux evaluations of the neural approximation, with no artificial viscosity and no underlying discretization mesh; (ii) entropy enforcement in integral form for the full Kru\v{z}kov family, enabling the unique selection of the entropy solution for non-convex fluxes; and (iii) a conditional $L^1$ convergence estimate towards the entropy solution, provided that the weak conservation loss, the entropy loss, and the sampling residuals vanish simultaneously. The present work addresses these three points jointly within a single, architecturally simple framework.

We propose Weak and Entropy PINNs (WE-PINNs), a mesh-free space--time control-volume formulation for nonlinear conservation laws. The central idea is to replace the strong-form PDE residual by integral conservation constraints on the boundary $\partial D$ of many sampled space--time control volumes
$D \subset \Omega_{x,t}$,
\[
    \int_{\partial D}
    \Bigl( U\,n_t + F(U)\cdot n_x \Bigr)\, dS = 0,
\]
complemented by integral entropy inequalities and a total-variation control term. The control volumes are re-sampled dynamically at multiple spatial and temporal scales during training, and boundary integrals are evaluated by tensor-product Gauss--Legendre quadrature. The entropy inequality is enforced in integral form for the full Kru\v{z}kov family
$\eta_\kappa(u) = |u-\kappa|$, with the entropy index $\kappa$ sampled stochastically during training, which is essential for selecting the physically admissible solution in the case of non-convex fluxes. By construction, the conservation and entropy residuals of WE-PINNs are evaluated solely through boundary flux integrals over the sampled space--time control volumes, requiring no automatic differentiation of the neural approximation. No artificial viscosity proportional to a mesh size $\Delta x$ is introduced, and no fixed discretization mesh underlies the formulation. A total-variation control term, added separately as a soft regularization to support the convergence analysis, is the only component involving finite-difference evaluations of the network values. The method employs a single feedforward network and requires no auxiliary potential, no test-function network, and no saddle-point optimization. The resulting loss remains meaningful when shocks develop, since it does not require the pointwise existence of $\partial_t u$ or $\nabla u$, enforces conservation in the sense of distributions, implicitly accounts for the Rankine--Hugoniot jump conditions, and selects entropy-admissible solutions through integral Kru\v{z}kov-type inequalities.

Building on this construction, the contributions of this work are threefold. First, we formulate WE-PINNs as a space--time control-volume method that combines integral conservation, integral Kru\v{z}kov entropy enforcement for the full one-parameter family through stochastic Kru\v{z}kov sampling, and total-variation control, within a single architecturally simple framework. Second, we establish a conditional $L^1$ convergence estimate based on the Bouchut--Perthame framework~\cite{bouchut1998} for scalar conservation laws. The estimate links the $L^1$ error to the discrete entropy loss and to explicit sampling-consistency residuals quantifying the gap between the continuous weak residuals and their reconstructions from the sampled control volumes. To our knowledge, this is the first such conditional estimate for a space--time control-volume PINNs with integral entropy enforcement. The result is restricted to scalar conservation laws, in line with the current state of the underlying error theory. Third, we provide empirical validation on scalar problems with convex and non-convex fluxes (Burgers, Buckley--Leverett) and on the compressible Euler system in one and two spatial dimensions, with the system-level experiments reported as robustness studies rather than theoretically guaranteed extensions.

The paper is organized as follows. Section~\ref{sec:methodology} presents the mathematical background and the weak space--time formulation. Section~\ref{sec:wepinns} develops the WE-PINNs framework, including the loss function, sampling strategy, quadrature rules, integral entropy constraints, and total-variation control. Section~\ref{sec:convergence} establishes the conditional $L^1$ convergence analysis for scalar conservation laws. Numerical experiments on the Burgers and Buckley--Leverett equations and on the one- and two-dimensional compressible Euler equations are reported in Section~\ref{sec:results}. Section~\ref{sec:conclusion} concludes the paper.

\section{Mathematical Background and Weak Formulation}
\label{sec:methodology}
Consider a system of nonlinear conservation laws in divergence form
\begin{equation}
\partial_t U + \nabla_x \cdot F(U) = 0,
\qquad U : \Omega_{x,t} \to \mathbb{R}^m,
\label{eq:conservation_law}
\end{equation}
on $\Omega_{x,t} = \Omega_x \times (0,T)$, where $U$ is the vector of conserved variables and $F : \mathbb{R}^m \to \mathbb{R}^{m \times d}$ is the flux, assumed Lipschitz continuous on bounded sets. Componentwise,
\[
F(U) = \big(F^{(1)}(U), \dots, F^{(d)}(U)\big),
\qquad F^{(j)}(U) \in \mathbb{R}^m,
\]
with $F^{(j)}(U)$ the flux in direction $x_j$. Equation~\eqref{eq:conservation_law} expresses local balance: the time variation of $U$ inside any region equals the net flux through its boundary. For the space--time formulation used below, we work in $(d+1)$ dimensions. For a space--time subdomain $D \subset \Omega_{x,t}$, $n = (n_x, n_t) \in \mathbb{R}^{d+1}$ denotes the outward unit normal on $\partial D$, with $n_x \in \mathbb{R}^d$ and $n_t \in \mathbb{R}$, and $dS$ the surface measure on $\partial D$.

\subsection{Weak conservation in space--time}
\label{sec:weak_form}

Solutions of nonlinear conservation laws may develop shock discontinuities in finite time. Once this happens, the classical derivatives $\partial_t U$ and $\nabla_x U$ cease to exist pointwise and the differential form of~\eqref{eq:conservation_law} no longer applies. Conservation must then be interpreted weakly. The conservation law holds in the weak sense if, for any smooth compactly supported test function $\varphi \in C_c^1(\Omega_{x,t}; \mathbb{R}^m)$,
\[
\int_0^T \int_{\Omega_x}
\left(
U \cdot \partial_t \varphi
+ \sum_{j=1}^d F^{(j)}(U) \cdot \partial_{x_j} \varphi
\right) dx\, dt = 0.
\]
This identity remains well-defined when $U$ contains discontinuities, provided $U$ and $F(U)$ are integrable, and accommodates shock solutions directly.

A convenient equivalent representation rewrites \eqref{eq:conservation_law} as a divergence in space--time. With
\[
J(U) := \big(F(U),\, U\big) \in \mathbb{R}^{m \times (d+1)},
\]
the conservation law reads $\nabla_{x,t} \cdot J(U) = 0$. The divergence theorem applied over any subdomain $D \subset \Omega_{x,t}$ then yields
\begin{equation}
\int_{\partial D}
\left( U\, n_t + F(U) \cdot n_x \right) dS = 0,
\qquad \forall\, D \subset \Omega_{x,t}.
\label{eq:weak_conservation}
\end{equation}
The time variation of the conserved quantities inside $D$ is balanced by the net flux across $\partial D$. The equality is vector-valued, with one equation per conserved component. The Rankine--Hugoniot jump conditions are implicit in~\eqref{eq:weak_conservation}, so no explicit shock tracking is needed.

\subsection{Entropy condition and admissibility}
\label{sec:entropy}

Identity~\eqref{eq:weak_conservation} accommodates discontinuous solutions but does not guarantee uniqueness. Nonphysical weak solutions, such as expansion shocks, satisfy the conservation law in the distributional sense without being physically admissible. The entropy condition provides the selection criterion.

An entropy pair consists of a convex entropy function $\eta : \mathbb{R}^m \to \mathbb{R}$ and associated entropy fluxes
\[
q(U) = \big(q^{(1)}(U), \dots, q^{(d)}(U)\big),
\qquad q^{(j)} : \mathbb{R}^m \to \mathbb{R},
\]
linked by the compatibility relations
\[
\nabla_U q^{(j)}(U) = \nabla_U \eta(U)\, \nabla_U F^{(j)}(U),
\qquad j = 1, \dots, d.
\]
For smooth $U$, $\partial_t \eta(U) + \nabla_x \cdot q(U) = 0$. In the presence of discontinuities, entropy may be produced across shocks, and the physically admissible solution satisfies the entropy inequality
\[
\partial_t \eta(U) + \nabla_x \cdot q(U) \leq 0
\]
in the distributional sense. 
For scalar conservation laws, this admissibility condition rules out nonphysical weak solutions and yields uniqueness. For systems, an entropy inequality associated with a chosen mathematical entropy pair provides a physically motivated admissibility constraint, but not a general uniqueness characterization. Integrating over an arbitrary space--time subdomain $D$ and applying the divergence theorem gives the integral form
\begin{equation}
\int_{\partial D}
\left( \eta(U)\, n_t + q(U) \cdot n_x \right) dS \leq 0,
\qquad \forall\, D \subset \Omega_{x,t}.
\label{eq:entropy_inequality}
\end{equation}

For scalar conservation laws, the Kru\v{z}kov family $\eta_\kappa(u) = |u - \kappa|$, $q_\kappa(u) = \mathrm{sgn}(u-\kappa)(F(u) - F(\kappa))$, $\kappa \in \mathbb{R}$, has a specific status: enforcing \eqref{eq:entropy_inequality} for every $\kappa$ is equivalent to the Oleinik admissibility condition and singles out the unique entropy solution, including for non-convex fluxes. The WE-PINNs formulation of Section~\ref{sec:wepinns} and the convergence analysis of Section~\ref{sec:convergence} both use this characterization. For systems, a single mathematical entropy pair $(\eta, q)$ associated with the model is used; the Euler choice is reported in~\ref{app:entropies}.

In the scalar Kru\v{z}kov setting, identities~\eqref{eq:weak_conservation} and~\eqref{eq:entropy_inequality} characterize the entropy solution. In the general system setting, they provide the conservation and entropy-admissibility constraints that form the basis of the loss functional in Section~\ref{sec:wepinns}.


\section{Weak and Entropy PINNs Formulation (WE-PINNs)}
\label{sec:wepinns}

WE-PINNs are built on a space--time integral formulation of conservation and entropy admissibility. Instead of enforcing the governing equations pointwise, the method penalizes violations of \eqref{eq:weak_conservation} and \eqref{eq:entropy_inequality} over a collection of dynamically sampled space--time control volumes. The loss has three components: a weak conservation loss for the integral identity \eqref{eq:weak_conservation} over sampled control volumes (Section~\ref{sec:cons_loss}); an entropy loss for the integral entropy inequality \eqref{eq:entropy_inequality}, applied to the full Kru\v{z}kov family in the scalar case (Section~\ref{sec:ent_loss}); and a total-variation control term used as a soft regularization to support the convergence analysis (Section~\ref{sec:tvd}). The first two components reduce to boundary flux evaluations on each face of the sampled control volumes and require no automatic differentiation of the neural approximation. Only the total-variation term involves finite-difference evaluations of the network values.

\subsection{Space--time control-volume formulation}
\label{sec:cv_formulation}

The space--time conservation identity~\eqref{eq:weak_conservation} is the structural basis of WE-PINNs. The method enforces this identity over space--time control volumes rather than pointwise, which bypasses the need for classical differentiability. We use control volumes of the form
\[
D = K \times [t^-, t^+] \subset \Omega_{x,t},
\]
where $K \subset \Omega_x$ is a bounded spatial region with sufficiently regular boundary $\partial K$. The only constraint on $K$ is the regularity needed for the divergence theorem. Decomposing $\partial D$ into the two time faces $K \times \{t^-\}$, $K \times \{t^+\}$, and the lateral boundary $\partial K \times [t^-, t^+]$, identity \eqref{eq:weak_conservation} reduces to
\[
\int_K U(x, t^+)\, dx
- \int_K U(x, t^-)\, dx
+ \int_{t^-}^{t^+} \int_{\partial K}
F(U) \cdot n_x\, dS_x\, dt
= 0,
\]
which expresses conservation over $K$: the change of the conserved quantities between $t^-$ and $t^+$ equals the net flux through $\partial K$. The formulation does not depend on the space dimension $d$ and remains valid for shock discontinuities, since it follows from the space--time divergence theorem alone. The multi-scale distribution of the sampled space--time control volumes is illustrated in Fig.~\ref{fig:CV_1}.

\begin{figure}[H]
    \centering
    \includegraphics[width=0.8\textwidth]{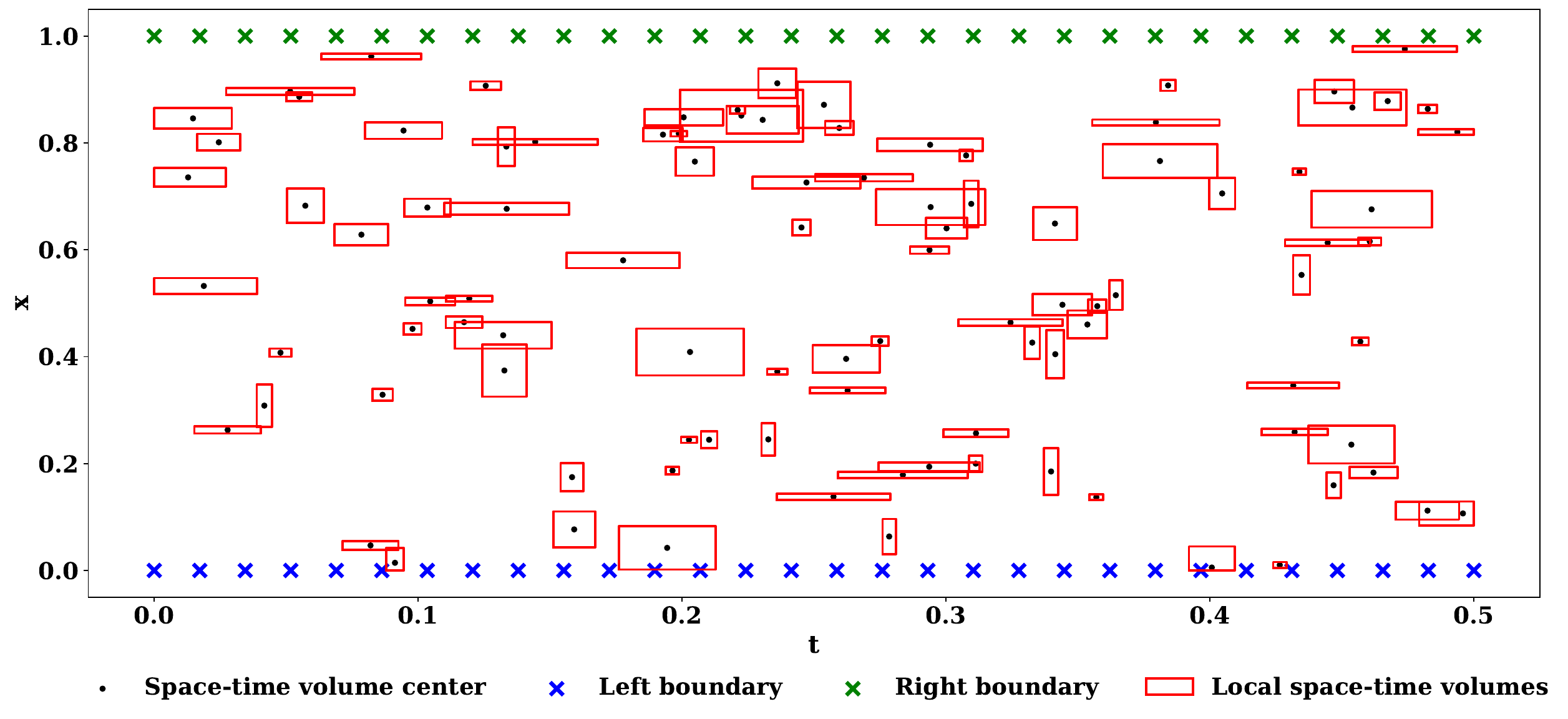}
    \includegraphics[width=0.8\textwidth]{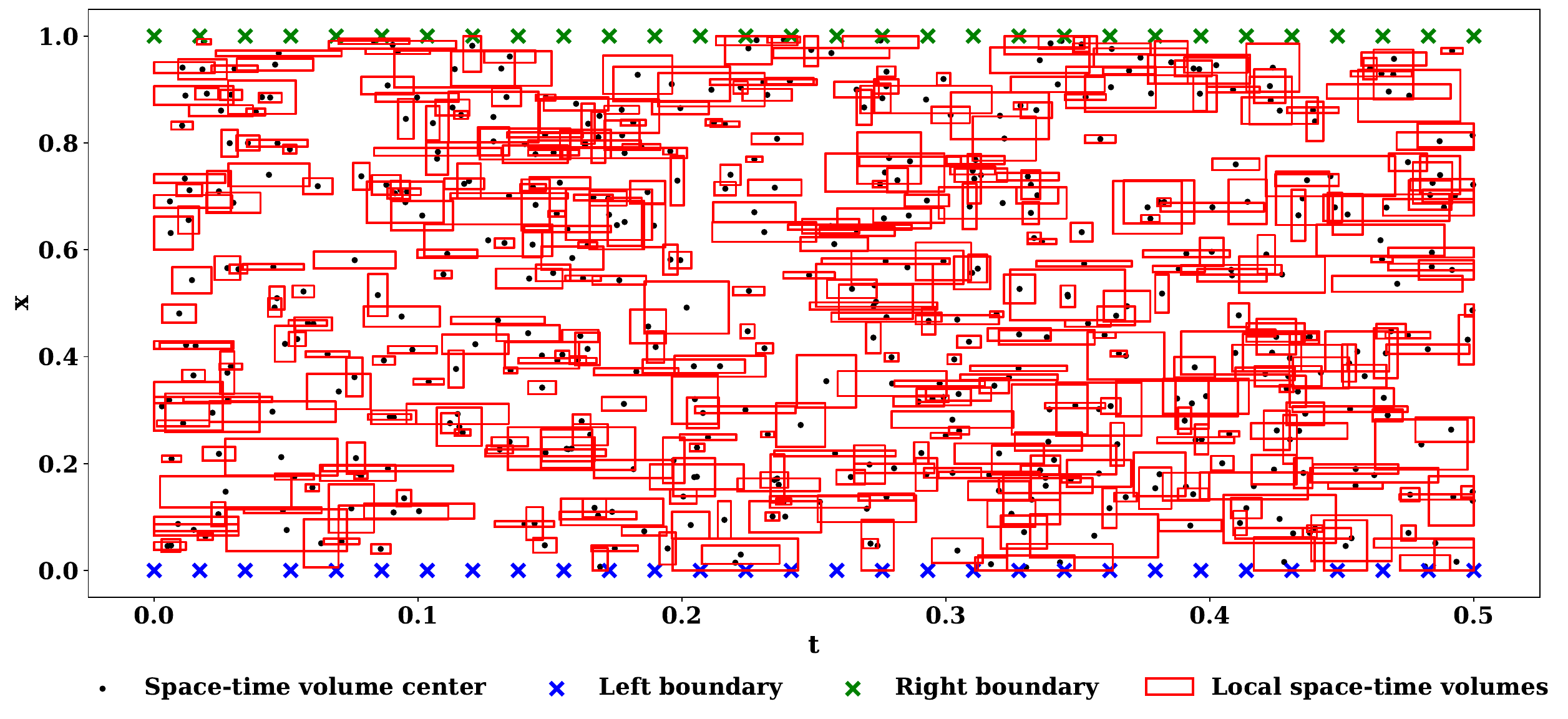}
    \includegraphics[width=0.8\textwidth]{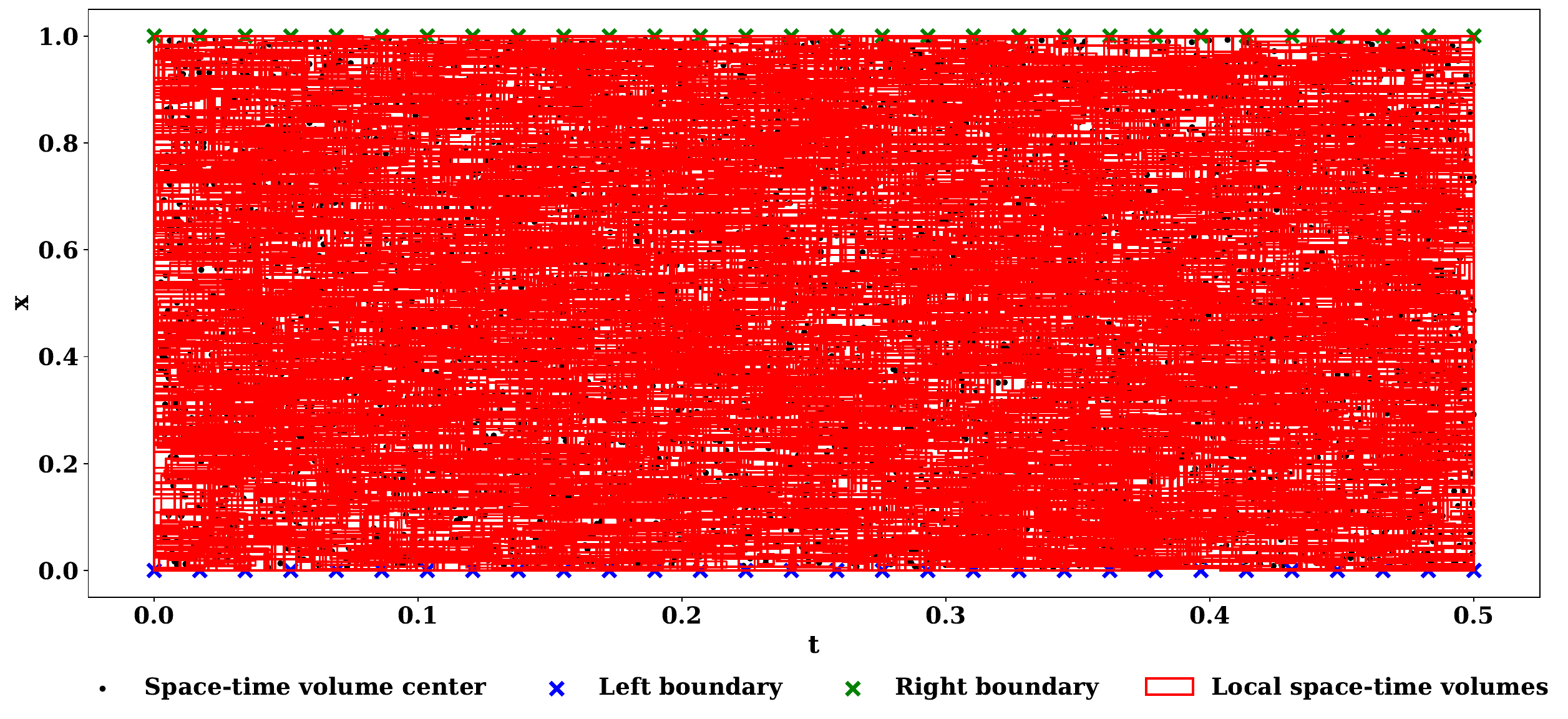}
    \caption{Local space--time control volumes for the weak conservation identity~\eqref{eq:weak_conservation}. Each sampled control volume $D_i$ is a space--time rectangle over which the identity is enforced. The three panels correspond to 100 (top), 500 (middle), and 5000 (bottom) sampled control volumes.}
    \label{fig:CV_1}
\end{figure}

\medskip
\noindent\textbf{Axis-aligned control volumes.} We use axis-aligned rectangles in 1D and boxes in higher dimensions for simplicity, although other geometries such as polygons or splines are admissible:
\[
D = [x_1, x_2] \times [t_1, t_2] \quad (d = 1),
\]
and similarly in $d = 2$ and $d = 3$ via Cartesian products of spatial intervals. The outward normal $n = (n_x, n_t)$ is then constant on each face of $\partial D$, which simplifies the evaluation of the boundary integrals in the weak and entropy residuals.

\medskip
\noindent\textbf{Multi-scale sampling.}
Shocks are localized structures. Enforcing conservation at a single spatial or temporal scale is not enough. Control volumes are therefore sampled with variable sizes
\[
\ell_x = x_2 - x_1 \in [\ell_{\min}, \ell_{\max}],
\qquad
\ell_t = t_2 - t_1 \in [\tau_{\min}, \tau_{\max}].
\]
Large domains capture the global balance of conserved quantities; smaller ones resolve sharp gradients. Mixing scales prevents the training process from missing localized shock structures.

\medskip
\noindent\textbf{Partition status of the sampled control volumes.}
The convergence analysis of Section~\ref{sec:convergence} uses a deterministic partition $\{D_i\}_{i=1}^{N_D}$ of $Q := \Omega_x \times (0,T)$ with $\sum_i |D_i| = |Q|$. In the practical implementation, control volumes are sampled stochastically at multiple scales at each training iteration, and the resulting collection may overlap or leave parts of $Q$ uncovered. The multi-scale sampling is therefore a Monte-Carlo surrogate of the partition-based residuals: in expectation over the sampling distribution, the partition assumption is recovered on average, and the empirical residuals approximate their continuous counterparts. A formal probabilistic concentration analysis of this surrogate is left for future work.

\medskip
\noindent\textbf{One spatial dimension.}
Each boundary face reduces to a one-dimensional integral, which we evaluate by Gauss--Legendre quadrature:
\[
\int_a^b g(s)\, ds
\approx \frac{b - a}{2}
\sum_{q=1}^Q w_q\,
g\!\left( \frac{a + b}{2} + \frac{b - a}{2}\xi_q \right),
\]
with $\{\xi_q, w_q\}_{q=1}^Q$ the standard Gauss nodes and weights on $[-1, 1]$. We prefer deterministic quadrature to Monte Carlo estimates because it produces a smoother loss, which is easier to optimize during the early stages of training.

\medskip
\noindent\textbf{Two and three spatial dimensions.}
In higher dimensions, each boundary face is rectangular and we use tensor-product Gauss--Legendre rules. In two dimensions,
\[
\int_{x_1}^{x_2} \int_{y_1}^{y_2} g(x, y)\, dx\, dy
\approx \frac{(x_2 - x_1)(y_2 - y_1)}{4}
\sum_{p=1}^{Q_x} \sum_{q=1}^{Q_y}
w_p v_q\, g\big( x(\xi_p),\, y(\eta_q) \big),
\]
where the mapping from $[-1, 1]^2$ to the physical face is performed independently in each coordinate direction. The same tensor-product strategy extends to three dimensions. Combining multi-scale sampling with deterministic quadrature gives reliable residual evaluations across the space--time domain, including near shocks.

\subsection{Weak conservation loss}
\label{sec:cons_loss}

Let $U_\theta(x, t)$ denote the neural network approximation of the conserved variables, with parameters $\theta$. We satisfy the initial condition through a temporal ansatz:
\[
U_\theta(x, t) = (1 - \tau(t))\, U_0(x) + \tau(t)\,
\mathcal{N}_\theta(x, t),
\]
where $\tau(t)$ is differentiable with $\tau(0) = 0$ and $\tau(T) = 1$ (we use $\tau(t) = t / T$) and $\mathcal{N}_\theta$ is a standard neural network. The ansatz gives $U_\theta(\cdot, 0) = U_0(\cdot)$ almost everywhere, which removes the initial condition error term from the convergence analysis. For Riemann-type initial conditions with discontinuous $U_0$, the identity holds a.e., which is enough for the $L^1$ estimate of Section~\ref{sec:convergence}.

The weak space--time control-volume formulation is enforced numerically by penalizing violations of the conservation identity over the collection $\{D_i\}_{i=1}^{N_D}$. For each $D_i$, the weak conservation residual is
\[
R_\theta(D_i)
= \int_{\partial D_i}
\left( U_\theta\, n_t + F(U_\theta) \cdot n_x \right) dS.
\]
The residual measures the flux imbalance across $\partial D_i$, that is, the deviation of $U_\theta$ from exact conservation over $D_i$. Conservation is componentwise, so the residual is vector-valued, $R_\theta(D_i) \in \mathbb{R}^m$. Boundary conditions enter directly: when a face of $D_i$ coincides with the physical boundary of $\Omega_x$, the flux integration along that face imposes the prescribed boundary data through $R_\theta(D_i)$.

The weak conservation loss is the volume-weighted average of the squared norms of these residuals:
\[
\mathcal{L}_{\mathrm{cons}}(\theta)
= \sum_{i=1}^{N_D}
\frac{\| R_\theta(D_i) \|_2^2}{|D_i|}.
\]
The formulation enforces conservation in integral form, which remains meaningful for discontinuous solutions. Because the loss depends only on boundary flux integrals, it avoids the explicit spatial derivatives of $U_\theta$ that destabilize training in the presence of steep gradients or shocks. Enforcing the constraint over many control volumes distributed throughout the domain preserves local balance while controlling the global consistency.

\subsection{Entropy inequality loss}
\label{sec:ent_loss}

To enforce physical admissibility and select the entropy-consistent solution, an additional term based on
\eqref{eq:entropy_inequality} enters the loss.

\medskip
\noindent\textbf{Scalar case: Kru\v{z}kov entropy family.} For scalar conservation laws, we use the Kru\v{z}kov family $\eta_\kappa(u) = |u - \kappa|$, $q_\kappa(u) = \mathrm{sgn}(u - \kappa)(F(u) - F(\kappa))$ for $\kappa \in K = [-M_\infty, M_\infty]$, where $M_\infty$ is a uniform $L^\infty$ bound on the approximation (see Section~\ref{sec:convergence}). For each control volume $D_i$ and each entropy parameter $\kappa \in K$, the entropy residual is
\[
E_\theta(D_i, \kappa)
= \int_{\partial D_i}
\left( \eta_\kappa(U_\theta)\, n_t + q_\kappa(U_\theta) \cdot n_x \right) dS.
\]
For entropy solutions, $E_\theta(D_i, \kappa) \leq 0$ on every admissible control volume and for every $\kappa$. Only positive violations are penalized, through a one-sided volume-weighted quadratic loss with a supremum over $\kappa$:
\[
\mathcal{L}_{\mathrm{ent}}(\theta)
= \sup_{\kappa \in K}
\sum_{i=1}^{N_D}
\frac{\big( \max\{0,\, E_\theta(D_i, \kappa)\} \big)^2}{|D_i|}.
\]
The supremum over $\kappa$ is approximated by stochastic sampling: at each training iteration, a single $\kappa \in K$ is drawn uniformly at random and used to evaluate the entropy residual on every sampled control volume. The resulting stochastic Kru\v{z}kov sampling strategy gives a Monte--Carlo approximation of the full entropy family; at finite training budget it should be read as a stochastic estimator, not as deterministic enforcement for every $\kappa$. This is needed for non-convex fluxes such as the Buckley--Leverett model, where one entropy pair is not enough to select the Oleinik solution (see Section~\ref{sec:BL}).

\medskip
\noindent\textbf{System case: single entropy pair.} For the Euler system, one mathematical entropy pair $(\eta, q)$ associated with the model is used instead of the full Kru\v{z}kov family. The entropy residual and loss become
\[
E_\theta(D_i)
= \int_{\partial D_i}
\left( \eta(U_\theta)\, n_t + q(U_\theta) \cdot n_x \right) dS,
\qquad
\mathcal{L}_{\mathrm{ent}}(\theta)
= \sum_{i=1}^{N_D}
\frac{\big( \max\{0,\, E_\theta(D_i)\} \big)^2}{|D_i|},
\]
with no supremum and no stochastic sampling. The explicit entropy pairs for Euler appear in~\ref{app:entropies}. The
stochastic Kru\v{z}kov sampling strategy is therefore reserved for the scalar case, in line with the scope of the convergence analysis in Section~\ref{sec:convergence}.

\medskip
This construction matches the underlying theory: the inequality remains satisfied when it should, and positive deviations corresponding to entropy violations are penalized. Equality is not enforced, since entropy production across shocks is physically admissible. The entropy loss thus adds a one-sided constraint to the weak conservation loss, selecting the relevant weak solution without suppressing shock dissipation.

\subsection{Total variation control}
\label{sec:tvd}

The convergence analysis of Section~\ref{sec:convergence} (assumption (H3) of Theorem~\ref{thm:convergence}) requires a uniform bound on the total variation,
\[
\sup_{t \in [0, T]} TV(U_\theta(\cdot, t)) \leq M,
\]
for some constant $M > 0$ independent of $\theta$. The Bouchut--Perthame estimate (Theorem~\ref{thm:BP}) depends on this hypothesis. To promote it during training, we add a total-variation control term
\[
\mathcal{L}_{\mathrm{TVD}}(\theta)
= \sup_{t \in [0, T]} \int_{\Omega} |\nabla_x U_\theta(x, t)|\, dx,
\]
used as a soft regularization. The term does not enforce $TV(U_\theta) \leq M$ deterministically, but in all numerical experiments of Section~\ref{sec:results} we observe that $TV(U_\theta(\cdot, t))$ stays uniformly bounded during training, giving an empirical counterpart to (H3). Enforcing (H3) architecturally, for instance, through projection or clipping, is a natural direction for future work.

In practice, the supremum and the spatial integrals are estimated at each training iteration by sampling a random cloud of points across the space--time domain:
\[
\mathcal{L}_{\mathrm{TVD}}(\theta) \approx
\max_{j = 1, \dots, N_t}
\sum_{k=1}^{N_x - 1} | U_\theta(x_{k+1}, t_j) - U_\theta(x_k, t_j) |,
\]
with $(x_k, t_j)$ drawn randomly at each step. In the two-dimensional experiments, the same regularizer is applied directionwise by summing finite differences along sampled horizontal and vertical point clouds, which promotes bounded variation in each coordinate direction without introducing a fixed finite-volume mesh. The TVD term is the only component of the loss involving finite-difference evaluations of the network values; the conservation and entropy losses depend only on boundary flux integrals of $U_\theta$ and $F(U_\theta)$.

\subsection{Total loss function}
\label{sec:total_loss}

The full WE-PINNs objective is
\[
L(\theta)
= \lambda_1\, \mathcal{L}_{\mathrm{cons}}(\theta)
+ \lambda_2\, \mathcal{L}_{\mathrm{ent}}(\theta)
+ \lambda_{\mathrm{TVD}}\, \mathcal{L}_{\mathrm{TVD}}(\theta),
\]
with $\lambda_1 = \lambda_2 = \lambda_{\mathrm{TVD}} = 1$ in all experiments. Minimizing $L(\theta)$ enforces the space--time flux balance, the entropy inequality, and bounded variation over the sampled regions.

\subsection{Training procedure}
\label{sec:algorithm}

Algorithm~\ref{alg:weak_entropy_pinn} summarizes the training procedure. In the numerical experiments, parameters are first optimized with Adam and then refined with an L-BFGS stage under the matched comparison protocol of~\ref{app:baselines}. The loss weights $\lambda_1$, $\lambda_2$, and $\lambda_{\mathrm{TVD}}$ are all set to $1$.

\begin{algorithm}[htbp]
\caption{Training algorithm for WE-PINNs}
\label{alg:weak_entropy_pinn}
\begin{algorithmic}[1]
\State \textbf{Input:} initial condition $U_0$, flux $F$, entropy
pair(s), domain $\Omega_{x,t}$, loss weights
$\lambda_1 = \lambda_2 = \lambda_{\mathrm{TVD}} = 1$
\State Initialize neural network parameters $\theta$
\State Build temporal ansatz
$U_\theta(x,t) = (1-\tau(t))\, U_0(x) + \tau(t)\, \mathcal{N}_\theta(x,t)$
\While{Adam stage not converged}
\State Sample space--time subdomains $\{D_i\}_{i=1}^{N_D}$ at
multiple scales
\If{scalar case}
\State Sample an entropy index $\kappa \in K$ uniformly at random
(stochastic Kru\v{z}kov sampling)
\EndIf
\For{$i = 1$ to $N_D$}
\State Evaluate $R_\theta(D_i)$ and the entropy residual by
Gauss--Legendre quadrature on each face of $\partial D_i$
\EndFor
\State Compute $\mathcal{L}_{\mathrm{cons}}$ and
$\mathcal{L}_{\mathrm{ent}}$ as in
Sections~\ref{sec:cons_loss}--\ref{sec:ent_loss}
\State Sample a random cloud of points and compute
$\mathcal{L}_{\mathrm{TVD}}$ as in Section~\ref{sec:tvd}
\State Assemble
$L(\theta) = \lambda_1\mathcal{L}_{\mathrm{cons}} + \lambda_2\mathcal{L}_{\mathrm{ent}}
+ \lambda_{\mathrm{TVD}}\mathcal{L}_{\mathrm{TVD}}$
\State Update $\theta$ with Adam:
$\theta \leftarrow \mathrm{Adam}(\theta, \nabla_\theta L(\theta))$
\EndWhile
\State Refine $\theta$ with L-BFGS using the same loss $L(\theta)$
\State \textbf{Output:} trained network $U_\theta$
\end{algorithmic}
\end{algorithm}

\section{Convergence Analysis}
\label{sec:convergence}

The convergence analysis of WE-PINNs rests on three ingredients: a continuous weak entropy truncation error in dual form; explicit bounds linking this error to the discrete network losses together with sampling-consistency residuals; and the Bouchut--Perthame stability estimate for scalar conservation laws. The result is conditional and applies to scalar conservation laws, which matches the scope of the underlying Bouchut--Perthame theory. Extension to systems of conservation laws would need a corresponding stability framework for the system case, which is not available with the same generality. The numerical experiments on the Euler system in Section~\ref{sec:euler} are reported as robustness studies rather than theoretically guaranteed extensions of the present analysis.

\subsection{Setting and assumptions}
\label{sec:cv_setting}

Let $Q = \Omega \times (0, T)$ denote the space--time domain. Consider the scalar conservation law
\begin{equation}
\partial_t u + \nabla_x \cdot F(u) = 0,
\qquad u(x, 0) = u_0(x),
\label{eq:scalar_cl}
\end{equation}
with $u_0 \in L^\infty(\Omega) \cap BV(\Omega)$ and $F$ Lipschitz continuous on bounded sets. Let $U_\theta$ denote the WE-PINNs approximation. We assume there exists $M_\infty > 0$, independent of $\theta$, such that
\begin{equation}
\| U_\theta \|_{L^\infty(Q)} \leq M_\infty,
\qquad
\| u_0 \|_{L^\infty(\Omega)} \leq M_\infty,
\label{eq:linfty_bound}
\end{equation}
and restrict the entropy parameter to the compact set $K = [-M_\infty, M_\infty]$. This restriction keeps $|F(U_\theta) - F(\kappa)|$ uniformly bounded for $\kappa \in K$.

Two test function classes are used in the analysis:
\begin{equation}
\mathcal{A} = \Bigl\{
\varphi \in C_c^1(Q) :\,
0 \leq \varphi \leq 1,\;
\| \partial_t \varphi \|_{L^\infty}
+ \| \nabla_x \varphi \|_{L^\infty} \leq 1
\Bigr\},
\label{eq:test_class_A}
\end{equation}
\begin{equation}
\mathcal{B} = \Bigl\{
\psi \in C_c^1(Q) :\,
|\psi| \leq 1,\;
\| \partial_t \psi \|_{L^\infty}
+ \| \nabla_x \psi \|_{L^\infty} \leq 1
\Bigr\}.
\label{eq:test_class_B}
\end{equation}
$\mathcal{A}$ collects nonnegative test functions for the entropy inequality; $\mathcal{B}$ collects signed test functions for the conservation equality. The Lipschitz constraint is what keeps the dual residuals finite in the presence of discontinuities.

\subsection{Continuous weak residuals}
\label{sec:cv_continuous}

For $\kappa \in K$, the Kru\v{z}kov entropy and entropy flux are
\begin{equation}
\eta_\kappa(U_\theta) = | U_\theta - \kappa |,
\qquad
q_\kappa(U_\theta) = \mathrm{sgn}(U_\theta - \kappa)\bigl( F(U_\theta) - F(\kappa) \bigr).
\label{eq:kruzkov_pair}
\end{equation}
For $\varphi \in \mathcal{A}$, the weak entropy residual is defined with the sign convention
\begin{equation}
R_{\kappa, \varphi}(U_\theta)
= - \int_Q
\Bigl(
\eta_\kappa(U_\theta)\, \partial_t \varphi
+ q_\kappa(U_\theta) \cdot \nabla_x \varphi
\Bigr)\, dx\, dt.
\label{eq:weak_ent_res}
\end{equation}
This convention matches the distributional entropy inequality $\partial_t \eta_\kappa(u) + \nabla_x \cdot q_\kappa(u) \leq 0$: for the exact entropy solution and any nonnegative $\varphi \in \mathcal{A}$, integration by parts gives $R_{\kappa, \varphi}(u) \leq 0$. The continuous weak entropy truncation error is then
\begin{equation}
E_R(U_\theta)
= \sup_{\kappa \in K,\; \varphi \in \mathcal{A}}
R_{\kappa, \varphi}(U_\theta).
\label{eq:trunc_error}
\end{equation}
The zero test function is in $\mathcal{A}$, so $E_R(U_\theta) \geq 0$; $E_R(U_\theta)$ is the maximal positive violation of the Kru\v{z}kov entropy inequalities.

For $\psi \in \mathcal{B}$, the weak conservation residual is
\begin{equation}
C_\psi(U_\theta)
= - \int_Q
\Bigl(
U_\theta\, \partial_t \psi
+ F(U_\theta) \cdot \nabla_x \psi
\Bigr)\, dx\, dt,
\label{eq:weak_cons_res}
\end{equation}
and the associated continuous defect is
\begin{equation}
E_C(U_\theta) = \sup_{\psi \in \mathcal{B}} | C_\psi(U_\theta) |.
\label{eq:cons_defect}
\end{equation}

\subsection{Discrete losses and sampling residuals}
\label{sec:cv_discrete}

The analysis uses a deterministic space--time partition $\{D_i\}_{i=1}^{N_D}$ of $Q$ with $\sum_i |D_i| = |Q|$ and barycenters $z_i$. The multi-scale stochastic sampling of Section~\ref{sec:cv_formulation} is a Monte--Carlo surrogate of this partition-based setting.

The discrete conservation residual on $D_i$ is
\begin{equation}
R_{\mathrm{cons}}(D_i)
= \int_{\partial D_i}
\bigl(
U_\theta\, n_t + F(U_\theta) \cdot n_x
\bigr)\, dS,
\label{eq:R_cons_disc}
\end{equation}
and the discrete entropy production for $\kappa \in K$ is
\begin{equation}
E_\theta(D_i, \kappa)
= \int_{\partial D_i}
\bigl(
\eta_\kappa(U_\theta)\, n_t + q_\kappa(U_\theta) \cdot n_x
\bigr)\, dS.
\label{eq:E_disc}
\end{equation}
The volume-weighted discrete losses are
\begin{equation}
\mathcal{L}_{\mathrm{cons}}(\theta)
= \sum_{i=1}^{N_D}
\frac{| R_{\mathrm{cons}}(D_i) |^2}{| D_i |},
\label{eq:L_cons_disc}
\end{equation}
\begin{equation}
\mathcal{L}_{\mathrm{ent}}(\theta)
= \sup_{\kappa \in K}
\sum_{i=1}^{N_D}
\frac{\bigl( \max\{0,\, E_\theta(D_i, \kappa)\} \bigr)^2}{| D_i |}.
\label{eq:L_ent_disc}
\end{equation}
The supremum over $\kappa \in K$ in~\eqref{eq:L_ent_disc} matches the use of the full Kru\v{z}kov family; in the implementation, it is estimated by stochastic Kru\v{z}kov sampling as discussed in Section~\ref{sec:ent_loss}.

The gap between the continuous weak residuals \eqref{eq:weak_ent_res} and \eqref{eq:weak_cons_res} and their reconstructions from the sampled control volumes is quantified by two sampling-consistency residuals:
\begin{equation}
R_{\mathrm{samp}}^{\mathrm{ent}}(\{D_i\})
= \sup_{\kappa \in K,\; \varphi \in \mathcal{A}}
\left|
R_{\kappa, \varphi}(U_\theta)
- \sum_{i=1}^{N_D} \varphi(z_i)\, E_\theta(D_i, \kappa)
\right|,
\label{eq:R_samp_ent}
\end{equation}
\begin{equation}
R_{\mathrm{samp}}^{\mathrm{cons}}(\{D_i\})
= \sup_{\psi \in \mathcal{B}}
\left|
C_\psi(U_\theta)
- \sum_{i=1}^{N_D} \psi(z_i)\, R_{\mathrm{cons}}(D_i)
\right|.
\label{eq:R_samp_cons}
\end{equation}
We keep these residuals as explicit terms in the convergence estimate rather than absorbing them in a formal $\mathcal{O}(h_{\max})$ remainder, because such a rate is not automatic and would need additional regularity and mesh-regularity assumptions (see Remark~\ref{rem:rates}).

\subsection{Discrete-to-continuous estimates}
\label{sec:cv_lemmas}

\begin{lemma}[Entropy residual controlled by the entropy loss] \label{lem:entropy_bound} Let $U_\theta \in L^\infty(Q)$ and let $F$ be Lipschitz continuous on bounded sets. With the notation of Sections~\ref{sec:cv_continuous}--\ref{sec:cv_discrete},
\begin{equation}
E_R(U_\theta)
\leq \sqrt{|Q|}\, \sqrt{\mathcal{L}_{\mathrm{ent}}(\theta)}
+ R_{\mathrm{samp}}^{\mathrm{ent}}(\{D_i\}).
\label{eq:lem1_bound}
\end{equation}
\end{lemma}

\begin{proof}
Fix $\kappa \in K$ and $\varphi \in \mathcal{A}$. By definition of $R_{\mathrm{samp}}^{\mathrm{ent}}$,
\[
R_{\kappa, \varphi}(U_\theta)
= \sum_{i=1}^{N_D} \varphi(z_i)\, E_\theta(D_i, \kappa)
+ \varepsilon_{\mathrm{samp}}^{\mathrm{ent}},
\qquad
| \varepsilon_{\mathrm{samp}}^{\mathrm{ent}} |
\leq R_{\mathrm{samp}}^{\mathrm{ent}}(\{D_i\}).
\]
Since $0 \leq \varphi(z_i) \leq 1$,
\[
\sum_{i=1}^{N_D} \varphi(z_i)\, E_\theta(D_i, \kappa)
\leq \sum_{i=1}^{N_D} \varphi(z_i)\, \max\{0,\, E_\theta(D_i, \kappa)\}.
\]
Multiplying and dividing each term by $\sqrt{|D_i|}$ and using the Cauchy--Schwarz inequality,
\[
\sum_{i=1}^{N_D} \varphi(z_i)\, \max\{0,\, E_\theta(D_i, \kappa)\}
\leq
\Biggl(
\sum_{i=1}^{N_D}
\frac{\bigl( \max\{0,\, E_\theta(D_i, \kappa)\} \bigr)^2}{|D_i|}
\Biggr)^{1/2}
\Biggl(
\sum_{i=1}^{N_D} |D_i|
\Biggr)^{1/2}.
\]
The partition property gives $\sum_i |D_i| = |Q|$, and the first factor is bounded by $\sqrt{\mathcal{L}_{\mathrm{ent}}(\theta)}$, so
\[
R_{\kappa, \varphi}(U_\theta)
\leq \sqrt{|Q|}\, \sqrt{\mathcal{L}_{\mathrm{ent}}(\theta)}
+ R_{\mathrm{samp}}^{\mathrm{ent}}(\{D_i\}).
\]
Taking the supremum over $\kappa \in K$ and $\varphi \in \mathcal{A}$ gives~\eqref{eq:lem1_bound}.
\end{proof}

\begin{lemma}[Conservation residual controlled by the conservation loss]
\label{lem:cons_bound}
Under the same assumptions as Lemma~\ref{lem:entropy_bound},
\begin{equation}
E_C(U_\theta)
\leq \sqrt{|Q|}\, \sqrt{\mathcal{L}_{\mathrm{cons}}(\theta)}
+ R_{\mathrm{samp}}^{\mathrm{cons}}(\{D_i\}).
\label{eq:lem2_bound}
\end{equation}
\end{lemma}

\begin{proof}
Fix $\psi \in \mathcal{B}$. By definition of $R_{\mathrm{samp}}^{\mathrm{cons}}$,
\[
C_\psi(U_\theta)
= \sum_{i=1}^{N_D} \psi(z_i)\, R_{\mathrm{cons}}(D_i)
+ \varepsilon_{\mathrm{samp}}^{\mathrm{cons}},
\qquad
| \varepsilon_{\mathrm{samp}}^{\mathrm{cons}} |
\leq R_{\mathrm{samp}}^{\mathrm{cons}}(\{D_i\}).
\]
Since $|\psi(z_i)| \leq 1$ and Cauchy--Schwarz applies as in Lemma~\ref{lem:entropy_bound},
\[
\Biggl|
\sum_{i=1}^{N_D} \psi(z_i)\, R_{\mathrm{cons}}(D_i)
\Biggr|
\leq
\Biggl(
\sum_{i=1}^{N_D}
\frac{| R_{\mathrm{cons}}(D_i) |^2}{| D_i |}
\Biggr)^{1/2}
\sqrt{|Q|}
= \sqrt{|Q|}\, \sqrt{\mathcal{L}_{\mathrm{cons}}(\theta)}.
\]
Taking the supremum over $\psi \in \mathcal{B}$ gives~\eqref{eq:lem2_bound}.
\end{proof}

\subsection{Bouchut--Perthame stability estimate}
\label{sec:cv_BP}

\begin{theorem}[Bouchut--Perthame]
\label{thm:BP}
Let $u$ be the unique entropy solution of~\eqref{eq:scalar_cl} with $u_0 \in L^\infty(\Omega) \cap BV(\Omega)$. Let $v$ be a bounded measurable approximate solution satisfying $\sup_{t \in [0, T]} TV(v(\cdot, t)) < \infty$. Then, for any $T > 0$,
\begin{equation}
\| v(\cdot, T) - u(\cdot, T) \|_{L^1(\Omega)}
\leq \| v(\cdot, 0) - u_0 \|_{L^1(\Omega)}
+ C_{\mathrm{BP}}\,
\Bigl(
TV(u_0)\,
\sup_{t \in [0, T]} TV(v(\cdot, t))\,
E_R(v)
\Bigr)^{1/2},
\label{eq:BP_estimate}
\end{equation}
where $C_{\mathrm{BP}}$ depends only on the Lipschitz constant of $F$ on bounded sets.
\end{theorem}

\subsection{Main convergence theorem}
\label{sec:cv_main}

\begin{theorem}[Conditional convergence of WE-PINNs]
\label{thm:convergence}
Let $u$ be the unique entropy solution of~\eqref{eq:scalar_cl} with $u_0 \in L^\infty(\Omega) \cap BV(\Omega)$, and let $U_\theta$ be the WE-PINNs approximation built from the temporal ansatz $U_\theta(\cdot, 0) = u_0(\cdot)$ a.e. Assume:
\begin{enumerate}[label=(H\arabic*)]
\item $F$ is Lipschitz continuous on bounded sets;
\item there exists $M_\infty > 0$, independent of $\theta$, such that $\| U_\theta \|_{L^\infty(Q)} \leq M_\infty$;
\item there exists $M > 0$, independent of $\theta$, such that $\sup_{t \in [0, T]} TV(U_\theta(\cdot, t)) \leq M$;
\item the discrete losses $\mathcal{L}_{\mathrm{cons}}(\theta)$ and $\mathcal{L}_{\mathrm{ent}}(\theta)$ are defined by~\eqref{eq:L_cons_disc}--\eqref{eq:L_ent_disc};
\item the sampling residuals $R_{\mathrm{samp}}^{\mathrm{ent}}(\{D_i\})$ and $R_{\mathrm{samp}}^{\mathrm{cons}}(\{D_i\})$ are defined by~\eqref{eq:R_samp_ent}--\eqref{eq:R_samp_cons}.
\end{enumerate}
The $L^1$ entropy error then satisfies
\begin{equation}
\| U_\theta(\cdot, T) - u(\cdot, T) \|_{L^1(\Omega)}
\leq K_{\mathrm{opt}}
\Bigl(
\mathcal{L}_{\mathrm{ent}}(\theta)^{1/4}
+ \bigl( R_{\mathrm{samp}}^{\mathrm{ent}}(\{D_i\}) \bigr)^{1/2}
\Bigr),
\label{eq:main_bound_L1}
\end{equation}
with
\begin{equation}
K_{\mathrm{opt}}
= C_{\mathrm{BP}}\, \sqrt{TV(u_0)\, M}\,
\max\bigl\{1,\, |Q|^{1/4}\bigr\}.
\label{eq:K_opt}
\end{equation}
The weak conservation defect satisfies
\begin{equation}
E_C(U_\theta)
\leq \sqrt{|Q|}\, \sqrt{\mathcal{L}_{\mathrm{cons}}(\theta)}
+ R_{\mathrm{samp}}^{\mathrm{cons}}(\{D_i\}).
\label{eq:main_bound_cons}
\end{equation}
For any sequence of approximations $U_{\theta_N}$ and partitions $\{D_i^{(N)}\}$ with
\[
\mathcal{L}_{\mathrm{cons}}(\theta_N) \to 0,
\quad
\mathcal{L}_{\mathrm{ent}}(\theta_N) \to 0,
\quad
R_{\mathrm{samp}}^{\mathrm{cons}}(\{D_i^{(N)}\}) \to 0,
\quad
R_{\mathrm{samp}}^{\mathrm{ent}}(\{D_i^{(N)}\}) \to 0,
\]
one has $U_{\theta_N}(\cdot, T) \to u(\cdot, T)$ in $L^1(\Omega)$ and $E_C(U_{\theta_N}) \to 0$.
\end{theorem}

\begin{proof}
Apply Theorem~\ref{thm:BP} with $v = U_\theta$. The temporal ansatz gives $\| U_\theta(\cdot, 0) - u_0 \|_{L^1(\Omega)} = 0$ since the identity holds a.e., and assumption (H3) supplies the uniform total variation bound needed by Theorem~\ref{thm:BP}. Hence
\[
\| U_\theta(\cdot, T) - u(\cdot, T) \|_{L^1(\Omega)}
\leq C_{\mathrm{BP}}\, \sqrt{TV(u_0)\, M}\, E_R(U_\theta)^{1/2}.
\]
Lemma~\ref{lem:entropy_bound} gives
\[
E_R(U_\theta)
\leq \sqrt{|Q|}\, \sqrt{\mathcal{L}_{\mathrm{ent}}(\theta)}
+ R_{\mathrm{samp}}^{\mathrm{ent}}(\{D_i\}).
\]
The subadditive inequality $(a + b)^{1/2} \leq a^{1/2} + b^{1/2}$ for $a, b \geq 0$ yields
\[
E_R(U_\theta)^{1/2}
\leq |Q|^{1/4}\, \mathcal{L}_{\mathrm{ent}}(\theta)^{1/4}
+ \bigl( R_{\mathrm{samp}}^{\mathrm{ent}}(\{D_i\}) \bigr)^{1/2}.
\]
Substituting and using the definition~\eqref{eq:K_opt} of $K_{\mathrm{opt}}$ gives~\eqref{eq:main_bound_L1}. The conservation estimate~\eqref{eq:main_bound_cons} follows from Lemma~\ref{lem:cons_bound}. The convergence claim is immediate from~\eqref{eq:main_bound_L1}--\eqref{eq:main_bound_cons}.
\end{proof}

\subsection{Discussion and structural remarks}
\label{sec:cv_remarks}

\begin{remark}[On the partition vs.\ multi-scale stochastic sampling]
\label{rem:partition_vs_sampling}
Theorem~\ref{thm:convergence} is stated for a deterministic partition $\{D_i\}$ of $Q$ with $\sum_i |D_i| = |Q|$. In the practical implementation of Section~\ref{sec:cv_formulation}, space--time control volumes are sampled dynamically at multiple scales and may overlap, so the collection at a given training iteration is not a strict partition. The bounds of Lemmas~\ref{lem:entropy_bound}--\ref{lem:cons_bound} apply, at each iteration, to the underlying idealized partition obtained by taking expectations over the sampling distribution; the multi-scale sampling is then a Monte--Carlo surrogate of the deterministic partition-based residuals. A deterministic inequality for the stochastic multi-scale sampler would need additional probabilistic concentration arguments and is left for future work.
\end{remark}

\begin{remark}[On the entropy loss and stochastic Kru\v{z}kov sampling] The definition~\eqref{eq:L_ent_disc} of $\mathcal{L}_{\mathrm{ent}}$ involves a supremum over $\kappa \in K$. In the implementation of Section~\ref{sec:ent_loss}, this supremum is estimated by stochastic sampling of $\kappa \in K$ at each training iteration. The stochastic Kru\v{z}kov sampling strategy gives a Monte--Carlo estimate of the continuous supremum over the full entropy family, which is needed for non-convex fluxes such as the Buckley--Leverett model. At finite sample size it is therefore an estimator of the family constraint, not a pointwise guarantee that all entropy parameters have been enforced.
\end{remark}

\begin{remark}[On the assumption $U_\theta(\cdot, 0) = u_0(\cdot)$ for discontinuous initial condition] The Riemann-type initial conditions used in Section~\ref{sec:results} have $u_0$ bounded with finite total variation but discontinuous. The temporal ansatz $U_\theta(x, t) = (1 - \tau(t))\, u_0(x) + \tau(t)\, \mathcal{N}_\theta(x, t)$ gives $U_\theta(\cdot, 0) = u_0(\cdot)$ almost everywhere, which is enough for the $L^1$ estimate~\eqref{eq:main_bound_L1} and matches the regularity $u_0 \in L^\infty(\Omega) \cap BV(\Omega)$ required by Theorem~\ref{thm:BP}.
\end{remark}

\begin{remark}[On the total-variation hypothesis (H3)]
\label{rem:tvd_soft}
The uniform total-variation bound (H3) is required by Theorem~\ref{thm:BP} and is a theoretical hypothesis. WE-PINNs does not enforce it as a hard constraint: $\mathcal{L}_{\mathrm{TVD}}$ in Section~\ref{sec:tvd} is a soft regularization that promotes bounded total variation during training. Empirically, $TV(U_\theta(\cdot, t))$ stays uniformly bounded throughout the simulations of Section~\ref{sec:results}, which gives an empirical counterpart to (H3). Enforcing (H3) architecturally is left for future work.
\end{remark}

\begin{remark}[On potential mesh-dependent rates for the sampling residuals]
\label{rem:rates}
Under additional regularity assumptions on $U_\theta$, the flux $F$, and the admissible test functions, and provided the partitions are geometrically regular, one may attempt to prove rates of the form
\[
R_{\mathrm{samp}}^{\mathrm{ent}}(\{D_i\}) = \mathcal{O}(h_{\max}),
\qquad
R_{\mathrm{samp}}^{\mathrm{cons}}(\{D_i\}) = \mathcal{O}(h_{\max}),
\]
where $h_{\max}$ is the maximum cell diameter. Via Theorem~\ref{thm:convergence}, these rates would translate into a contribution $\mathcal{O}(h_{\max}^{1/2})$ to the $L^1$ error. The present analysis claims no such rate: the sampling residuals stay as explicit sampling-consistency terms, and empirical convergence rates with respect to the discrete loss are reported in Section~\ref{subsec:burger_convergence}.
\end{remark}

\begin{remark}[Empirical status of the loss--error rates]
\label{rem:structured_rate}
Theorem~\ref{thm:convergence} establishes the exponent $1/4$ in $\mathcal{L}_{\mathrm{ent}}$ under the stated assumptions. The finite-data slopes reported in Section~\ref{subsec:burger_convergence} are empirical diagnostics: on the last decade of Table~\ref{tab:convergence_slope}, the observed slopes are approximately $0.30$ on the Cartesian grid and $0.35$ on the random grid. These values confirm a positive loss--error trend but do not provide a clean numerical separation between a structured-grid and a random-grid asymptotic rate. We therefore do not claim an empirical $\mathcal{O}(\mathcal{L}^{1/2})$ rate from these data.
\end{remark}

\begin{remark}[Scope of the convergence result]
Theorem~\ref{thm:convergence} gives a conditional $L^1$ convergence estimate for scalar conservation laws, matching the scope of the Bouchut--Perthame error theory. Extension to systems would require a corresponding stability framework, which is not available with the same generality. The numerical experiments on the Euler system in Section~\ref{sec:euler} are reported as robustness studies rather than theoretically guaranteed extensions of Theorem~\ref{thm:convergence}.
\end{remark}

\section{Numerical Experiments}\label{sec:results}

This section presents a comprehensive numerical assessment of the WE-PINNs framework. The evaluation pursues three objectives: to validate the convergence estimate established in Section~\ref{sec:convergence}, to demonstrate the ability of the method to capture entropy solutions in the presence of shocks, rarefactions, and contact discontinuities, and to confirm that the weak space--time control-volume formulation extends naturally from scalar equations to systems of conservation laws.

The validation proceeds from prototype scalar problems to increasingly demanding configurations. The inviscid Burgers equation serves as the canonical scalar benchmark, admitting closed-form entropy solutions for a variety of Riemann configurations and isolating the core difficulties of shock formation and entropy admissibility. The Buckley--Leverett equation introduces the additional difficulty of a non-convex flux, requiring enforcement of the full Kru\v{z}kov entropy family to select the Oleinik-admissible solution. The framework is then evaluated on hyperbolic systems: the compressible Euler equations, which admit shock waves, rarefaction fans, and contact discontinuities.

Unless stated otherwise, all experiments share the following configuration. The neural network approximation $U_\theta$ is a fully connected feedforward network with $\tanh$ activations. For WE-PINNs, the initial condition is satisfied exactly through the temporal ansatz introduced in Section~\ref{sec:wepinns}; for the CI-PINNs, cvPINNs, and standard PINNs baseline methods, the initial condition is treated through the corresponding loss formulation. Consequently, whenever a table averages errors over a set of reporting times that includes $t=0$, any nonzero initial-time baseline error is included in the reported average. For WE-PINNs, the exact temporal ansatz gives zero error at $t=0$, so averaged WE-PINNs values that include $t=0$ must be consistent with the corresponding later-time errors. Space--time control volumes are sampled dynamically at multiple scales according to the multi-scale strategy of Section~\ref{sec:wepinns}, and boundary flux integrals are evaluated by Gauss--Legendre quadrature on each face of the control volume. Specific choices of network depth, width, quadrature order, and spatial domain are reported individually for each test case.

Accuracy is quantified at the reported times. For a time $t_j$, we use
\begin{equation}
  E_{L^p}(t_j) =
  \frac{\|U_\theta(\cdot,t_j) - U(\cdot,t_j)\|_{L^p(\Omega)}}{\|U(\cdot,t_j)\|_{L^p(\Omega)}},
  \qquad p = 1, 2, \infty,
  \label{eq:rel_error}
\end{equation}
where $U$ denotes the reference entropy solution. When a table reports a single value over several times, it is the arithmetic mean of $E_{L^p}(t_j)$ over the listed reporting times; final-time tables correspond to a single reporting time. In all test cases, the WE-PINNs solution is compared against both a high-resolution reference solution and the CI-PINNs, cvPINNs, and standard PINNs baseline methods, all of which are re-implemented in the same training harness with matched network architecture, optimizer, training budget, spatial domain, and initial/boundary data. Method-specific loss terms, weights, quadrature choices, and collocation settings are documented in~\ref{app:baselines}. The entropy pair associated with each model is provided in~\ref{app:entropies}.

\subsection{Inviscid Burgers Equation}
\label{sec:burgers}

The WE-PINNs framework is first evaluated on the one-dimensional inviscid Burgers equation
\begin{equation}
  \partial_t u \;+\; \partial_x\!\left(\frac{u^2}{2}\right) = 0,
  \qquad (x,t)\in\Omega\times(0,T),
  \label{eq:burgers}
\end{equation}
on the spatial domain $\Omega=[-1,1]$ with final time $T=1$. The Burgers equation is the canonical scalar nonlinear conservation law: even smooth initial conditions develop shock discontinuities in finite time through nonlinear steepening, and the entropy condition is required to select the unique physically admissible weak solution among all distributional solutions. This combination of features makes Eq.~\eqref{eq:burgers} the standard first benchmark for any shock-capturing method. Unless otherwise stated, the three main Burgers benchmarks below use the default configuration chosen consistently with the parametric study reported in Section~\ref{subsec:burger_convergence}: the neural network $u_\theta$ comprises $8$ hidden layers of $32$ neurons each with $\tanh$ activations, boundary flux integrals are evaluated using $8$-point Gauss--Legendre quadrature on each control-volume face, and $2500$ space--time control volumes are sampled at multiple scales according to the strategy of Section~\ref{sec:wepinns} (see Fig.~\ref{fig:CV_1}). Although the sweep in Section~\ref{subsec:burger_convergence} uses odd quadrature orders $Q\in\{3,5,7,9,11\}$, the default $Q=8$ lies in the stable $Q\geq7$ range identified by that study. The error study itself deliberately varies the control-volume distribution, quadrature order, depth, and width, and is therefore not part of the fixed-configuration comparison.

\subsubsection{Riemann Shock Problem}
\label{subsec:burger_shock}

The first test case is the Riemann initial condition
\begin{equation}
  u(x,0) =
  \begin{cases}
    1, & x \le -0.25, \\
    0, & x > -0.25,
  \end{cases}
\end{equation}
whose entropy solution is a single right-moving shock propagating at the Rankine--Hugoniot speed $s = \tfrac{1}{2}$. This test directly assesses whether the integral conservation identity~\eqref{eq:weak_conservation}, which implicitly encodes the Rankine--Hugoniot jump conditions, enforces the correct shock speed and maintains a sharp, well-localized discontinuity throughout the simulation.

The solution profiles at $t=0$, $t=0.5$, and $t=1.0$ are shown in Fig.~\ref{fig:Riemann_Problem}. The shock predicted by WE-PINNs is captured at the analytical location at every reported time, with no spurious oscillations or smearing on either side of the discontinuity, and the constant plateau states are fully preserved. The standard PINNs, by contrast, diffuse the shock over a wide region and underestimate the wave speed, a behavior consistent with the divergence of the strong-form residual near a discontinuity reported in~\cite{de2024wpinns}.

\begin{figure}[H]
  \centering
  \includegraphics[width=\textwidth]{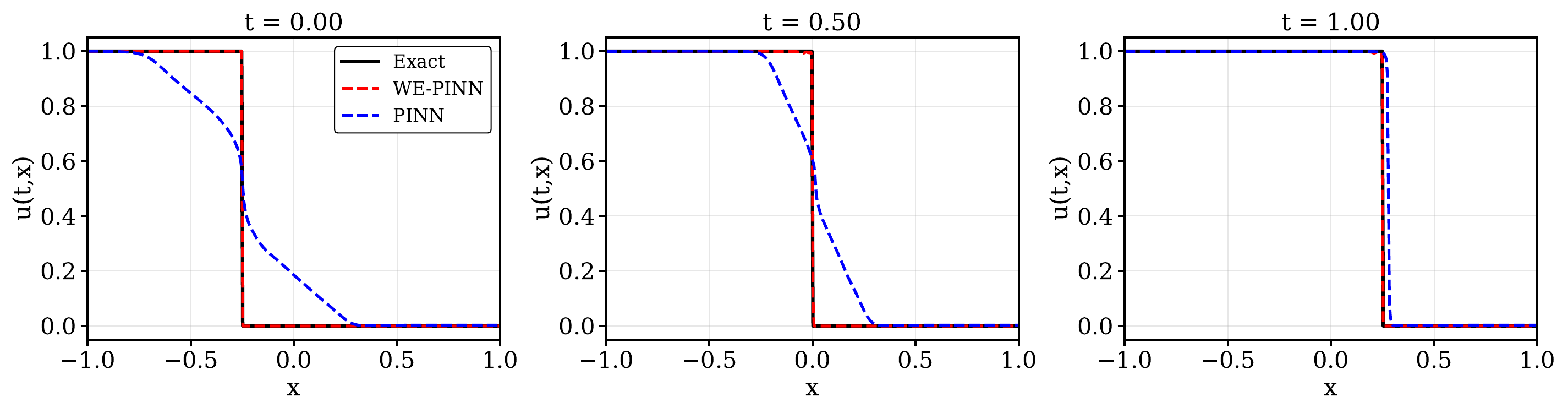}
  \caption{Solution profiles for the Burgers Riemann shock problem at $t=0$, $t=0.5$, and $t=1.0$. Comparison between the exact entropy solution, WE-PINNs, and the standard PINNs.}
  \label{fig:Riemann_Problem}
\end{figure}

The training history is reported in Fig.~\ref{fig:burger_shock_loss}. The total WE-PINNs loss decreases, and all individual components $\mathcal{L}_{\mathrm{cons}}$, $\mathcal{L}_{\mathrm{ent}}$, and $\mathcal{L}_{\mathrm{TVD}}$ settle to low values. The strongest decrease is observed in $\mathcal{L}_{\mathrm{TVD}}$, indicating that the weak conservation identity, the entropy inequality, and the TVD regularization are enforced together during training. The PDE residual of the standard PINNs, by contrast, stagnates at a significantly higher plateau. The space--time evolution in Fig.~\ref{fig:Riemann_Problem_space_time} shows that the WE-PINNs solution faithfully reproduces the diagonal shock trajectory, with the relative $L^2$ error tightly concentrated along the shock line, whereas the standard PINNs error spreads across a broad region of the space--time domain.

\begin{figure}[H]
  \centering
  \includegraphics[width=\textwidth]{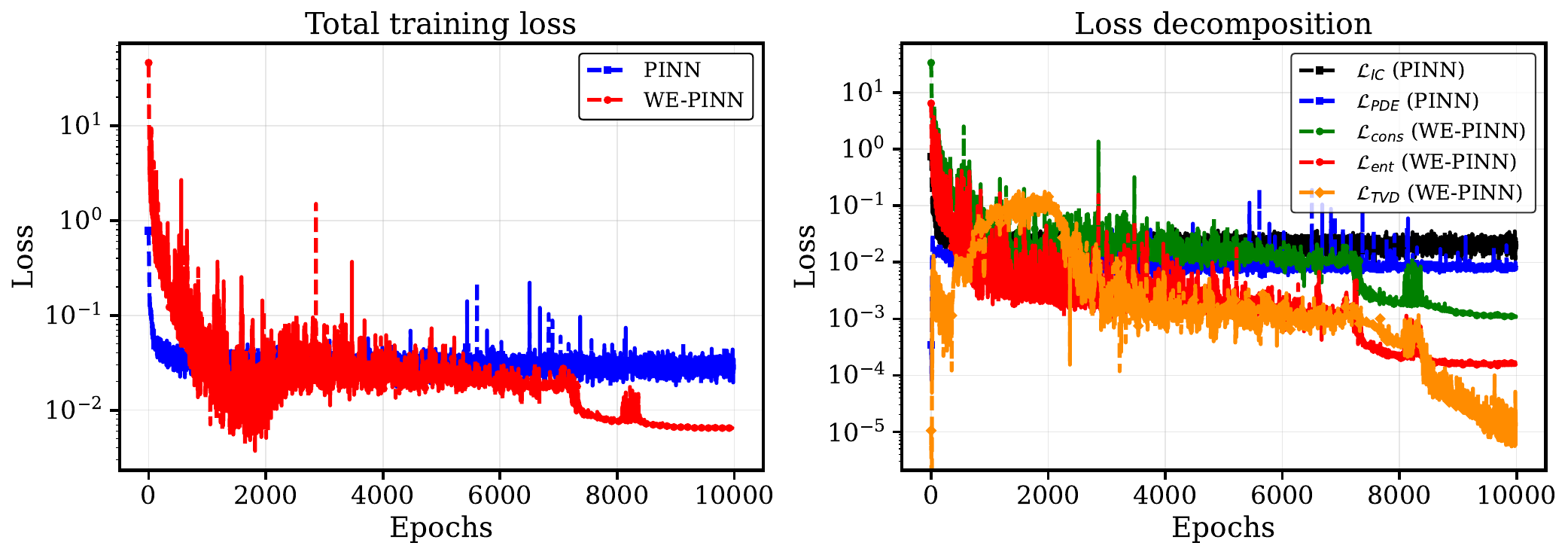}
  \caption{Training loss history for the Burgers Riemann shock problem. \emph{Left}: total loss for both methods. \emph{Right}: individual components $\mathcal{L}_{\mathrm{cons}}$, $\mathcal{L}_{\mathrm{ent}}$ and $\mathcal{L}_{\mathrm{TVD}}$ (WE-PINNs); $\mathcal{L}_{\mathrm{PDE}}$ and $\mathcal{L}_{\mathrm{IC}}$ (standard PINNs).}
  \label{fig:burger_shock_loss}
\end{figure}

\begin{figure}[htbp]
  \centering
  \includegraphics[width=\textwidth]{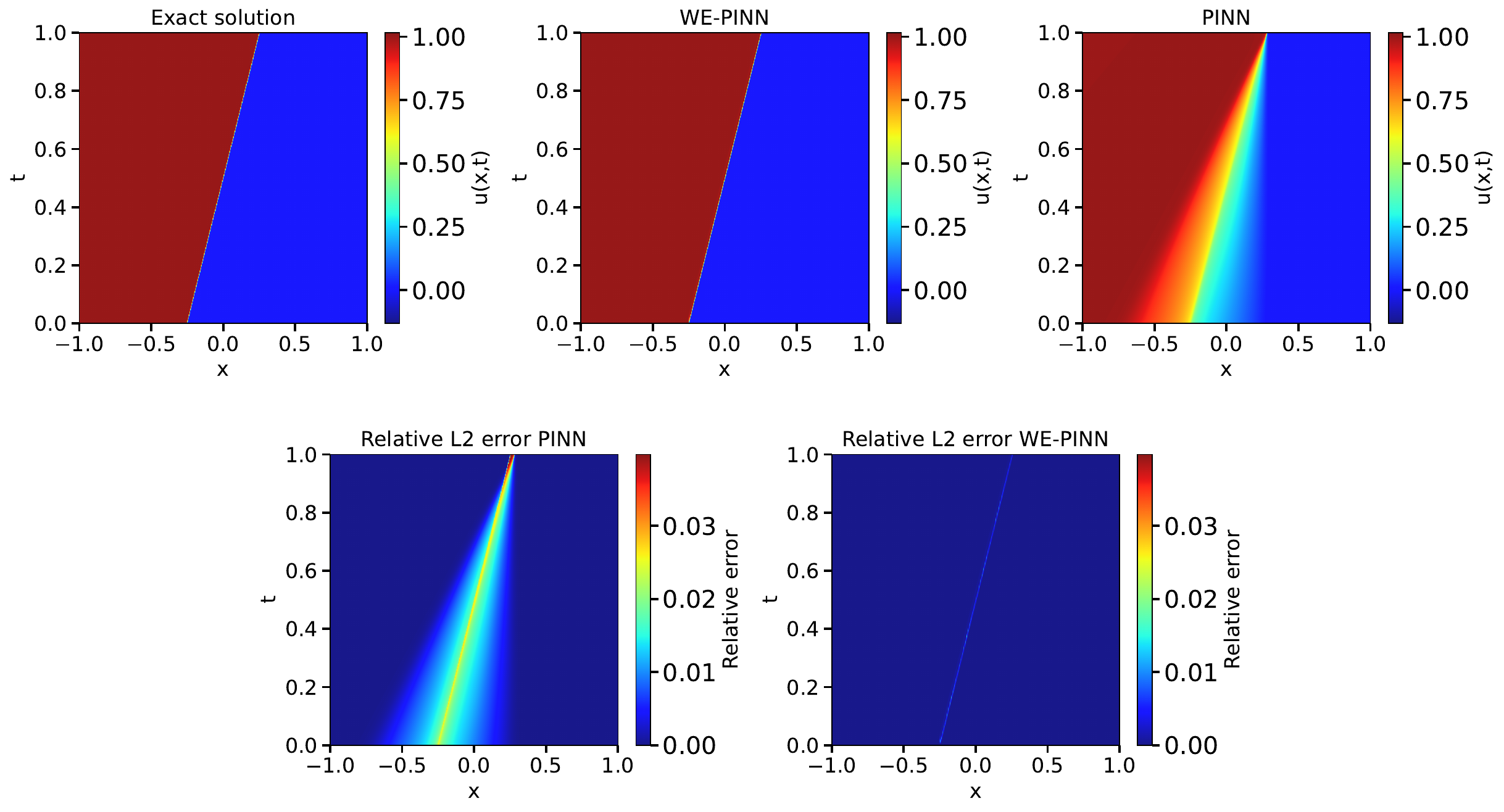}
  \caption{Space--time evolution of the solution for the Burgers Riemann shock problem. \emph{Top row}: exact solution, WE-PINNs solution, and standard PINNs solution. \emph{Bottom row}: pointwise relative $L^2$ error maps for WE-PINNs and the standard PINNs.}
  \label{fig:Riemann_Problem_space_time}
\end{figure}
\begin{table}[h!]
\centering
\caption{Relative $L^1$, $L^2$, and $L^\infty$ errors (in units of $10^{-3}$) averaged over $t=0$, $0.5$, and $1.0$ for the Burgers Riemann shock problem, shown for WE-PINNs, CI-PINNs, cvPINNs, and the standard PINNs. \textbf{Bold} indicates the best result and \underline{underline} the second best in each error group.}
\label{tab:burger_riemann_errors}
\begin{tabular}{lccc}
\toprule
Model &  $E_{L^1}$ &  $E_{L^2}$ & $E_{L^\infty}$   \\
\midrule
WE-PINNs       & \textbf{2.42}    & \textbf{14.09}    & \textbf{197.20} \\
CI-PINNs        & 66.05            & 177.50            & 711.80 \\
cvPINNs         & \underline{7.37} & \underline{36.08} & 523.00 \\
Standard PINNs  & 116.80           & 146.90            & \underline{432.40} \\
\bottomrule
\end{tabular}
\end{table}
The relative errors averaged over $t=0$, $0.5$, and $1.0$ are listed in Table~\ref{tab:burger_riemann_errors} for WE-PINNs and three representative PINNs-based methods: the Coupled Integral PINNs (CI-PINNs)~\cite{CIPINN2024}, the control-volume PINNs (cvPINNs)~\cite{Patel2022}, and the standard PINNs. The table shows a clear advantage for WE-PINNs across all three norms. In $L^1$, the error is $2.42\times10^{-3}$, about three times below cvPINNs, one order of magnitude below CI-PINNs, and nearly $50$ times below the standard PINNs. The same ordering appears in $L^2$ and $L^\infty$, showing that the improvement is not limited to an averaged measure but also extends to pointwise accuracy near the shock.

\subsubsection{Rarefaction--Shock Interaction}
\label{subsec:burger_interaction}

The second test case uses the compactly supported initial condition
\begin{equation}
  u(x,0) =
  \begin{cases}
    1, & -0.5 < x \le 0, \\
    0, & \text{otherwise},
  \end{cases}
\end{equation}
which generates a left-expanding rarefaction fan from $x=-0.5$ and a right-moving shock from $x=0$. The two wave families propagate in the same direction; the rarefaction overtakes the shock at a finite interaction time, the intermediate plateau $u=1$ is consumed, and the post-interaction solution reduces to a single weakening shock. This topology change must be captured without any problem-specific treatment, making the test a rigorous assessment of the ability of the weak space--time formulation to handle simultaneous smooth and discontinuous wave structures within a single training run.

The solution profiles in Fig.~\ref{fig:Rarefaction_Shock} show that WE-PINNs resolves the rarefaction fan smoothly while maintaining a sharp shock at every reported time. The space--time plot in Fig.~\ref{fig:Interaction_space_time} faithfully reproduces the moment of wave interaction and the subsequent single-shock regime. The standard PINNs, by contrast, diffuse both wave families heavily and fail to capture the interaction phenomenon.

\begin{figure}[H]
  \centering
  \includegraphics[width=\textwidth]{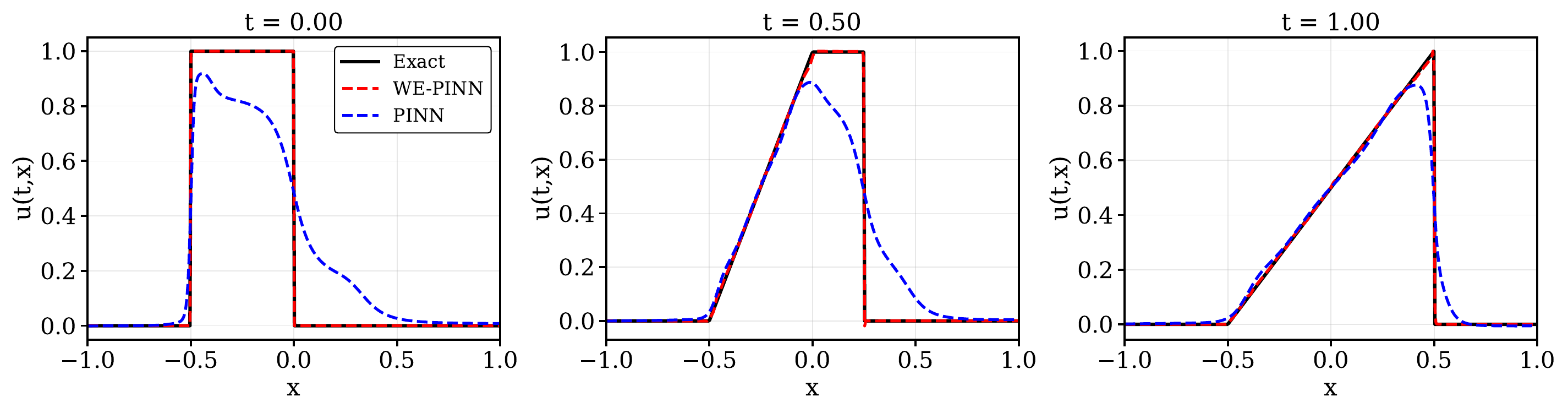}
  \caption{Solution profiles for the Burgers rarefaction--shock interaction problem at $t=0$, $t=0.5$, and $t=1.0$. WE-PINNs resolves the rarefaction fan and the shock simultaneously and captures the wave interaction correctly.}
  \label{fig:Rarefaction_Shock}
\end{figure}

\begin{figure}[H]
  \centering
  \includegraphics[width=\textwidth]{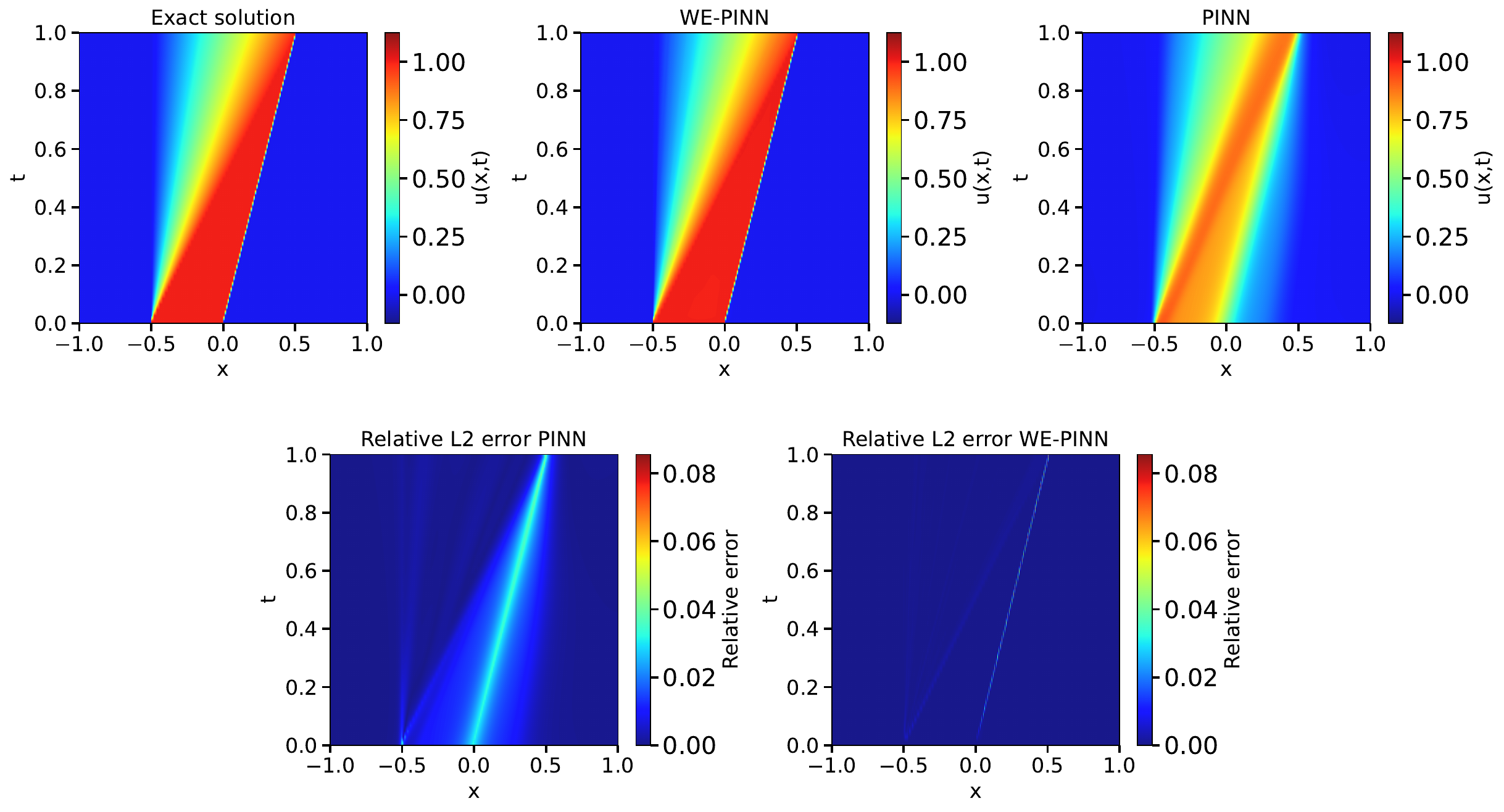}
  \caption{Space--time evolution of the solution for the Burgers rarefaction--shock interaction problem. The convergence of the two wave fronts and the subsequent single-shock regime are both captured accurately by WE-PINNs.}
  \label{fig:Interaction_space_time}
\end{figure}

The relative errors averaged over $t=0$, $0.5$, and $1.0$ are reported in Table~\ref{tab:burger_interaction_errors}. WE-PINNs is again the most accurate method in all three norms. Its $L^1$ error, $8.03\times10^{-3}$, is about five times smaller than cvPINNs and nearly $19$ times smaller than CI-PINNs and the standard PINNs. The gap remains visible in $L^2$ and $L^\infty$, where the competing methods show substantially larger errors near the interacting waves. This supports the visual observation from Figs.~\ref{fig:Rarefaction_Shock}--\ref{fig:Interaction_space_time}: the weak space--time formulation tracks both the rarefaction and the shock without losing the interaction topology.

\begin{table}[htbp]
\centering
\caption{Relative $L^1$, $L^2$, and $L^\infty$ errors (in units of $10^{-3}$) averaged over $t=0$, $0.5$, and $1.0$ for the Burgers rarefaction--shock interaction, shown for WE-PINNs, CI-PINNs, cvPINNs, and the standard PINNs. \textbf{Bold} indicates the best result and \underline{underline} the second best in each error group.}
\label{tab:burger_interaction_errors}
\begin{tabular}{lccc}
\toprule
Model &  $E_{L^1}$ &  $E_{L^2}$ & $E_{L^\infty}$   \\
\midrule
WE-PINNs       & \textbf{8.03}     & \textbf{32.51}     & \textbf{299.60} \\
CI-PINNs        & 149.90            & 310.00             & \underline{700.80} \\
cvPINNs         & \underline{40.97} & \underline{107.40} & 709.30 \\
Standard PINNs  & 150.30            & 213.90             & 716.10 \\
\bottomrule
\end{tabular}
\end{table}

\subsubsection{Sinusoidal Initial Condition}
\label{subsec:burger_sin}

The final Burgers benchmark uses the smooth periodic initial condition $u(x,0) = -\sin(\pi x)$ with periodic boundary conditions on $\Omega=[-1,1]$. This configuration validates a qualitatively different regime: the solution is initially smooth, undergoes nonlinear steepening, and forms a shock at $t = 1/\pi$. The numerical solution must therefore be tracked through both the pre-shock smooth phase and the post-shock regime within a single training run, with the shock-formation time itself learned implicitly through the weak formulation.

The solution profiles at $t=0$, $t=0.5$, and $t=1.0$ are shown in Fig.~\ref{fig:sin_burger}. WE-PINNs faithfully reproduces the steepening of the profile up to shock formation and captures the post-shock state at $t=1.0$ with a sharp discontinuity at $x=0$.
The space--time error maps in Fig.~\ref{fig:sin_burger_space_time} clearly illustrate the contrast between the two methods: the WE-PINNs error remains small and uniformly distributed across the entire space--time domain, whereas the standard PINNs error is strongly concentrated in a band along the shock trajectory.

\begin{figure}[H]
  \centering
  \includegraphics[width=\textwidth]{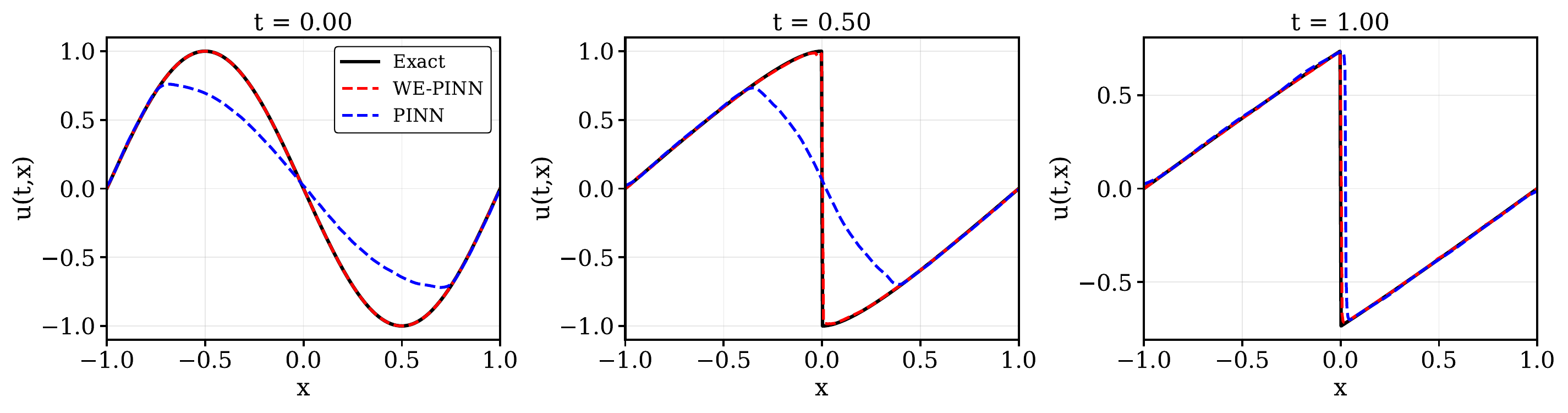}
  \caption{Solution profiles for the Burgers equation with sinusoidal initial condition $u(x,0)=-\sin(\pi x)$ at $t=0$, $t=0.5$, and $t=1.0$. WE-PINNs captures the full evolution from the smooth initial profile through shock formation to the final post-shock state.}
  \label{fig:sin_burger}
\end{figure}

\begin{figure}[htbp]
  \centering
  \includegraphics[width=\textwidth]{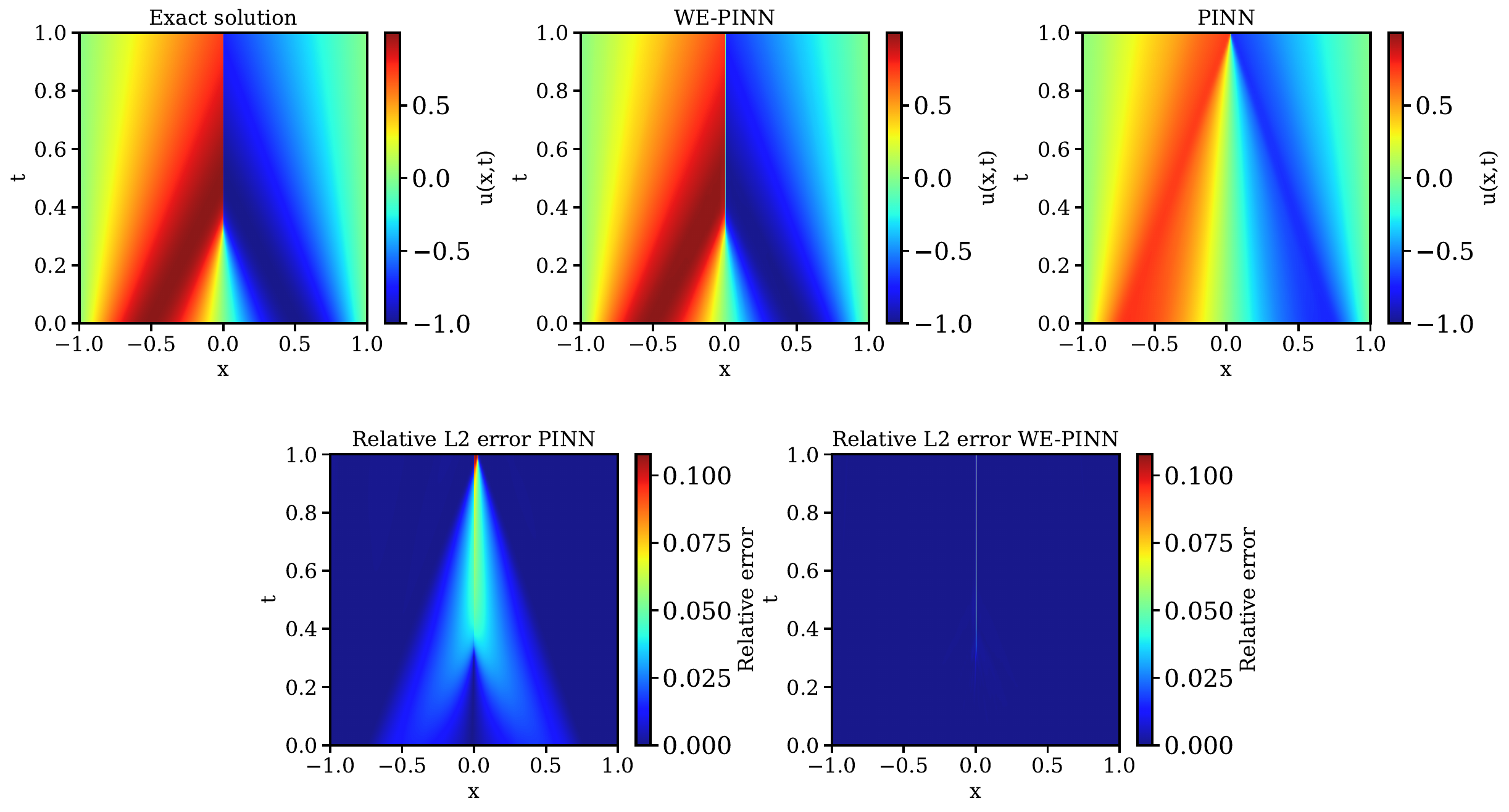}
  \caption{Space--time solution fields (\emph{top row}) and pointwise relative $L^2$ error maps (\emph{bottom row}) for the Burgers equation with the sinusoidal initial condition on $[-1,1]\times[0,1]$. The WE-PINNs error is small and uniformly distributed, whereas the standard PINNs error is concentrated along the shock trajectory.}
  \label{fig:sin_burger_space_time}
\end{figure}
The relative errors averaged over $t=0$, $0.5$, and $1.0$ are listed in Table~\ref{tab:burger_sin_errors}. On this smooth-to-shock benchmark, WE-PINNs gives the smallest averaged errors in all three norms. The $L^1$ error is $3.14\times10^{-3}$, roughly six times below cvPINNs and nearly two orders of magnitude below CI-PINNs and the standard PINNs. The $L^2$ and $L^\infty$ errors follow the same ranking, with the largest separation appearing after shock formation. These results indicate that the weak space--time conservation identity~\eqref{eq:weak_conservation} captures both the onset of the shock and its subsequent propagation.

\begin{table}[htbp]
\centering
\caption{Relative $L^1$, $L^2$, and $L^\infty$ errors (in units of $10^{-3}$) averaged over $t=0$, $0.5$, and $1.0$ for the Burgers Sinusoidal Initial Condition, shown for WE-PINNs, CI-PINNs, cvPINNs, and the standard PINNs. \textbf{Bold} indicates the best result and \underline{underline} the second best in each error group.}
\label{tab:burger_sin_errors}
\begin{tabular}{lccc}
\toprule
Model &  $E_{L^1}$ &  $E_{L^2}$ & $E_{L^\infty}$   \\
\midrule
WE-PINNs       & \textbf{3.14}     & \textbf{51.06}     & \textbf{277.50} \\
CI-PINNs        & 220.20            & 447.40             & \underline{1069.00} \\
cvPINNs         & \underline{20.48} & \underline{120.90} & 1307.00 \\
Standard PINNs  & 183.50            & 327.40             & 1121.00 \\
\bottomrule
\end{tabular}
\end{table}

\subsubsection{Error Study}
\label{subsec:burger_convergence}

To empirically validate the error estimate of Theorem~\ref{thm:convergence} and document the choice of the default Burgers configuration, we present two complementary numerical studies on the Burgers equation with sinusoidal initial condition, which contains a smooth phase and a shock phase. The first examines the relationship between the $L^1$ error and the total loss $\mathcal{L}$ for a fixed network architecture and quadrature rule. The second performs a systematic parametric sweep over the quadrature order $Q$, the network depth $L$, and the number of neurons per layer $N$, in order to assess the robustness of the observed loss error behavior across a range of architectural configurations.

\medskip
\noindent\textit{Study 1: Convergence verification.}
Two control-volume distributions are considered: a regular Cartesian mesh and a random mesh, both comprising the same number of control volumes (Fig.~\ref{fig:mesh_burger_study}). The corresponding $L^1$ errors and total losses are reported in Table~\ref{tab:convergence_slope} and illustrated in Fig.~\ref{fig:error_burger_study}. The data show a clear positive relationship between the loss and the $L^1$ error, but the finite-data slopes do not cleanly separate the two grid types. On the last loss decade, the slopes are approximately $0.30$ for the Cartesian grid and $0.35$ for the random grid. These values are compatible with the conditional loss--error control of Theorem~\ref{thm:convergence}, but they do not justify claiming a clean empirical $\mathcal{O}(\mathcal{L}^{1/2})$ structured-grid rate from Table~\ref{tab:convergence_slope}.

\begin{table}[H]
\centering
\caption{$L^1$ error and total loss $\mathcal{L}$ at selected training stages for the Cartesian and random meshes applied to the Burgers equation, evaluated at $t = T/2$ and $t = T$.}
\label{tab:convergence_slope}
\renewcommand{\arraystretch}{1.}
\begin{tabular}{llcccc}
\toprule
\multirow{2}{*}{\textbf{Grid type}}
& \multirow{2}{*}{\textbf{Total loss}}
& \multirow{2}{*}{\textbf{Epoch}}
& \multicolumn{2}{c}{$E_{L^1}$} \\
\cmidrule(lr){4-5}
& & & $\boldsymbol{T/2}$ & $\boldsymbol{T}$ \\
\midrule
\multirow{4}{*}{\rotatebox{90}{Cartesian}}
& $10^{-1}$ & 2269   & $9.491\times 10^{-2}$ & $5.829\times 10^{-2}$ \\
& $10^{-2}$ & 4326   & $3.708\times 10^{-2}$ & $5.353\times 10^{-2}$ \\
& $10^{-3}$ & 62746  & $1.566\times 10^{-2}$ & $2.070\times 10^{-2}$ \\
& $10^{-4}$ & 123815 & $9.784\times 10^{-3}$ & $1.039\times 10^{-2}$ \\
\midrule
\multirow{4}{*}{\rotatebox{90}{Random}}
& $10^{-1}$ & 5476   & $7.268\times 10^{-2}$ & $1.118\times 10^{-1}$ \\
& $10^{-2}$ & 5609   & $2.894\times 10^{-2}$ & $3.982\times 10^{-2}$ \\
& $10^{-3}$ & 20672  & $1.942\times 10^{-2}$ & $2.530\times 10^{-2}$ \\
& $10^{-4}$ & 65540  & $9.725\times 10^{-3}$ & $1.119\times 10^{-2}$ \\
\bottomrule
\end{tabular}
\end{table}

\begin{figure}[H]
    \centering
    \includegraphics[width=\textwidth]{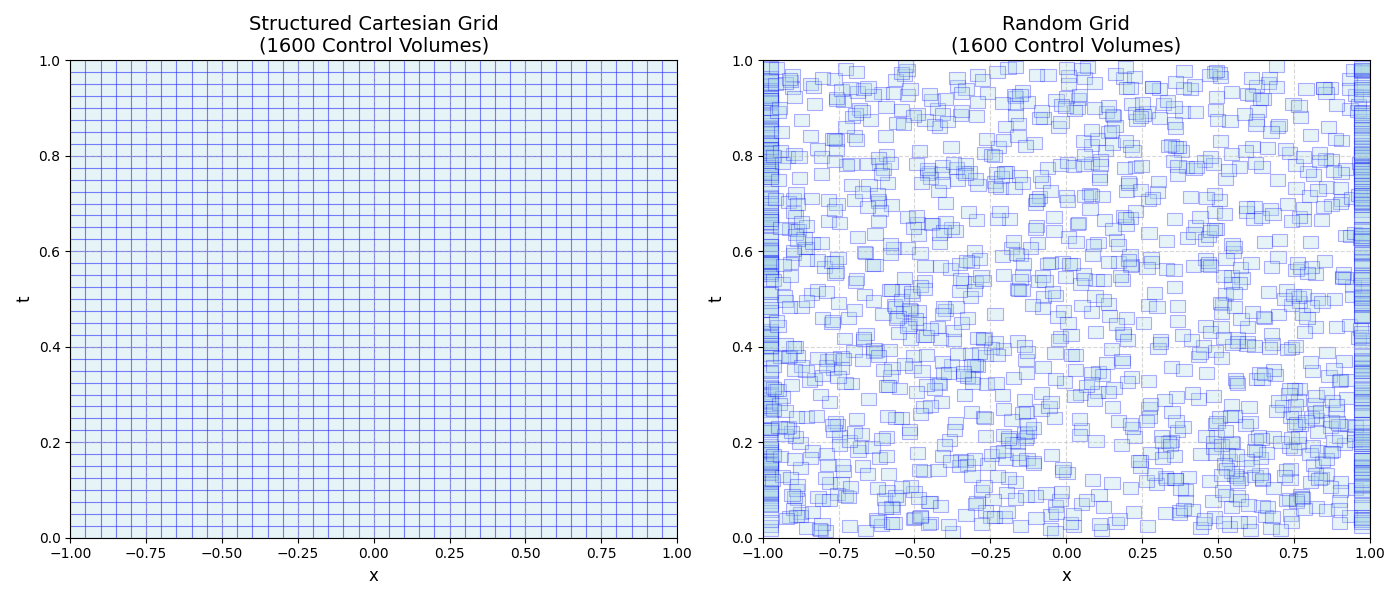}
    \caption{Control-volume distributions for the regular Cartesian mesh (left) and the random mesh (right), both comprising the same number of control volumes.}
    \label{fig:mesh_burger_study}
\end{figure}

\begin{figure}[H]
    \centering
    \begin{subfigure}{0.48\textwidth}
        \centering
        \includegraphics[width=\linewidth]{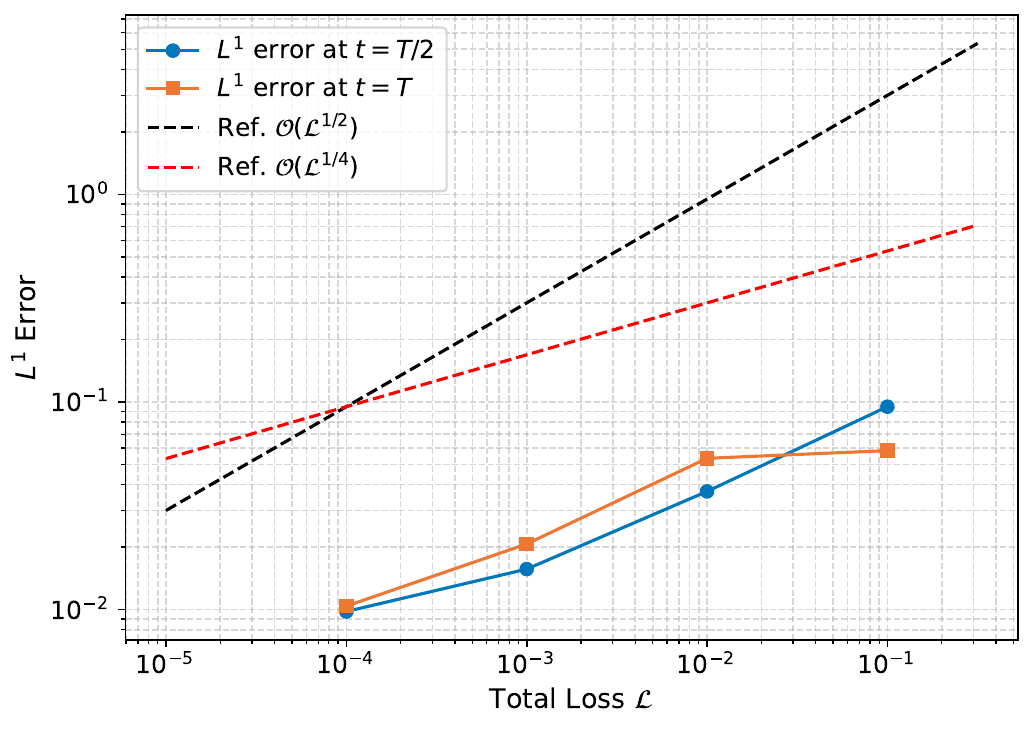}
        \caption{$L^1$ error on the Cartesian mesh, with curves at $t = T/2$ and $t = T$.}
        \label{fig:error_burger_half}
    \end{subfigure}
    \hfill
    \begin{subfigure}{0.48\textwidth}
        \centering
        \includegraphics[width=\linewidth]{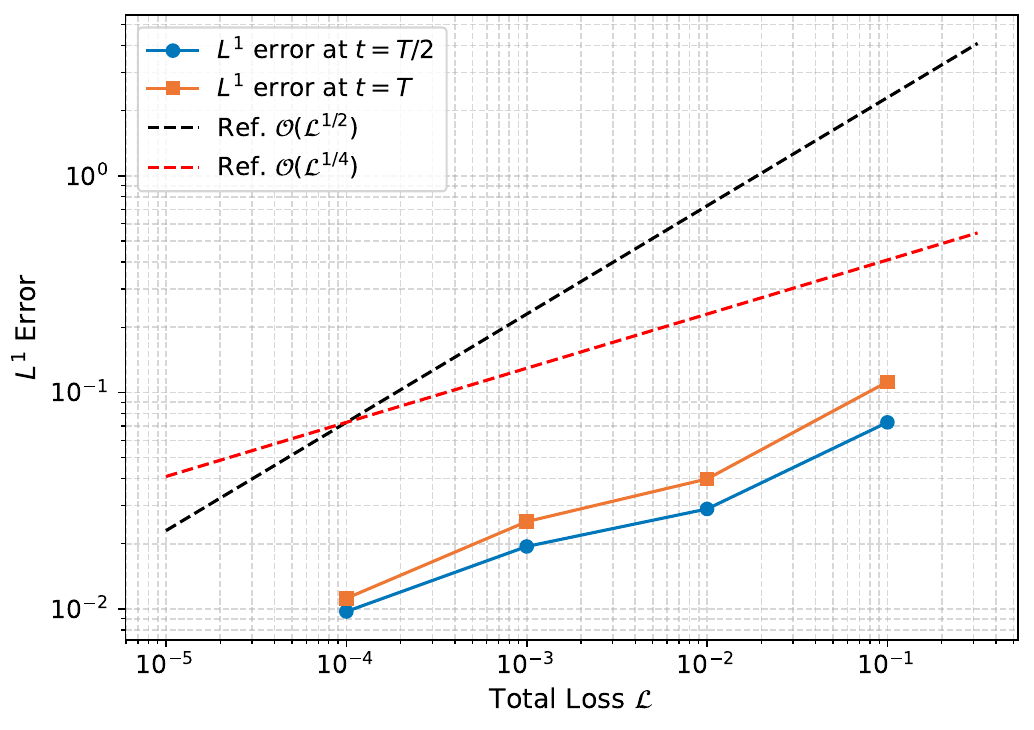}
        \caption{$L^1$ error on the random mesh, with curves at $t = T/2$ and $t = T$.}
        \label{fig:error_burger_full}
    \end{subfigure}
    \caption{$L^1$ error versus total loss $\mathcal{L}$ for the Burgers equation on the regular Cartesian mesh (left) and the random mesh (right), at $t = T/2$ (blue) and $t = T$ (orange). Dashed lines show reference slopes $\mathcal{O}(\mathcal{L}^{1/2})$ (black) and $\mathcal{O}(\mathcal{L}^{1/4})$ (red).}
    \label{fig:error_burger_study}
\end{figure}

\medskip
\noindent\textit{Study 2: Parametric analysis.}
Figure~\ref{fig:para_error_structured} reports the $L^1$ error for the Cartesian grid as a function of $\mathcal{L}$ for all combinations of $Q \in \{3, 5, 7, 9, 11\}$, $L \in \{4, 8, 16\}$, and $N \in \{16, 32, 64\}$, evaluated at $t = T/2$ and $t = T$ where $T=1$. The complete numerical data underlying this sweep are reported in~\ref{app:parametric}. Several observations emerge from these results. First, the empirical loss error trend is preserved across the vast majority of configurations and at both evaluation times, indicating that the discrete loss remains a useful surrogate for the $L^1$ error across the tested network architectures. Errors at $t = T/2$ are systematically smaller than at $t = T$, consistent with the fact that the solution has developed fewer sharp features at intermediate times. Second, increasing the quadrature order $Q$ consistently reduces both the total loss and the $L^1$ error, particularly for moderate network architectures; configurations with $Q \geq 7$ achieve markedly lower errors across all tested depths and widths. Third, the effects of network depth $L$ and width $N$ are non-monotone: moderate architectures such as $L = 8$ with $N = 32$ yield competitive accuracy across all quadrature orders, whereas excessively large networks do not provide systematic improvement. In particular, the configurations $(Q = 9,\, L = 16,\, N = 64)$ and $(Q = 11,\, L = 16,\, N = 64)$ exhibit optimization failures, with the $L^1$ error diverging well above the observed reference trend, suggesting that over-parameterized networks combined with high-order quadrature introduce ill-conditioning in the loss landscape. Taken together, these results establish that a moderate architecture comprising $L = 8$ hidden layers and $N = 32$ neurons per layer, paired with a quadrature order $Q \geq 7$, provides the best trade-off between accuracy, consistency with the observed loss error trend, and optimization stability. Results for the random mesh are shown in Fig.~\ref{fig:para_error_random}.

\begin{figure}[htbp]
    \centering
    \includegraphics[width=1\textwidth]{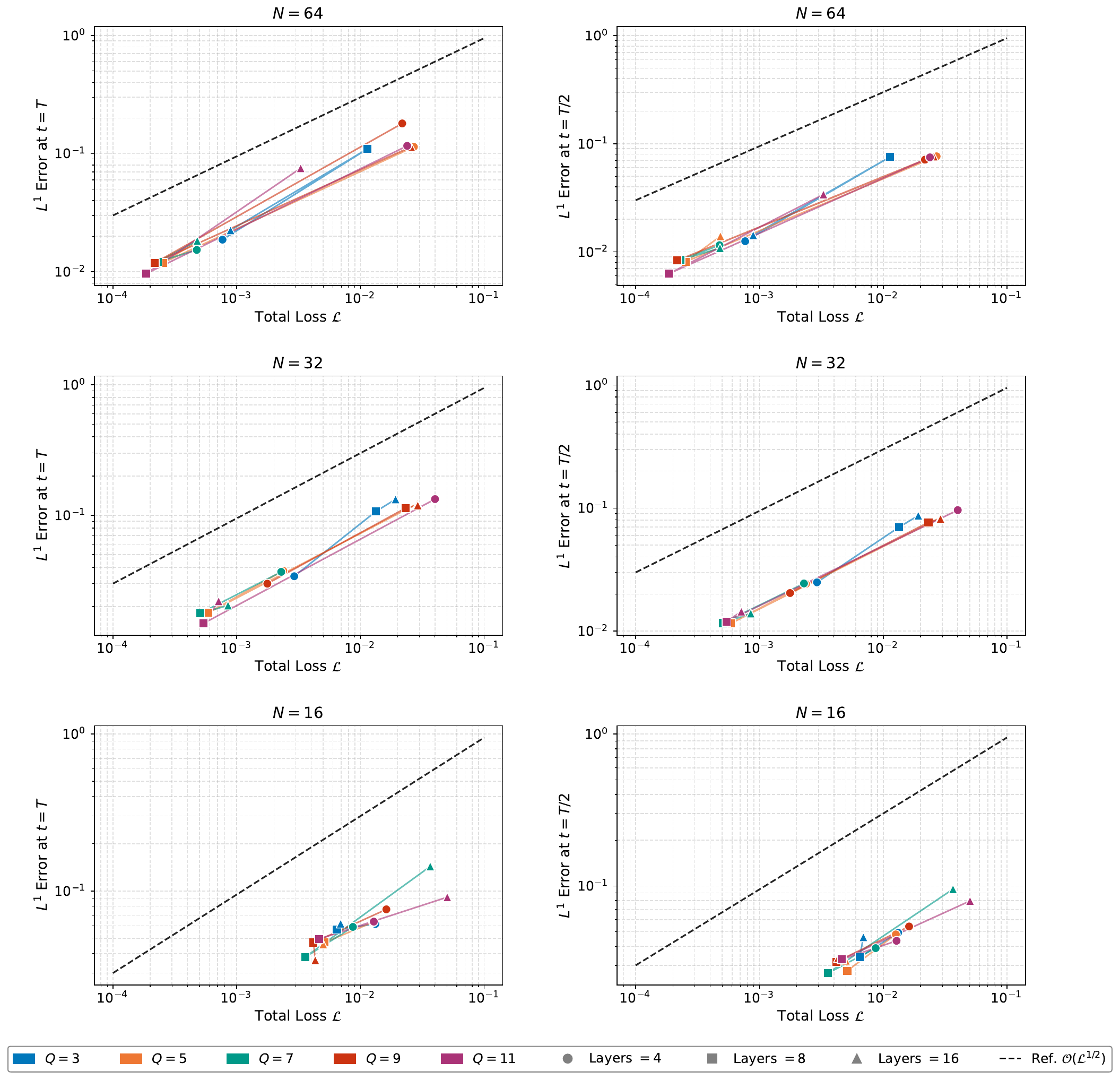}
    \caption{$L^1$ error as a function of the total loss $\mathcal{L}$ for the Burgers equation on the structured Cartesian grid, for varying quadrature order $Q$ (colors), network depth $L$ (markers), and number of neurons per layer $N$ (panels), evaluated at $t = T$ (left column) and $t = T/2$ (right column). The dashed line indicates the reference slope $\mathcal{O}(\mathcal{L}^{1/2})$.}
    \label{fig:para_error_structured}
\end{figure}

\begin{figure}[htbp]
    \centering
    \includegraphics[width=1\textwidth]{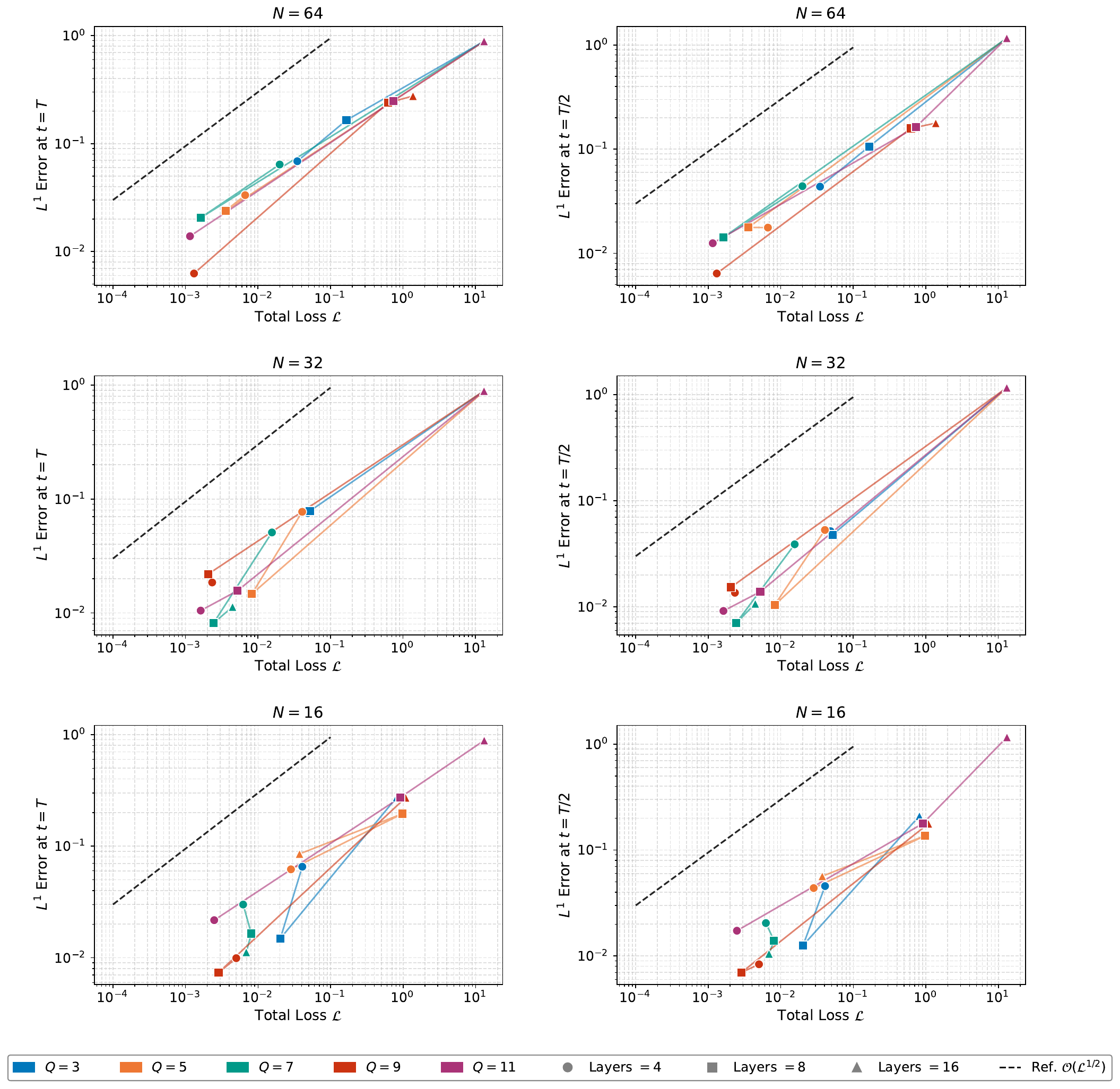}
    \caption{$L^1$ error as a function of the total loss $\mathcal{L}$ for the Burgers equation on the random (unstructured) grid, for varying quadrature order $Q$ (colors), network depth $L$ (markers), and number of neurons per layer $N$ (panels), evaluated at $t = T$ (left column) and $t = T/2$ (right column). The dashed line indicates the theoretical reference slope $\mathcal{O}(\mathcal{L}^{1/4})$.}
    \label{fig:para_error_random}
\end{figure}

Across the three Burgers benchmarks, Riemann shock, rarefaction--shock interaction, and sinusoidal smooth-to-shock evolution, WE-PINNs gives the smallest averaged errors in every reported norm. The gains are most pronounced in shock-dominated and mixed wave regimes, where strong form residual minimization is most limited. On the sinusoidal benchmark, cvPINNs is the closest competitor in $L^1$ and $L^2$, but WE-PINNs still remains the most accurate method overall. The empirical convergence studies in Section~\ref{subsec:burger_convergence} show a consistent decrease of the $L^1$ error as the total loss decreases; the last-decade slopes in Table~\ref{tab:convergence_slope} are approximately $0.30$ on the structured grid and $0.35$ on the random grid, so the data are presented as loss--error diagnostics rather than as evidence for a clean structured-grid asymptotic rate. Together, these results establish WE-PINNs as a robust mesh-free solver for scalar conservation laws and provide the foundation for the extensions to non-convex fluxes and to systems of conservation laws presented in the following sections.

\subsection{Buckley--Leverett Equation}
\label{sec:BL}

The WE-PINNs framework is next assessed on the Buckley--Leverett equation, a classical model for two-phase immiscible fluid flow in porous media arising in petroleum engineering and groundwater hydrology. The one-dimensional conservation law reads
\begin{equation}
  \partial_t u \;+\;
  \partial_x f(u) = 0,
  \qquad
  f(u) = \frac{u^{2}}{u^{2}+a(1-u)^{2}},
  \qquad (x,t)\in\Omega\times(0,T),
  \label{eq:BL}
\end{equation}
where $u\in[0,1]$ is the water saturation and the viscosity ratio is set to $a=0.25$. Unlike the Burgers flux, the Buckley--Leverett flux $f(u)$ is non-convex: it possesses an inflection point in the interval $(0,1)$, and consequently, a single strictly convex entropy pair is mathematically insufficient to select the unique physically admissible weak solution. The relevant selection mechanism is the Oleinik condition, which is equivalent to enforcing the full one-parameter Kru\v{z}kov entropy family~\cite{Kruzkov1970}
\begin{equation}
  \eta_\kappa(u) = |u-\kappa|,
  \qquad
  q_\kappa(u) = \mathrm{sgn}(u-\kappa)\bigl(f(u)-f(\kappa)\bigr),
  \qquad \kappa\in[0,1].
\end{equation}
In the WE-PINNs loss, the entropy inequality~\eqref{eq:entropy_inequality} is approximated for this family by stochastically sampling values of $\kappa\in[0,1]$ at each training epoch. This is the Buckley--Leverett specialization of the theoretical interval $K=[-M_\infty,M_\infty]$: since the saturation satisfies $u\in[0,1]$, the relevant entropy parameters are sampled in the physical range. The stochastic Kru\v{z}kov sampling strategy is therefore a Monte--Carlo approximation of the full family at finite training budget; the control-volume structure, the boundary flux quadrature, and the network architecture remain unchanged. All experiments are performed on the spatial domain $\Omega=[-1,1]$, with final time $T=0.70$ for the compound wave problem and $T=0.2$ for the rarefaction--shock interaction, chosen so that the wave structure remains within the domain throughout the simulation. As for the Burgers equation, the neural network comprises $8$ hidden layers of $32$ neurons each with $\tanh$ activation and 8-point Gauss--Legendre quadrature.

\subsubsection{Compound Wave Problem (Water Injection)}
\label{subsec:BL_shock}

The first test case is the Riemann initial condition associated with a water injection process,
\begin{equation}
  u(x,0) =
  \begin{cases}
    1, & x \le -0.5, \\
    0, & x > -0.5,
  \end{cases}
\end{equation}
which models the displacement of oil by water across a sharp front. Owing to the non-convex flux, the exact entropy solution does not consist of a single shock but develops a compound wave: a smooth rarefaction fan directly attached to a leading shock propagating at the Welge tangent speed. This structure is notoriously difficult for standard numerical methods and classical PINNs, both of which tend to converge toward non-physical weak solutions or systematically misestimate the shock speed when the full entropy condition is not enforced~\cite{fuks2020limitations}.

The WE-PINNs solution at three reference times is shown in Fig.~\ref{fig:BL_Rarefaction_Shock}, compared against the exact entropy solution. The rarefaction portion is resolved smoothly and the leading shock is captured sharply at the analytical Welge location, with no spurious oscillations near the front. The standard PINNs, by contrast, produce a heavily diffused profile and fail to recover both the compound structure and the correct shock position, illustrating the failure mode predicted by the strong-form analysis of~\cite{fuks2020limitations}.

\begin{figure}[H]
  \centering
  \includegraphics[width=\textwidth]{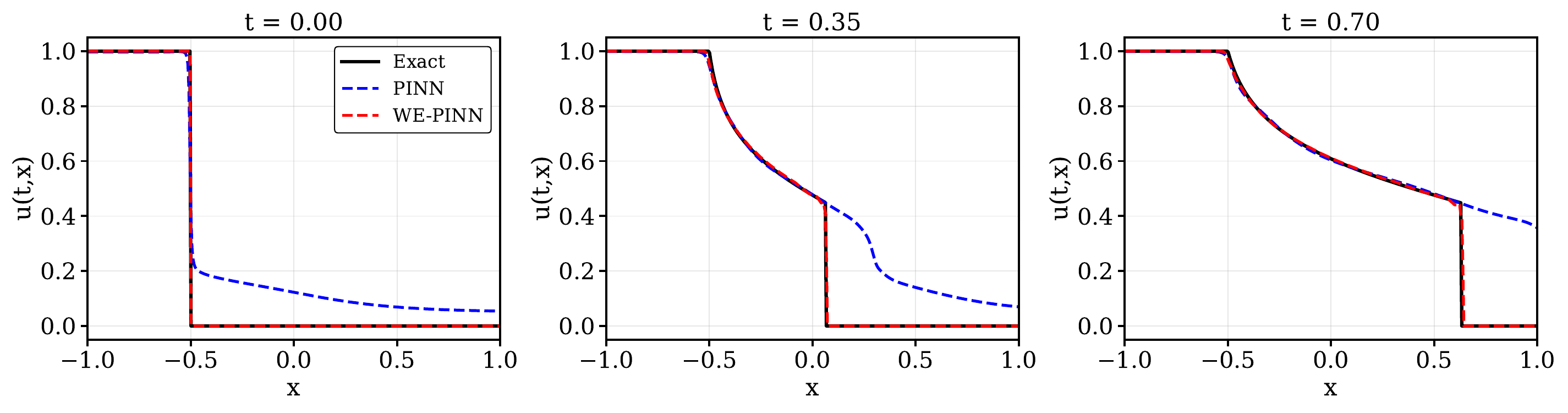}
  \caption{Solution profiles for the Buckley--Leverett compound wave (water injection) problem at $t=0$, $t=0.35$, and $t=0.70$. The smooth rarefaction fan and the attached leading shock are both captured accurately by WE-PINNs. Comparison against the exact entropy solution and the standard PINNs.}
  \label{fig:BL_Rarefaction_Shock}
\end{figure}

The training history is reported in Fig.~\ref{fig:BL_loss}. The total WE-PINNs loss decreases steadily, with both individual components $\mathcal{L}_{\mathrm{cons}}$ and $\mathcal{L}_{\mathrm{ent}}$ reaching low plateau values, confirming that the conservation identity and the Kru\v{z}kov entropy family are jointly enforced throughout training. The PDE residual of the standard PINNs stagnates at a significantly higher level throughout the entire optimization run. The space--time evolution in Fig.~\ref{fig:BL_shock_spacetime} further confirms that WE-PINNs reproduces the compound wave structure with a relative $L^2$ error tightly localized near the shock trajectory, whereas the standard PINNs error spreads broadly across the domain.

\begin{figure}[H]
  \centering
  \includegraphics[width=\textwidth]{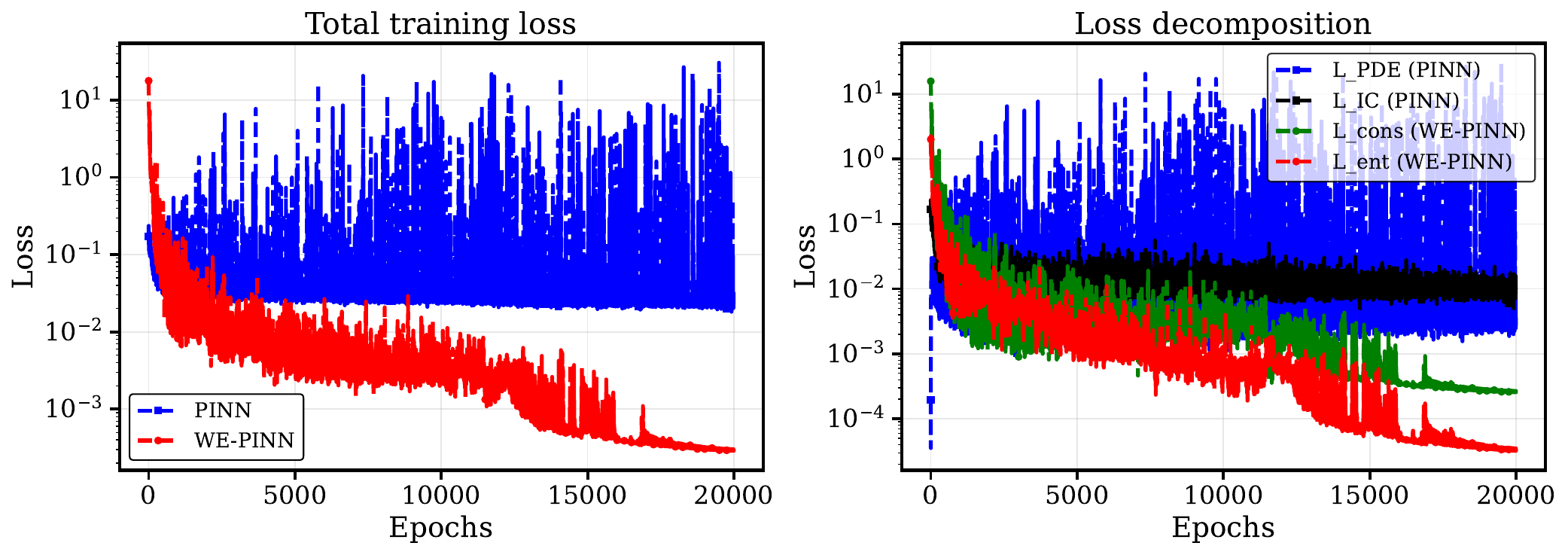}
  \caption{Training loss history for the Buckley--Leverett compound wave problem. \emph{Left}: total loss for both methods. \emph{Right}: displayed individual components $\mathcal{L}_{\mathrm{cons}}$ and $\mathcal{L}_{\mathrm{ent}}$ (WE-PINNs); $\mathcal{L}_{\mathrm{PDE}}$ and $\mathcal{L}_{\mathrm{IC}}$ (standard PINNs). The TVD component is included in the total WE-PINNs loss but is not shown separately in this panel.}
  \label{fig:BL_loss}
\end{figure}

\begin{figure}[H]
  \centering
  \includegraphics[width=\textwidth]{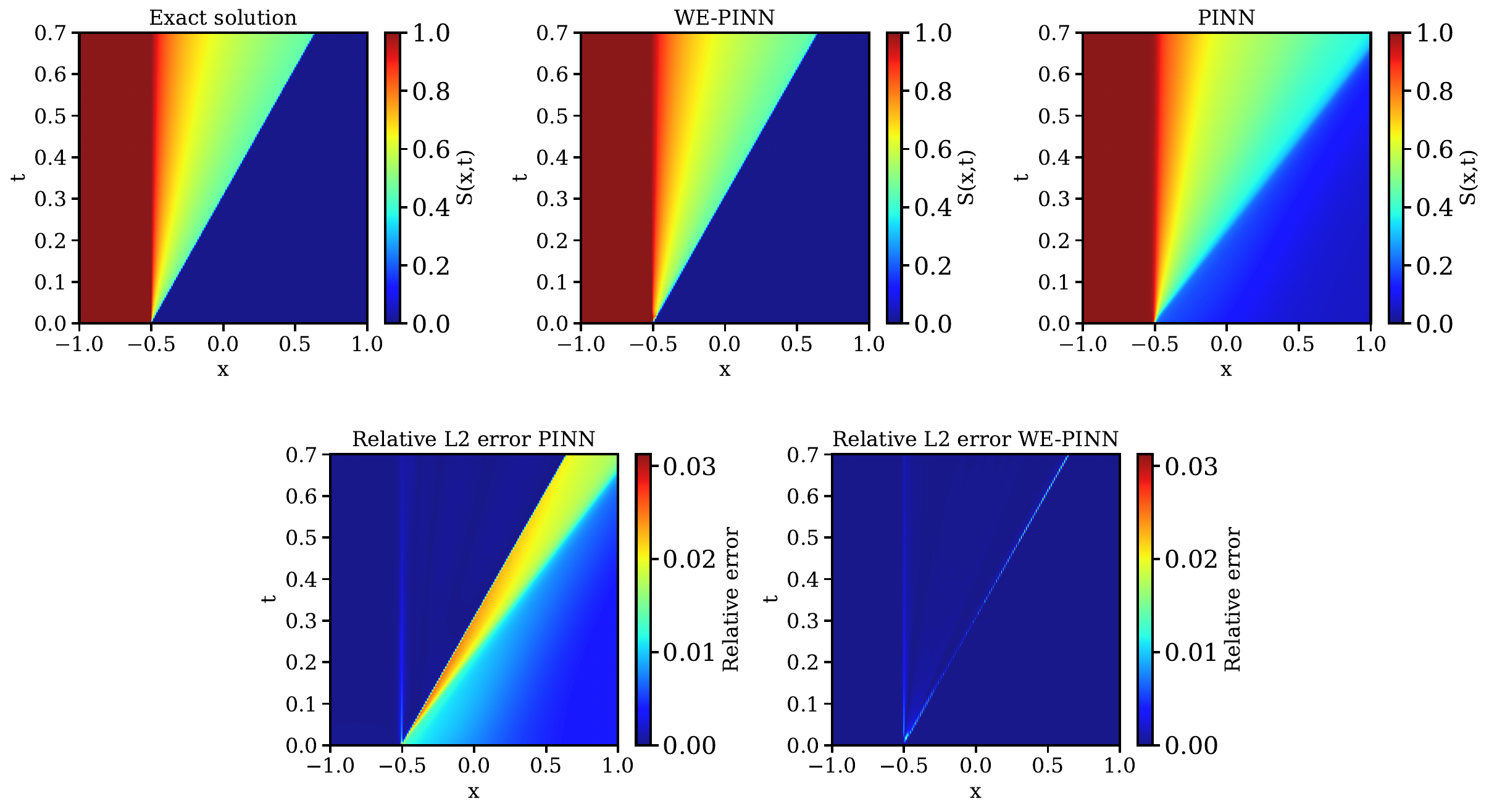}
  \caption{Space--time evolution of the Buckley--Leverett solution for the compound wave problem. \emph{Top row}: Exact solution, WE-PINNs solution, and standard PINNs solution. \emph{Bottom row}: pointwise relative $L^2$ error maps. The compound wave structure, a smooth rarefaction fan directly attached to a leading shock, is reproduced faithfully by WE-PINNs, while the standard PINNs solution is strongly diffused.}
  \label{fig:BL_shock_spacetime}
\end{figure}

The relative errors averaged over $t=0$, $0.35$, and $0.70$ are listed in Table~\ref{tab:BL_shock_errors}. WE-PINNs gives the best errors in all three norms, with CI-PINNs the closest competitor. In $L^1$, the error is $5.11\times10^{-3}$, about one order of magnitude below cvPINNs and nearly $38$ times below the standard PINNs. The $L^2$ and $L^\infty$ results show the same behavior: CI-PINNs remains competitive, but WE-PINNs keeps the compound wave sharper and places the leading shock more accurately.
\begin{table}[htbp]
\centering
\caption{Relative $L^1$, $L^2$, and $L^\infty$ errors (in units of $10^{-3}$) averaged over $t=0$, $0.35$, and $0.70$ for the Compound Wave Problem (Water Injection), shown for WE-PINNs, CI-PINNs, cvPINNs, and the standard PINNs. \textbf{Bold} indicates the best result and \underline{underline} the second best in each error group.}
\label{tab:BL_shock_errors}
\begin{tabular}{lccc}
\toprule
Model &  $E_{L^1}$ &  $E_{L^2}$ & $E_{L^\infty}$   \\
\midrule
WE-PINNs       & \textbf{5.11}    & \textbf{14.10}    & \textbf{116.30} \\
CI-PINNs        & \underline{7.04} & \underline{23.56} & \underline{188.90} \\
cvPINNs         & 54.46            & 129.70            & 538.00 \\
Standard PINNs  & 192.70           & 256.80            & 448.10 \\

\bottomrule
\end{tabular}
\end{table}

The compound wave configuration is particularly discriminating: methods that do not enforce the complete Kru\v{z}kov entropy family systematically converge to non-physical weak solutions or misestimate the shock speed, since the non-convex flux admits multiple weak solutions that satisfy a single entropy pair but violate the Oleinik condition. The results in Table~\ref{tab:BL_shock_errors} show that this entropy treatment improves accuracy without adding auxiliary potential networks, saddle-point optimization, or discrete flux reconstruction. The stochastic Kru\v{z}kov sampling strategy of Section~\ref{sec:ent_loss} is the only algorithmic modification relative to the convex-flux case, yet it suffices to enforce the Oleinik entropy condition and select the physical compound wave solution.

\subsubsection{Rarefaction--Shock Interaction}
\label{subsec:BL_creno}

The second test case extends the analysis to a configuration involving two simultaneously active wave families. The compactly supported initial condition
\begin{equation}
  u(x,0) =
  \begin{cases}
    1, & -0.5 < x \le 0, \\
    0, & \text{otherwise},
  \end{cases}
\end{equation}
generates a rarefaction wave at the left interface $x=-0.5$ and a compound wave at the right interface $x=0$, with final time $T=0.2$. This configuration is significantly more demanding than the corresponding Burgers benchmark of Section~\ref{subsec:burger_interaction}: a compound wave structure must be resolved at the right interface, a smooth rarefaction must be resolved simultaneously at the left interface, and the non-physical weak solutions permitted by a purely convex entropy pair must be ruled out by the full Kru\v{z}kov family. The test therefore probes both the geometric flexibility of the space--time control-volume formulation and the effectiveness of the stochastic Kru\v{z}kov sampling strategy in selecting the entropy solution under a non-convex flux.

The solution profiles in Fig.~\ref{fig:BL_creno} show that WE-PINNs captures both the smooth rarefaction portion and the sharp shock with no spurious oscillations and no visible phase error at the shock position, in close agreement with the Godunov finite-volume reference computed with $N_x=2000$ cells. The space--time plot in Fig.~\ref{fig:BL_creno_spacetime} confirms that the two wave fronts and their subsequent interaction are reproduced faithfully throughout the simulation, with the relative $L^2$ error of WE-PINNs tightly localized along the wave trajectories.

\begin{figure}[H]
  \centering
  \includegraphics[width=\textwidth]{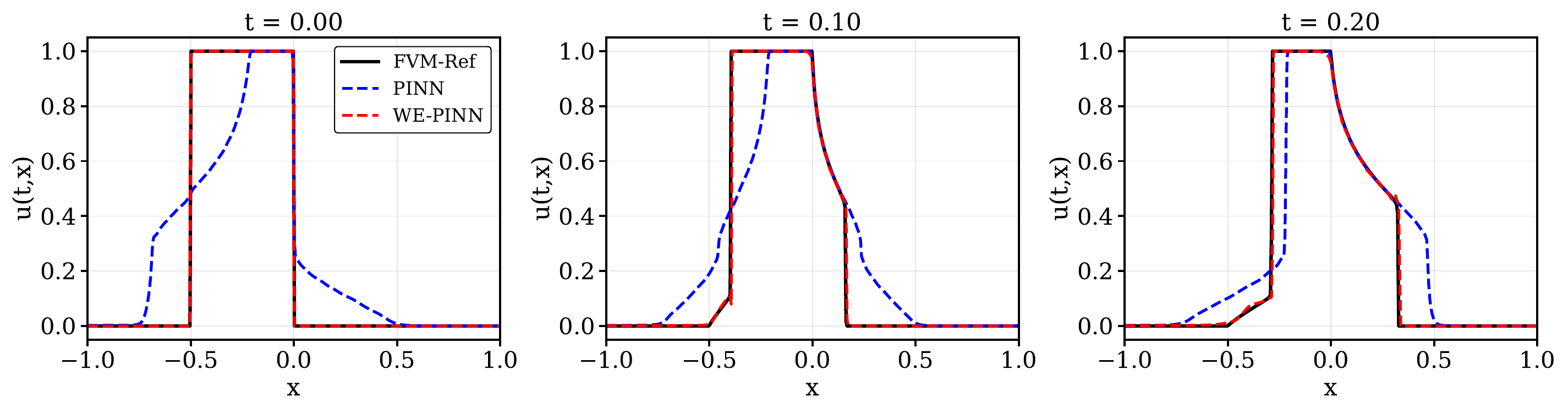}
  \caption{Solution profiles for the Buckley--Leverett rarefaction--shock interaction problem at selected times. Both the smooth rarefaction fan and the leading compound shock are resolved accurately by WE-PINNs. Comparison against the Godunov finite-volume reference ($N_x=2000$) and the standard PINNs.}
  \label{fig:BL_creno}
\end{figure}

\begin{figure}[htbp]
  \centering
  \includegraphics[width=\textwidth]{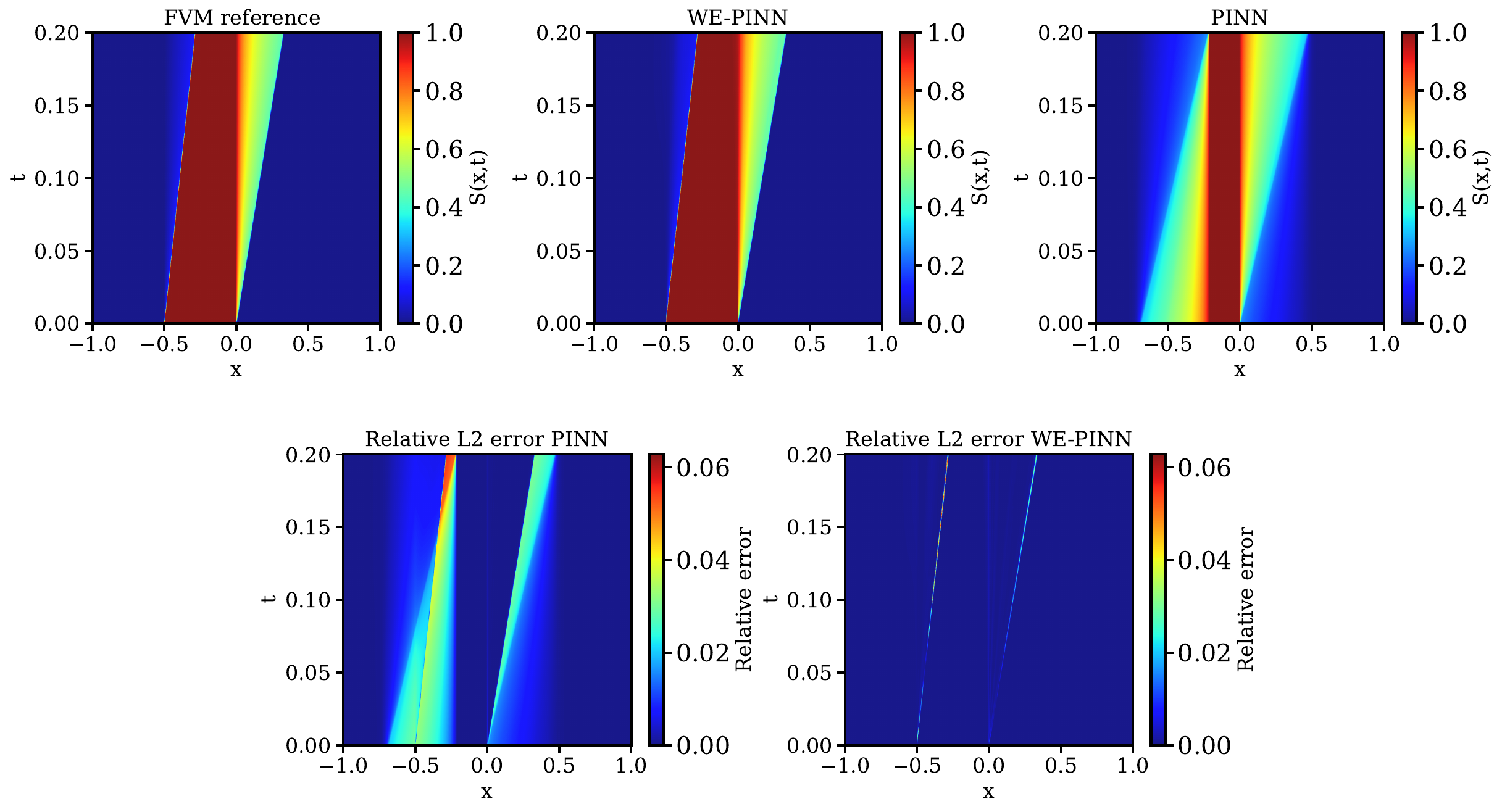}
  \caption{Space--time evolution of the Buckley--Leverett solution for the rarefaction--shock interaction problem. \emph{Top row}: reference solution, WE-PINNs solution, and standard PINNs solution. \emph{Bottom row}: pointwise relative $L^2$ error maps.}
  \label{fig:BL_creno_spacetime}
\end{figure}

The corresponding relative errors averaged over $t=0$, $0.10$, and $0.20$ are reported in Table~\ref{tab:BL_Rarefaction_Shock_errors}. WE-PINNs is the most accurate method in this second non-convex test as well. Its $L^1$ error is $8.21\times10^{-3}$, about $7$ times below cvPINNs and nearly $29$ times below the standard PINNs, while CI-PINNs remains between these two baselines. The $L^2$ and $L^\infty$ errors follow the same trend, with the largest gaps appearing near the compound shock.
\begin{table}[htbp]
\centering
\caption{Relative $L^1$, $L^2$, and $L^\infty$ errors (in units of $10^{-3}$) averaged over $t=0$, $0.10$, and $0.20$ for the rarefaction--shock interaction, shown for WE-PINNs, CI-PINNs, cvPINNs, and the standard PINNs. \textbf{Bold} indicates the best result and \underline{underline} the second best in each error group.}
\label{tab:BL_Rarefaction_Shock_errors}
\begin{tabular}{lccc}
\toprule
Model &  $E_{L^1}$ &  $E_{L^2}$ & $E_{L^\infty}$   \\
\midrule
WE-PINNs       & \textbf{8.21}     & \textbf{21.39}     & \textbf{173.80} \\
CI-PINNs        & 145.30            & 286.30             & 805.90 \\
cvPINNs         & \underline{56.13} & \underline{104.60} & \underline{416.60} \\
Standard PINNs  & 238.50            & 571.70             & 793.20 \\
\bottomrule
\end{tabular}
\end{table}
Taken together, the two Buckley--Leverett benchmarks confirm that WE-PINNs enforces the Oleinik entropy condition through the stochastic Kru\v{z}kov sampling strategy and selects the physically admissible solution in the presence of a non-convex flux, a regime where the standard PINNs and several competing methods fail to recover the correct wave structure. The results validate the extension of the WE-PINNs framework beyond the strictly convex case and motivate its application to hyperbolic systems in the following sections.


\subsection{Compressible Euler System}
\label{sec:euler}

We next consider the one-dimensional compressible Euler equations, which govern the motion of an inviscid compressible fluid and constitute a prototypical system of nonlinear hyperbolic conservation laws. In conservative form, the system reads
\begin{equation}
    \partial_t \mathbf{U} + \partial_x \mathbf{F}(\mathbf{U}) = 0,
\end{equation}
where the vector of conserved variables and the associated flux function are given by
\begin{equation}
    \mathbf{U} =
    \begin{pmatrix}
        \rho \\ \rho u \\ E
    \end{pmatrix},
    \qquad
    \mathbf{F}(\mathbf{U}) =
    \begin{pmatrix}
        \rho u \\ \rho u^2 + p \\ u(E + p)
    \end{pmatrix}.
\end{equation}
Here, $\rho$ denotes the mass density, $u$ the velocity, $p$ the pressure, and $E$ the total energy per unit volume. The system is closed by the ideal gas equation of state
\begin{equation}
    E = \frac{1}{2}\rho u^2 + \frac{p}{\gamma - 1},
\end{equation}
where $\gamma > 1$ is the ratio of specific heats. The entropy pair associated with this system is $\eta(\mathbf{U}) = -\rho s$ and $q(\mathbf{U}) = -\rho u s$, where $s = \ln(p / \rho^{\gamma})$ is the specific entropy of an ideal gas (see~\ref{app:entropies}).

The Euler system admits complex wave interactions involving shock waves, rarefaction waves, and contact discontinuities. The eigenvalues of the flux Jacobian $\partial \mathbf{F}/\partial \mathbf{U}$ are $\lambda_{1,2,3} = \{ u - c,\, u,\, u + c \}$, where $c = \sqrt{\gamma p / \rho}$ is the local speed of sound. The first and third characteristic families are genuinely nonlinear and generate shock waves and rarefaction fans, while the second family is linearly degenerate and gives rise to contact discontinuities across which density and entropy may jump while pressure and velocity remain continuous. The coexistence of these three wave families, each imposing distinct regularity requirements on the numerical approximation, makes the Euler system a considerably more demanding test than the scalar conservation laws of the preceding sections. To accommodate this increased complexity, we employ a network with $8$ hidden layers of $40$ neurons each with $\tanh$ activation, trained for $10000$ epochs in the 1D case, and $8$ hidden layers of $64$ neurons each with $\tanh$ activation, trained for $15000$ epochs in the 2D Euler case. Conservation is enforced through boundary flux integrals over dynamically sampled space--time control volumes, and entropy admissibility is incorporated in integral form via the constraint~\eqref{eq:entropy_inequality}.

\subsubsection{Sod Shock Tube Problem}
\label{sec:sod}

The classical Sod shock tube problem is considered as the primary benchmark for the one-dimensional Euler system. The computational domain is $\Omega = [-1,1]$, and the initial conditions are prescribed in terms of the primitive variables $(\rho, u, p)$ as
\begin{equation}
    \rho(x,0) =
    \begin{cases}
        1,     & x \leq 0, \\
        0.125, & x > 0,
    \end{cases}
    \qquad
    u(x,0) = 0,
    \qquad
    p(x,0) =
    \begin{cases}
        1,   & x \leq 0, \\
        0.1, & x > 0,
    \end{cases}
    \label{eq:sod_ic}
\end{equation}
with $\gamma = 1.4$. The exact entropy solution of this Riemann problem consists of three distinct wave structures propagating from the initial discontinuity at $x = 0$: a left-propagating centered rarefaction wave, a contact discontinuity travelling at intermediate speed, and a right-propagating shock wave. This benchmark is particularly demanding because it requires the simultaneous resolution of a smooth rarefaction fan, a sharp density jump across the contact discontinuity, and a strong compressive shock, each of which imposes distinct requirements on the accuracy and stability of the approximation.

The WE-PINNs solution profiles for density, velocity, pressure, and total energy at $t = 0$, $t = 0.2$, and $t = 0.4$ are shown in Fig.~\ref{fig:Sod_Shock_Tube}, compared against the exact entropy solution and the standard PINNs trained with an identical architecture and optimization budget. At all three reported times, the WE-PINNs solution captures the three-wave structure with high fidelity. The rarefaction fan is resolved smoothly without artificial steepening, the contact discontinuity is maintained as a sharp transition in the density and energy fields, and the shock wave is captured with minimal numerical smearing and without spurious post-shock oscillations. The correct shock speed follows directly from the integral conservation constraint, which implicitly enforces the Rankine--Hugoniot jump conditions without any special shock-tracking treatment. The standard PINNs, by contrast, significantly diffuse all three wave structures, underestimate the shock strength, and fail to maintain the sharp density jump at the contact discontinuity, a behavior consistent with the theoretical observation that strong-form residual minimization penalizes sharp gradients and biases the network towards overly smooth approximations~\cite{de2024wpinns}.

\begin{figure}[htbp]
    \centering
    \includegraphics[width=1.0\textwidth]{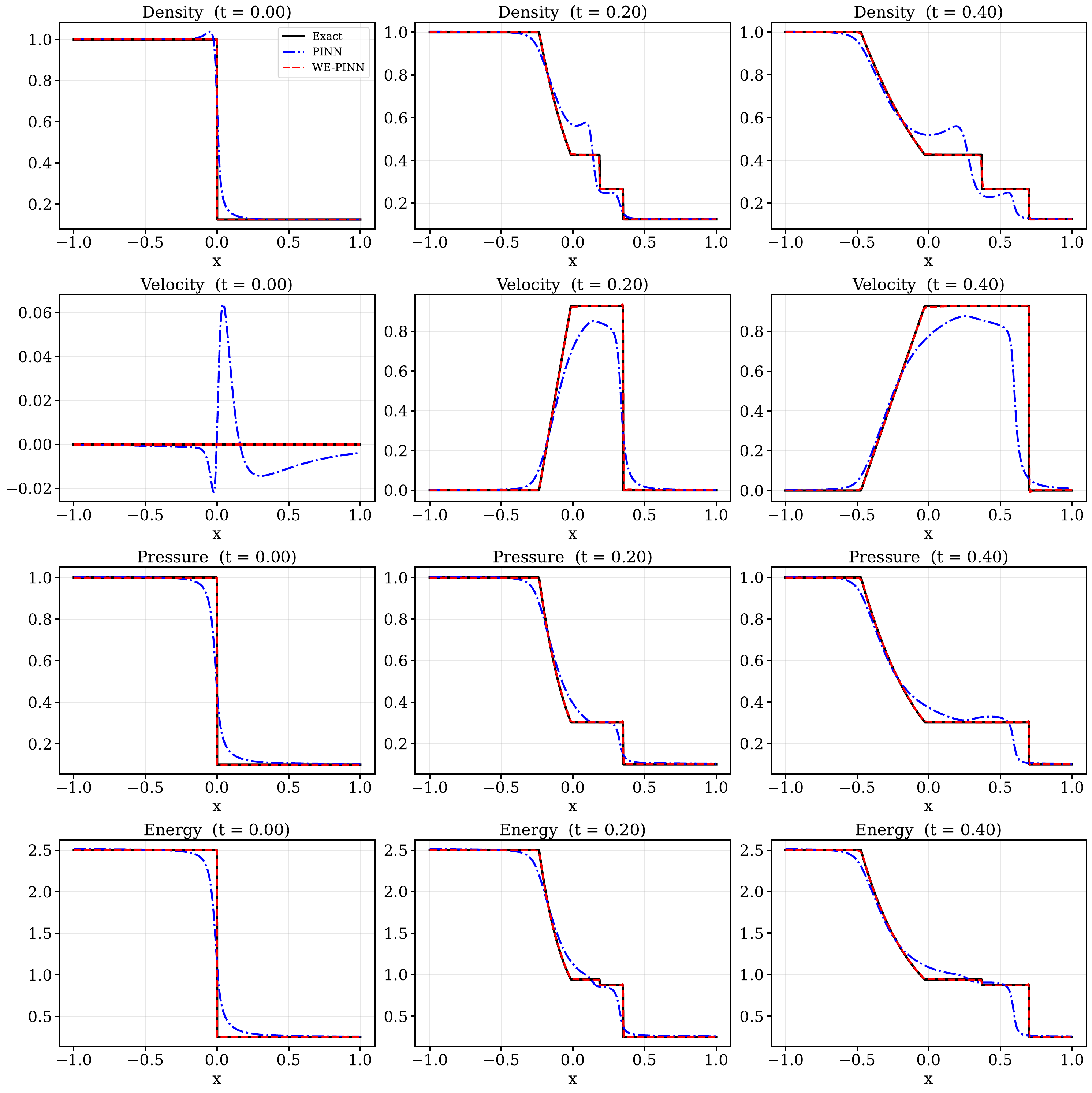}
    \caption{WE-PINNs solution of the one-dimensional Euler system for the Sod shock tube problem at $t = 0$, $0.2$, and $0.4$. Rows correspond to density $\rho$, velocity $u$, pressure $p$, and total energy $E$ over the domain $[-1,1]$. The exact entropy solution (black solid), WE-PINNs (red dashed), and the standard PINNs (blue dash-dotted) are superimposed. WE-PINNs accurately resolves the left-propagating rarefaction fan, the contact discontinuity, and the right-propagating shock wave simultaneously, while the standard PINNs diffuse all three wave structures.}
    \label{fig:Sod_Shock_Tube}
\end{figure}

The space--time evolution of density, velocity, and pressure is presented in Fig.~\ref{fig:Sod_space_time}. For each variable, the exact solution, the WE-PINNs approximation, the standard PINNs solution, and the corresponding relative $L^2$ error maps are shown side by side over the domain $[-1,1] \times [0, 0.4]$. The characteristic wave trajectories are clearly visible and consistent with the exact solution: the rarefaction fan spreads at the correct self-similar rate, the contact discontinuity translates at a nearly constant speed, and the shock front propagates as a sharp diagonal feature in the space--time plane. The plotted normalized pointwise error of WE-PINNs remains tightly localized along the wave trajectories. Its scale should not be identified with the global relative $L^\infty$ errors reported in the tables, which are computed from Eq.~\eqref{eq:rel_error} and are more sensitive to pointwise mismatch at the contact discontinuity and shock.
\begin{figure}[htbp]
    \centering
    \includegraphics[width=1\textwidth]{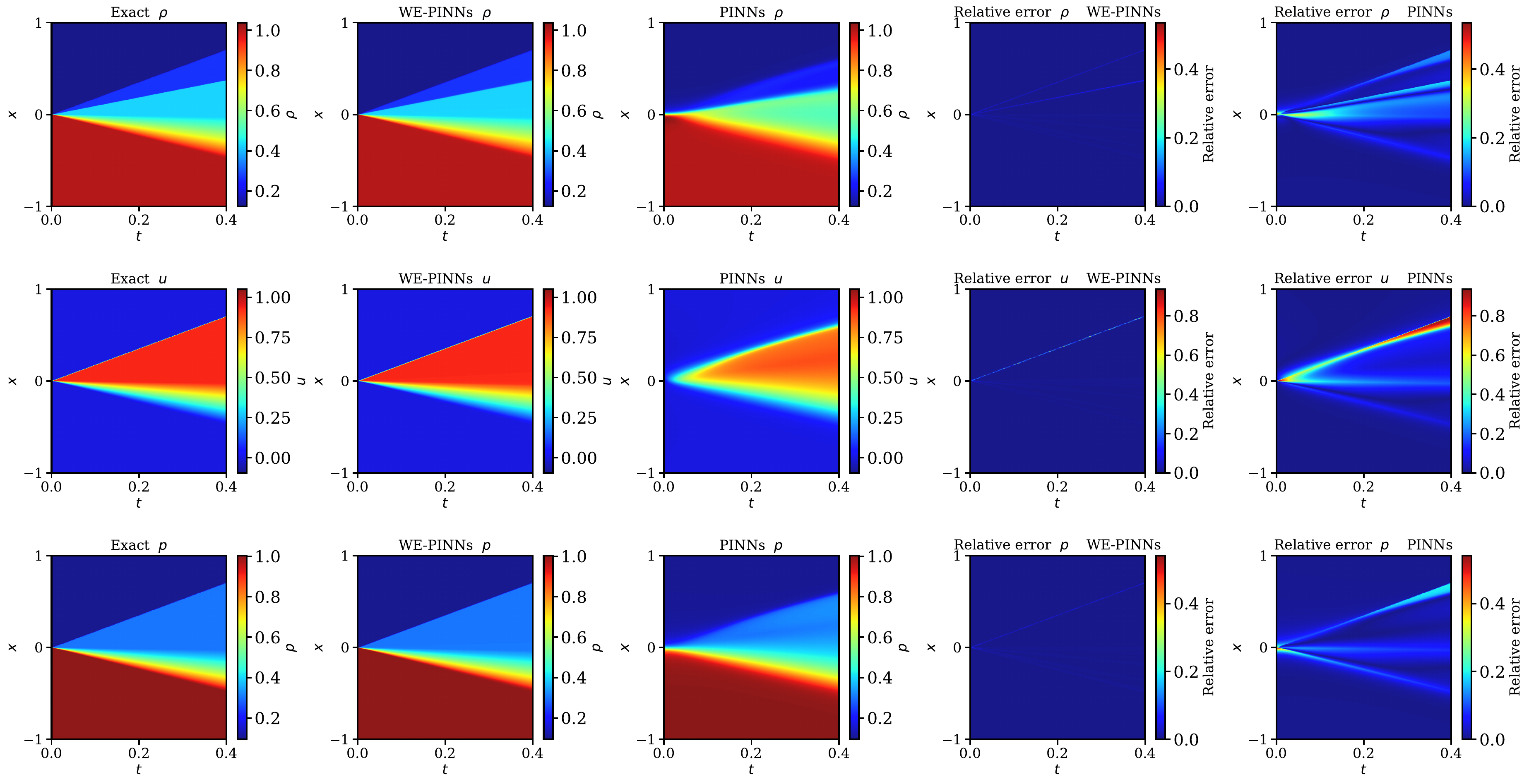}
    \caption{Space--time solutions and pointwise relative errors for the Sod shock tube problem. Rows correspond to density $\rho$, velocity $u$, and pressure $p$. Columns show, from left to right: exact solution, WE-PINNs solution, standard PINNs solution, relative $L^2$ error of WE-PINNs, and relative $L^2$ error of the standard PINNs.}
    \label{fig:Sod_space_time}
\end{figure}

The relative errors averaged over $t = 0.0$, $0.2$, and $0.4$ are reported in Table~\ref{tab:euler_errors} for all four conserved variables, comparing WE-PINNs against CI-PINNs, cvPINNs, and the standard PINNs. WE-PINNs gives the smallest $L^1$ error for every variable: $7.80\times10^{-4}$ for density, $5.45\times10^{-3}$ for velocity, $6.50\times10^{-4}$ for pressure, and $6.29\times10^{-4}$ for energy. For density, this is about $19$ times below cvPINNs and nearly two orders of magnitude below CI-PINNs and the standard PINNs. The advantage carries over to $L^2$ and $L^\infty$, with WE-PINNs errors remaining at the $10^{-3}$ level in $L^2$ for $\rho$, $p$, and $E$, and below $8.5\times10^{-2}$ in $L^\infty$ for all variables. The velocity component shows the strongest contrast: the standard PINNs heavily diffuse the contact discontinuity, while the integral formulation keeps the wave structure sharply localized and preserves the thermodynamic coupling across the conserved components.
\begin{table}[H]
\centering
\small
\setlength{\tabcolsep}{6pt}
\caption{Relative $L^1$, $L^2$, and $L^\infty$ errors (in units of $10^{-3}$)
averaged over $t=0$, $0.2$, and $0.4$ for the Sod shock tube problem.
\textbf{Bold} indicates the best result and \underline{underline} the second
best in each error group.}
\label{tab:euler_errors}
\begin{tabular}{llcccc}
\toprule
Norm & Method & $\rho$ & $u$ & $p$ & $E$ \\
\midrule
\multirow{4}{*}{$E_{L^1}$}
& WE-PINNs & \textbf{0.78} & \textbf{5.45} & \textbf{0.65} & \textbf{0.63} \\
& CI-PINNs & 46.90 & \underline{17.06} & 61.24 & 57.11 \\
& cvPINNs  & \underline{15.12} & 48.20 & \underline{9.62} & \underline{10.04} \\
& PINNs    & 46.10 & 378.90 & 43.89 & 43.62 \\
\midrule
\multirow{4}{*}{$E_{L^2}$}
& WE-PINNs & \textbf{3.79} & \textbf{20.87} & \textbf{3.39} & \textbf{3.46} \\
& CI-PINNs & 69.71 & 37.87 & 87.25 & 86.49 \\
& cvPINNs  & \underline{23.83} & \underline{28.57} & \underline{27.99} & \underline{31.87} \\
& PINNs    & 78.30 & 312.40 & 74.64 & 78.24 \\
\midrule
\multirow{4}{*}{$E_{L^\infty}$}
& WE-PINNs & \textbf{56.84} & \textbf{41.04} & \textbf{80.13} & \textbf{84.49} \\
& CI-PINNs & 282.80 & \underline{120.70} & 305.50 & 329.40 \\
& cvPINNs  & \underline{229.80} & 385.20 & \underline{297.70} & 330.30 \\
& PINNs    & 297.00 & 749.60 & 824.00 & \underline{326.40} \\
\bottomrule
\end{tabular}
\end{table}

The Sod results show that the weak space--time control-volume formulation extends naturally to systems of conservation laws. Weak conservation of the three components of $\mathbf{U}$ together with the integral entropy inequality yields a thermodynamically consistent multi-component solution without component-wise tuning or system-specific stabilization, and the sharp, oscillation-free resolution of the rarefaction, contact, and shock confirms that the Rankine--Hugoniot conditions and entropy admissibility are enforced weakly. These properties, established conditionally for the scalar case in Section~\ref{sec:convergence}, are observed to carry over to the system setting. A reproducibility study over $10$ independent runs (\ref{app:multirun}) confirms low sensitivity to initialization, with the standard deviation of the time-averaged $L^1$ error of order $10^{-4}$ for density and pressure and $10^{-3}$ for velocity.

\subsection{Two-Dimensional Euler Equations}
\label{sec:euler2d}

The WE-PINNs framework is further assessed on the two-dimensional compressible Euler equations, which represent the most demanding configuration considered in this work. The governing system takes the conservative form
\begin{equation}
    \frac{\partial \mathbf{U}}{\partial t}
    + \frac{\partial \mathbf{F}(\mathbf{U})}{\partial x}
    + \frac{\partial \mathbf{G}(\mathbf{U})}{\partial y}
    = 0,
    \label{eq:euler2d}
\end{equation}
where the conserved variable vector and the flux functions in the $x$- and $y$-directions are given by
\begin{equation}
    \mathbf{U} =
    \begin{pmatrix}
        \rho \\ \rho u \\ \rho v \\ E
    \end{pmatrix},
    \quad
    \mathbf{F}(\mathbf{U}) =
    \begin{pmatrix}
        \rho u \\ \rho u^2 + p \\ \rho u v \\ u(E + p)
    \end{pmatrix},
    \quad
    \mathbf{G}(\mathbf{U}) =
    \begin{pmatrix}
        \rho v \\ \rho u v \\ \rho v^2 + p \\ v(E + p)
    \end{pmatrix}.
\end{equation}
Here, $\rho(x,y,t)$ is the mass density, $(u,v)$ the velocity components, $p$ the pressure, and $E$ the total energy per unit volume, given by
\begin{equation}
    E = \frac{p}{\gamma - 1} + \frac{1}{2}\rho\left(u^2 + v^2\right),
\end{equation}
with $\gamma = 1.4$. The computational domain is $(x,y) \in [-1,1]^2$, and two discontinuous initial configurations of qualitatively different wave geometry are considered.

\subsubsection{Quadrant Riemann Problem}
\label{sec:euler2d_quadrant}

The first two-dimensional test case is the quadrant Riemann problem, in which the domain is partitioned into four quadrants with piecewise constant initial states:
\begin{equation}
    (\rho, u, v, p)(x,y,0) =
    \begin{cases}
        (2.0,\;\;0.75,\;\;\;0.5,\;\;1.0), & x \leq 0,\; y > 0, \\
        (1.0,\;\;0.75,\;-0.5,\;\;1.0),   & x > 0,\; y > 0, \\
        (1.0,\;-0.75,\;\;0.5,\;\;1.0),   & x \leq 0,\; y \leq 0, \\
        (3.0,\;-0.75,\;-0.5,\;\;1.0),    & x > 0,\; y \leq 0.
    \end{cases}
    \label{eq:quadrant_ic}
\end{equation}
This configuration generates a rich interaction pattern involving shock waves, rarefaction fans, and contact surfaces emanating simultaneously from all four quadrant boundaries and subsequently colliding near the origin of the domain. The resulting solution features strong spatial gradients, complex shock--shock and shock--rarefaction interactions, and regions of genuinely two-dimensional wave dynamics that cannot be accurately captured by dimension-by-dimension operator-splitting approaches. As no analytical closed-form reference solution exists for this configuration, the assessment combines qualitative consistency with high-resolution finite-volume computations and a quantitative comparison against a Godunov finite-volume reference on a $512\times512$ grid.

The WE-PINNs solution at selected times is shown in Fig.~\ref{fig:Euler_quadrant}. The density field correctly develops the characteristic four-way wave interaction pattern, with the quadrant boundaries giving rise to well-defined contact discontinuities and compressive shocks that propagate, interact, and merge over time. The velocity components $u$ and $v$ are resolved symmetrically with respect to the quadrant geometry and display the antisymmetric pattern consistent with the initial condition, confirming that the vector conservation structure of the system is respected throughout the simulation. The pressure field evolves towards a more uniform distribution in the central interaction region, in qualitative agreement with the expected post-interaction thermodynamic state. No spurious oscillations, artificial symmetry-breaking, or unphysical structures are observed, demonstrating that the multi-scale space--time sampling strategy effectively interrogates the wave interaction region at the resolutions required to capture its fine structure. A quantitative comparison against the high-resolution Godunov finite-volume reference is provided in Table~\ref{tab:2D_we_euler}.

\begin{table}[H]
\centering
\caption{Relative $L^1$, $L^2$, and $L^\infty$ errors for the two-dimensional compressible Euler equations (quadrant Riemann problem) at $t = 0.5$. Errors are computed over the full spatial domain $[-1,1]^2$ against the high-resolution Godunov finite-volume reference ($N_x \times N_y = 512 \times 512$).}
\label{tab:2D_we_euler}
\renewcommand{\arraystretch}{1.3}
\begin{tabular}{lccc}
\hline
Variable & $E_{L^1}$ & $E_{L^2}$ & $E_{L^\infty}$ \\
\hline
$\rho$ & $1.3960\times 10^{-2}$ & $4.0780\times 10^{-2}$ & $2.3871\times 10^{-1}$ \\
$u$    & $1.5380\times 10^{-2}$ & $2.2490\times 10^{-2}$ & $7.4037\times 10^{-1}$ \\
$v$    & $2.7190\times 10^{-2}$ & $1.6870\times 10^{-2}$ & $7.2690\times 10^{-1}$ \\
$p$    & $5.3400\times 10^{-3}$ & $8.8300\times 10^{-3}$ & $5.5510\times 10^{-2}$ \\
\hline
\end{tabular}
\end{table}
\begin{figure}[H]
    \centering
    \includegraphics[width=0.9\textwidth]{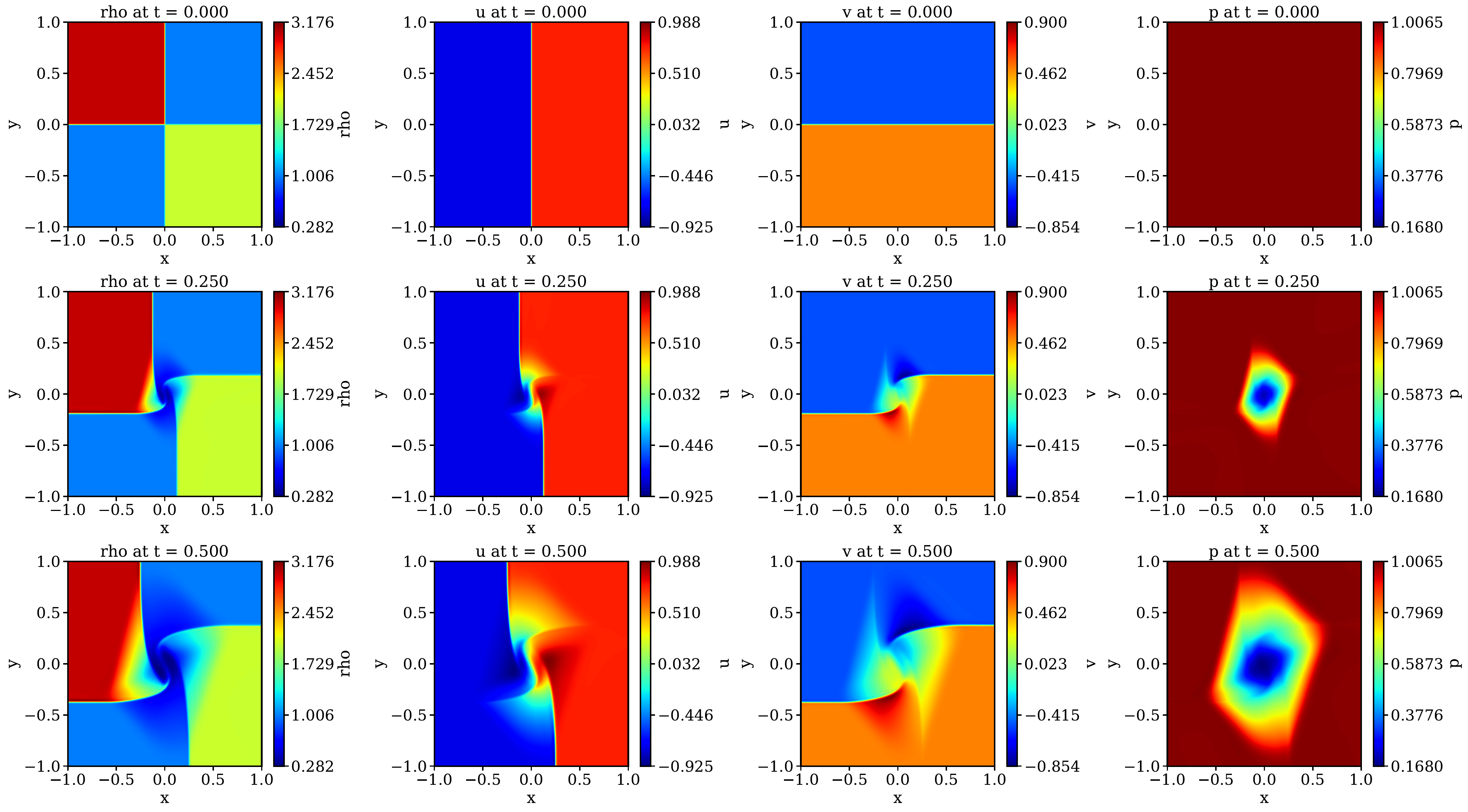}
    \caption{WE-PINNs solution of the two-dimensional Euler equations for the quadrant Riemann problem~\eqref{eq:quadrant_ic}. Shown are the density $\rho$, velocity components $u$ and $v$, and pressure $p$ over the domain $[-1,1]^2$ at selected times. The four-way wave interaction pattern, comprising contact discontinuities and compressive shocks propagating from each quadrant boundary, is faithfully captured without artificial symmetry breaking or spurious oscillations.}
    \label{fig:Euler_quadrant}
\end{figure}

\subsubsection{Radially Symmetric Riemann Problem}
\label{sec:euler2d_radial}

The second two-dimensional configuration is the radially symmetric Riemann problem, in which the initial discontinuity is a circular interface defined by
\begin{equation}
    \phi(x,y) = \sqrt{x^2 + y^2} - R, \qquad R = 0.13,
    \label{eq:radial_interface}
\end{equation}
with the initial condition
\begin{equation}
    (\rho, u, v, p)(x,y,0) =
    \begin{cases}
        (2.0,\; 0,\; 0,\; 15.0), & \phi(x,y) < 0, \\
        (1.0,\; 0,\; 0,\;\;\, 1.0), & \phi(x,y) \geq 0.
    \end{cases}
    \label{eq:radial_ic}
\end{equation}
The large pressure ratio of $15:1$ across the circular interface drives a strong outward-propagating shock wave and an inward-propagating rarefaction fan. Both waves preserve exact radial symmetry in the continuous solution, making this test particularly sensitive to anisotropic discretization errors. Mesh-based methods introduce directional biases that degrade the circular geometry of the shock front, especially on Cartesian grids. The mesh-free formulation of WE-PINNs, which samples space--time control volumes without reference to any underlying spatial lattice, is inherently free of such anisotropic artifacts and is therefore well-suited to this class of geometrically sensitive problems.

The WE-PINNs solution is presented in Fig.~\ref{fig:Euler_circular} at three successive times. The circular shock front expands outward at the correct radial speed and maintains its circular geometry throughout the simulation, with no discernible symmetry-breaking artifacts. The density and pressure fields display the expected annular structure: a sharp outer shock bounding a smooth intermediate region, transitioning to an approximately uniform high-density, low-pressure core. The velocity components $u$ and $v$ exhibit equal magnitudes and the correct antisymmetric spatial distribution consistent with purely radial outflow, confirming that the momentum equations are enforced consistently in both spatial directions. The total energy is correctly transported with the expanding wave structure, and no spurious reflections are observed at the domain boundaries. The corresponding relative errors at $t = 0.3$ are reported in Table~\ref{tab:euler2d_errors}.

\begin{table}[H]
\centering
\caption{Relative $L^1$, $L^2$, and $L^\infty$ errors for the two-dimensional compressible Euler equations (radially symmetric Riemann problem) at $t = 0.3$. Errors are computed over the full domain $[-1,1]^2$ against the high-resolution Godunov finite-volume reference ($N_x \times N_y = 512 \times 512$).}
\label{tab:euler2d_errors}
\renewcommand{\arraystretch}{1.3}
\begin{tabular}{lccc}
\hline
Variable & $E_{L^1}$ & $E_{L^2}$ & $E_{L^\infty}$ \\
\hline
$\rho$ & $2.0831\times 10^{-2}$ & $1.2033\times 10^{-2}$ & $3.5371\times 10^{-1}$ \\
$u$    & $3.1201\times 10^{-2}$ & $5.2808\times 10^{-2}$ & $3.8485\times 10^{-1}$ \\
$v$    & $6.2251\times 10^{-2}$ & $8.3637\times 10^{-2}$ & $4.6726\times 10^{-1}$ \\
$p$    & $1.2964\times 10^{-2}$ & $7.3641\times 10^{-2}$ & $3.8145\times 10^{-1}$ \\
\hline
\end{tabular}
\end{table}

\begin{figure}[H]
    \centering
    \includegraphics[width=0.9\textwidth]{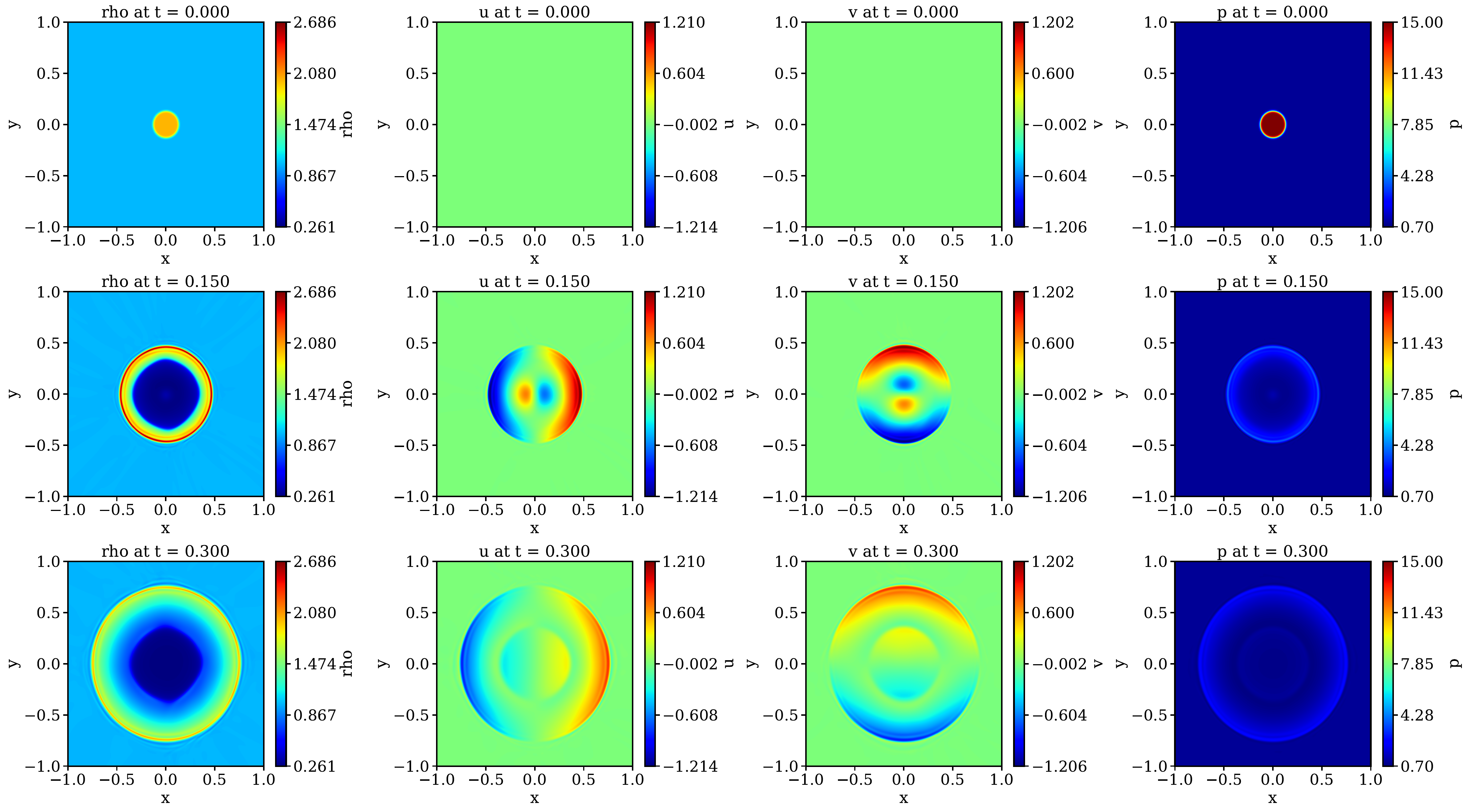}
    \caption{WE-PINNs solution of the two-dimensional Euler equations for the radially symmetric Riemann problem~\eqref{eq:radial_ic} with interface radius $R = 0.13$ and pressure ratio $15:1$. Shown are the density $\rho$, velocity components $u$ and $v$, and pressure $p$ over the domain $[-1,1]^2$ at three successive times. The circular shock front expands outward at the correct radial speed while preserving exact radial symmetry throughout the simulation, with no anisotropic artifacts introduced by the mesh-free formulation.}
    \label{fig:Euler_circular}
\end{figure}

The two-dimensional Euler experiments collectively demonstrate that the WE-PINNs framework extends naturally to multi-dimensional systems of conservation laws without any modification to the core methodology. The space--time control-volume formulation applies directly in two spatial dimensions through the tensor-product Gauss--Legendre quadrature described in Section~\ref{sec:cv_formulation}, and no special treatment is required for curved or non-axis-aligned discontinuities. The accurate resolution of the radially symmetric configuration, in particular, underscores the structural advantage of the mesh-free approach in problems where the solution geometry does not align with any preferred Cartesian direction. Taken together with the one-dimensional Sod shock tube results, these experiments confirm that WE-PINNs constitutes a robust and versatile framework for the numerical approximation of entropy solutions to nonlinear hyperbolic systems across multiple spatial dimensions.

\section{Conclusion}\label{sec:conclusion}

We have presented WE-PINNs, a mesh-free space--time control-volume framework for the approximation of entropy solutions to nonlinear hyperbolic conservation laws. The method replaces the classical strong-form residual minimization with a weak formulation derived from the divergence theorem, enforcing conservation through boundary flux integrals over dynamically sampled space--time control volumes and incorporating entropy admissibility in integral form via the full Kru\v{z}kov family of one-sided penalties. The resulting loss functional is free of spatial derivatives of the neural approximation, introduces no artificial viscosity, relies on no fixed discretization mesh, and remains well-defined in the presence of shock discontinuities. A single feedforward network suffices, with no auxiliary potential networks, no test-function networks, and no saddle-point optimization.

 The contributions of this work are threefold. First, we formulated WE-PINNs as a mesh-free weak space--time control-volume method that simultaneously enforces integral conservation, integral Kru\v{z}kov entropy conditions for the full one-parameter family, and total-variation control, thereby addressing the structural limitations identified in existing weak, integral, and finite-volume-informed PINNs formulations. Second, we established a conditional $L^1$ convergence estimate via the Bouchut--Perthame framework for scalar conservation laws: the $L^1$ error towards the entropy solution is controlled by the entropy loss and an explicit sampling-consistency residual, as $\|U_\theta(\cdot,T)-u(\cdot,T)\|_{L^1} \leq K_{\mathrm{opt}}(\mathcal{L}_{\mathrm{ent}}^{1/4} + (R_{\mathrm{samp}}^{\mathrm{ent}})^{1/2})$, while the weak conservation defect is controlled separately by the conservation loss. To the best of our knowledge, this is the first such conditional $L^1$ estimate for a fully mesh-free space--time control-volume PINNs formulation with integral Kru\v{z}kov entropy enforcement. Third, we provided comprehensive empirical validation across scalar equations with convex and non-convex fluxes and hyperbolic systems in one and two spatial dimensions.

Numerical experiments on the inviscid Burgers equation, the Buckley--Leverett equation and the compressible Euler equations demonstrate that WE-PINNs accurately resolves shock waves, rarefaction fans, contact discontinuities, and compound wave structures in both scalar and system settings. The stochastic Kru\v{z}kov sampling strategy enforces the Oleinik entropy condition for the non-convex Buckley--Leverett flux and selects the physically admissible solution in cases where methods based on weaker entropy information can fail. For the Euler system, the framework extends naturally to one- and two-dimensional configurations including oblique and radially symmetric discontinuities without any modification to the core methodology, demonstrating the geometric flexibility of the mesh-free formulation. Across the reported tests, the method remains stable in both smooth and shock-dominated regimes without artificial viscosity or explicit shock tracking, and its $L^1$ errors are typically several times to one order of magnitude below those of existing PINNs-based methods.

These results open several directions for future research. Extending the $L^1$ convergence analysis to systems of conservation laws remains the most important theoretical open problem, as the Bouchut--Perthame framework currently applies only to the scalar case. On the algorithmic side, adaptive strategies for dynamically selecting the loss weights $\lambda_1$, $\lambda_2$, and $\lambda_{\mathrm{TVD}}$ during training, higher-order quadrature schemes for non-rectangular control volumes, and residual-based adaptive sampling of the entropy index $\kappa$ in the Kru\v{z}kov family are natural extensions that could further improve accuracy and computational efficiency. Finally, extending this mesh-free framework to equations incorporating diffusion terms, such as the Navier--Stokes equations, and generalizing the formulation to handle more complex geometries of real-world engineering applications \cite{OUBARKA202658}, represent promising directions for future research.

\section*{Acknowledgments}
The authors gratefully acknowledge the support and computing resources provided by the Toubkal Supercomputer~\cite{kissami2025toubkal} (https://toubkal.um6p.ma) at UM6P, Morocco, where all computational experiments were carried out on NVIDIA A100 GPUs.
\section*{Code Availability}
The source code is publicly available at \href{https://github.com/oubarka-ismail/WE-PINNs}{https://github.com/oubarka-ismail/WE-PINNs}
\bibliographystyle{elsarticle-num} 
\bibliography{bibloigraphy}

\appendix

\section{Entropy Pairs for the Considered Models}
\label{app:entropies}

This appendix collects the convex entropy flux pairs $(\eta, q)$ employed throughout the paper for each of the three conservation laws under consideration. In each case, the entropy function $\eta$ is convex with respect to the conservative variables, and the entropy flux $q$ satisfies the compatibility condition
\begin{equation}
    \nabla_U q(U) = \nabla_U \eta(U)\, \nabla_U F(U),
    \label{eq:entropy_compatibility}
\end{equation}
where $F(U)$ denotes the physical flux of the corresponding system. This ensures that smooth solutions satisfy the entropy equality, whereas weak solutions across shocks satisfy the entropy inequality
\begin{equation}
    \partial_t\, \eta(U) + \nabla_{\mathbf{x}} \cdot q(U) \leq 0
    \label{eq:entropy_ineq}
\end{equation}
in the sense of distributions, consistent with the Lax entropy admissibility condition~\cite{Lax1973}. The integral form of this inequality, obtained by integrating over a space--time control volume $D$ and applying the divergence theorem, is the entropy loss $\mathcal{L}_{\mathrm{ent}}$ defined in Section~\ref{sec:ent_loss}.

\subsection{Burgers Equation}
\label{app:entropy_burgers}

For the scalar Burgers equation
\begin{equation}
    \partial_t\, u + \partial_x\!\left(\frac{u^2}{2}\right) = 0,
\end{equation}
we employ the Kru\v{z}kov entropy family~\cite{Kruzkov1970}. For any constant $\kappa \in \mathbb{R}$, the entropy--flux pair is defined by
\begin{equation}
    \eta(u;\, \kappa) = |u - \kappa|,
    \qquad
    q(u;\, \kappa) = \mathrm{sgn}(u - \kappa)
                \left(\frac{u^2}{2} - \frac{\kappa^2}{2}\right),
    \label{eq:kruzkov_burgers}
\end{equation}
where $\mathrm{sgn}$ denotes the sign function. A weak solution $u$ is the entropy solution in the sense of Kru\v{z}kov~\cite{Kruzkov1970} if and only if
\begin{equation}
    \partial_t\, |u - \kappa| \;+\;
    \partial_x\!\left[\,\mathrm{sgn}(u - \kappa)
    \left(\frac{u^2}{2} - \frac{\kappa^2}{2}\right)\right] \leq 0
    \label{eq:kruzkov_burgers_ineq}
\end{equation}
holds in the sense of distributions for all $\kappa \in \mathbb{R}$.

\subsection{Buckley--Leverett Equation}
\label{app:entropy_bl}

For the Buckley--Leverett equation
\begin{equation}
    \partial_t\, u + \partial_x\, f(u) = 0,
\end{equation}
we likewise employ the Kru\v{z}kov entropy family~\cite{Kruzkov1970}. The flux function $f$ is given by
\begin{equation}
    f(u) = \frac{u^2}{u^2 + a\,(1 - u)^2}, \qquad u \in [0,\,1],
    \label{eq:bl_flux_app}
\end{equation}
where $a > 0$ is the viscosity ratio parameter and $u$ denotes the water saturation. For any constant $\kappa \in \mathbb{R}$, the
entropy--flux pair reads
\begin{equation}
    \eta(u;\, \kappa) = |u - \kappa|,
    \qquad
    q(u;\, \kappa) = \mathrm{sgn}(u - \kappa)\!\left(f(u) - f(\kappa)\right),
    \label{eq:kruzkov_bl}
\end{equation}
and a weak solution $u$ is the entropy solution in the sense of Kru\v{z}kov~\cite{Kruzkov1970} if and only if
\begin{equation}
    \partial_t\,|u - \kappa| \;+\;
    \partial_x\!\left[\,\mathrm{sgn}(u - \kappa)
    \left(f(u) - f(\kappa)\right)\right] \leq 0
    \label{eq:kruzkov_bl_ineq}
\end{equation}
holds in the sense of distributions for all $\kappa \in \mathbb{R}$. This family is necessary to enforce the Oleinik entropy condition for the non-convex flux $f$ and select the physically admissible compound wave solution, as discussed in Section~\ref{sec:BL}.

\subsection{Compressible Euler Equations}
\label{app:entropy_euler}

For the compressible Euler equations
\begin{equation}
    \partial_t\, U + \nabla_{\mathbf{x}} \cdot F(U) = 0,
    \qquad
    U = \begin{pmatrix} \rho \\ \rho \mathbf{u} \\ E \end{pmatrix},
    \qquad
    F(U) = \begin{pmatrix} \rho \mathbf{u} \\
                           \rho \mathbf{u} \otimes \mathbf{u}
                           + p\,\mathbf{I} \\
                           (E + p)\,\mathbf{u}
            \end{pmatrix},
    \label{eq:euler_app}
\end{equation}
the mathematical entropy flux pair is given by~\cite{godlewski2013numerical}
\begin{equation}
    \eta(U) = -\rho\, s,
    \qquad
    \mathbf{q}(U) = -\rho\, \mathbf{u}\, s,
    \label{eq:euler_entropy}
\end{equation}
where $\rho$ denotes the density, $\mathbf{u} \in \mathbb{R}^d$ the velocity vector, $E$ the total energy per unit volume, $\mathbf{I}$ the identity tensor, and $s$ the specific thermodynamic entropy. For a polytropic ideal gas, $s$ is given by
\begin{equation}
    s = \ln\!\left(\frac{p}{\rho^{\gamma}}\right),
    \label{eq:specific_entropy}
\end{equation}
where $p$ is the thermodynamic pressure and $\gamma > 1$ is the ratio of specific heats. The pressure is recovered from the total energy via the equation of state
\begin{equation}
    p = (\gamma - 1)\!\left(E - \frac{1}{2}\,\rho\,
    |\mathbf{u}|^2\right).
    \label{eq:eos}
\end{equation}
The function $\eta$ defined in~\eqref{eq:euler_entropy} is convex with respect to the conservative variables $U = (\rho,\, \rho \mathbf{u},\, E)^{\top}$.

\section{Baseline configurations and the fair-comparison protocol}
\label{app:baselines}

All baseline methods compared against WE-PINNs in Section~\ref{sec:results}, the Coupled Integral PINNs (CI-PINNs)~\cite{CIPINN2024}, the control-volume PINNs (cvPINNs)~\cite{Patel2022}, and the standard PINNs, are re-implemented within the same training harness as WE-PINNs and run under a matched comparison protocol: the same network architecture, optimizer, training budget, spatial domain, and initial/boundary data are used whenever the method allows it. The loss formulation and its associated method-specific choices, such as loss weights, quadrature settings, control-volume counts, or collocation counts, are kept explicit and reported below. This choice is deliberate. It reduces differences due to network size, optimizer schedule, or training budget while preserving the defining loss structure of each baseline. A direct consequence is that our re-implementations do \emph{not} aim to reproduce the exact hyperparameters of the original references: each baseline retains its loss formulation faithfully while adopting the matched protocol where applicable, so the numerical values reported here are not expected to match, and should not be compared with, the values reported in the original papers, which were obtained under different architectures and training budgets.

For transparency, Tables~\ref{tab:cfg_cv}--\ref{tab:cfg_ci} document, side by side, the configuration used in this work and the settings reported in the original papers. The upper rows identify the architecture and optimization choices matched across methods; the lower rows record the method-specific numerical and loss choices used by each baseline.

\begin{table}[htbp]
\centering
\caption{cvPINNs: configuration used in this work under the matched comparison protocol versus the original formulation of Patel et al.~\cite{Patel2022} (Buckley--Leverett case, their Table~A.3). The upper block lists the matched architecture/optimization choices; the lower block lists the cvPINNs-specific numerics.}
\label{tab:cfg_cv}
\begin{tabular}{@{}l p{6.2cm} p{5.8cm}@{}}
\toprule
Item & This work (matched protocol) & Patel et al.~\cite{Patel2022} \\
\midrule
Network             & MLP $8\times32$, $\tanh$                              & MLP $8\times64$, ReLU \\
Initialization      & PyTorch default                                      & Glorot \\
Optimizer           & Adam $5\times10^{-3}$ (cosine $\to10^{-4}$) + L-BFGS  & Adam $10^{-4}$ \\
Training budget     & $2\times10^{4}$ Adam iterations (+ L-BFGS)            & $10^{6}$ iterations \\
\midrule
Control volumes     & $n_x\times n_t = 50\times50$                         & $200\times200$ cells \\
Face quadrature     & Gauss--Legendre, $q=8$                              & midpoint, $1$ segment \\
Residual scaling    & per-cell $/\,\mathrm{area}$                          & none (raw RMS) \\
Loss weighting      & adaptive (running-mean normalizer $+$ augmented Lagrangian) & fixed $\varepsilon_{\mathrm{ent}}=\varepsilon_{\mathrm{TVD}}=1$ \\
Entropy             & single pair, one-sided $\mathrm{relu}$               & single pair $\eta=u^2/2$ \\
TVD control         & $\mathrm{relu}(TV_{j+1}-TV_j)^2$                     & cell-based TV penalty \\
Artificial viscosity& off                                                  & off (Fig.~3; available, $\propto\Delta x^2$) \\
Initial condition   & weak (bottom control-volume face $=u_0$)             & weak (control-volume face) \\
\bottomrule
\end{tabular}
\end{table}

\begin{table}[htbp]
\centering
\caption{CI-PINNs: configuration used in this work under the matched comparison protocol versus Wang \& Yang~\cite{CIPINN2024}. The dual-network loss structure (Eq.~5 of~\cite{CIPINN2024}) is reproduced faithfully; the architecture and optimizer are matched with the other methods.}
\label{tab:cfg_ci}
\begin{tabular}{@{}l p{6.2cm} p{5.8cm}@{}}
\toprule
Item & This work (matched protocol) & Wang \& Yang~\cite{CIPINN2024} \\
\midrule
Architecture        & dual MLP (primitive $\tilde u$ $+$ potential $\tilde S$), each $8\times32$, $\tanh$ & dual fully-connected MLPs \\
Optimizer / budget  & Adam $5\times10^{-3}$ $+$ L-BFGS, $2\times10^{4}$ iters & not tabulated in~\cite{CIPINN2024} \\
Loss                & Eq.~(5): five terms (IBC, physical, coupling, entropy, adaptive-strong) & Eq.~(5) (identical) \\
Collocation         & mesh-free random, $n_f=5000$                         & mesh-free pointwise (no quadrature) \\
Loss weighting      & fixed, tuned ($\lambda_{\mathrm{ibc}}=10$, $\lambda_{\mathrm{phy}}=\lambda_{\mathrm{cpl}}=30$, $\lambda_{\mathrm{ent}}=\lambda_{\mathrm{str}}=1$) & fixed weights \\
Entropy term        & $\mathrm{relu}(u\,R)^2$, $\eta=u^2/2$, $R=u_t+f_x$   & $\max(0,\partial_t\eta+\nabla\!\cdot\!\phi)^2$, $\eta$ general \\
Adaptive-strong mask& $w=\sigma\!\left(10\,\mathrm{relu}(-u_x)\right)$      & $w=\sigma\!\left(k\,\mathrm{ReLU}(-\nabla\!\cdot\!\tilde v)\right)$ \\
Initial condition   & soft (IBC penalty)                                   & soft (IBC penalty) \\
\bottomrule
\end{tabular}
\end{table}

For CI-PINNs we additionally tuned the conservation-term weights ($\lambda_{\mathrm{phy}}=\lambda_{\mathrm{cpl}}=30$ rather than unit weights): with unit weights the coupled potential under-weights the Rankine--Hugoniot balance and the shock speed is systematically under-estimated. The baseline is therefore reported at its most favorable setting, which only strengthens the fairness of the comparison. The remaining method-specific knobs (the Gauss--Legendre order $q$ and the control-volume count for cvPINNs, the number of collocation points $n_f$ for CI-PINNs) are kept fixed across all test cases.

\section{Parametric Study: Influence of Quadrature Order and Network
Architecture on the Burgers Equation}
\label{app:parametric}

This appendix provides the complete numerical data underlying the parametric study summarized in Section~\ref{subsec:burger_convergence}. This study is separate from the fixed-configuration benchmark comparisons: it deliberately varies quadrature order, network depth, and network width to document the robustness of the loss--error trend and the choice of the default configuration used in the main Burgers tests. Table~\ref{tab:parametric_study} reports, for each combination of quadrature order $Q \in \{3, 5, 7, 9, 11\}$, network depth $L \in \{4, 8, 16\}$ hidden layers, and width $N \in \{16, 32, 64\}$ neurons per layer, the total loss $\mathcal{L}$ at the end of training, the $L^1$ error at $t = T$, and the wall-clock training and inference times. All experiments are performed on the Burgers equation with the sinusoidal initial condition of Section~\ref{subsec:burger_sin}, on a structured Cartesian grid of control volumes. The two optimization failures identified in Section~\ref{subsec:burger_convergence}, namely the configurations $(Q=9,\, L=16,\, N=64)$ and $(Q=11,\, L=16,\, N=64)$, are clearly visible as outliers with total loss and $L^1$ error well above the observed reference trend.

\begin{table}[htbp]
\centering
\caption{Parametric study results for the Burgers equation on a structured Cartesian grid. The influence of quadrature order $Q$, network depth (Layers), and network width (Neurons) on the total loss $\mathcal{L}$, $L^1$ error at $t=T$, training time $T_{\mathrm{train}}$ (s), and inference time $T_{\mathrm{infer}}$ (s) is reported. Configurations exhibiting optimization failure are identifiable as outliers with total loss of order $10^{-1}$ or larger.}
\label{tab:parametric_study}
\scriptsize
\renewcommand{\arraystretch}{1.1}
\setlength{\tabcolsep}{4pt}
\makebox[\textwidth][c]{%
\resizebox{\textwidth}{!}{%
\begin{tabular}{cc cccc cccc cccc}
\toprule
& &
\multicolumn{4}{c}{\textbf{Layers = 4}} &
\multicolumn{4}{c}{\textbf{Layers = 8}} &
\multicolumn{4}{c}{\textbf{Layers = 16}} \\
\cmidrule(lr){3-6}\cmidrule(lr){7-10}\cmidrule(lr){11-14}
\textbf{$Q$} & \textbf{Neurons} &
$\mathcal{L}$ & $L^1$ & $T_{\mathrm{train}}(s)$ & $T_{\mathrm{infer}}(s)$ &
$\mathcal{L}$ & $L^1$ & $T_{\mathrm{train}}(s)$ & $T_{\mathrm{infer}}(s)$ &
$\mathcal{L}$ & $L^1$ & $T_{\mathrm{train}}(s)$ & $T_{\mathrm{infer}}(s)$ \\
\midrule
\multirow{3}{*}{3}
& 16 & $2.19\times10^{-2}$ & $1.61\times10^{-2}$ & 367.2  & 0.0005
     & $1.23\times10^{-3}$ & $9.60\times10^{-3}$ & 481.5  & 0.0006
     & $2.47\times10^{-2}$ & $5.04\times10^{-2}$ & 772.9  & 0.0009 \\
& 32 & $5.51\times10^{-4}$ & $5.65\times10^{-3}$ & 370.5  & 0.0005
     & $9.48\times10^{-3}$ & $5.76\times10^{-2}$ & 482.5  & 0.0006
     & $2.53\times10^{-2}$ & $5.95\times10^{-2}$ & 770.5  & 0.0012 \\
& 64 & $2.45\times10^{-2}$ & $2.56\times10^{-2}$ & 362.6  & 0.0005
     & $7.94\times10^{-4}$ & $1.01\times10^{-2}$ & 478.3  & 0.0006
     & $5.63\times10^{-4}$ & $6.76\times10^{-3}$ & 803.0  & 0.0010 \\
\midrule
\multirow{3}{*}{5}
& 16 & $1.11\times10^{-3}$ & $5.84\times10^{-3}$ & 378.2  & 0.0006
     & $8.16\times10^{-3}$ & $4.43\times10^{-2}$ & 475.0  & 0.0006
     & $4.57\times10^{-3}$ & $4.51\times10^{-2}$ & 769.6  & 0.0011 \\
& 32 & $1.79\times10^{-4}$ & $5.57\times10^{-3}$ & 372.5  & 0.0005
     & $1.24\times10^{-2}$ & $5.53\times10^{-2}$ & 474.5  & 0.0007
     & $1.53\times10^{-2}$ & $2.98\times10^{-2}$ & 768.3  & 0.0011 \\
& 64 & $2.72\times10^{-3}$ & $4.13\times10^{-2}$ & 417.8  & 0.0005
     & $6.18\times10^{-4}$ & $5.47\times10^{-3}$ & 585.6  & 0.0006
     & $2.78\times10^{-2}$ & $3.37\times10^{-2}$ & 1012.3 & 0.0011 \\
\midrule
\multirow{3}{*}{7}
& 16 & $1.21\times10^{-4}$ & $3.32\times10^{-3}$ & 371.9  & 0.0005
     & $2.76\times10^{-4}$ & $6.78\times10^{-3}$ & 477.5  & 0.0006
     & $3.12\times10^{-4}$ & $4.76\times10^{-3}$ & 770.3  & 0.0009 \\
& 32 & $5.13\times10^{-4}$ & $7.08\times10^{-3}$ & 375.6  & 0.0005
     & $4.91\times10^{-5}$ & $3.92\times10^{-3}$ & 497.9  & 0.0006
     & $5.48\times10^{-5}$ & $5.13\times10^{-3}$ & 837.7  & 0.0009 \\
& 64 & $1.84\times10^{-4}$ & $4.72\times10^{-3}$ & 466.8  & 0.0005
     & $7.82\times10^{-4}$ & $3.80\times10^{-3}$ & 753.0  & 0.0006
     & $8.49\times10^{-4}$ & $5.43\times10^{-3}$ & 1282.5 & 0.0011 \\
\midrule
\multirow{3}{*}{9}
& 16 & $8.91\times10^{-5}$ & $6.25\times10^{-3}$ & 370.7  & 0.0005
     & $8.14\times10^{-5}$ & $2.32\times10^{-3}$ & 479.0  & 0.0006
     & $5.52\times10^{-4}$ & $5.25\times10^{-3}$ & 770.0  & 0.0011 \\
& 32 & $1.69\times10^{-4}$ & $2.22\times10^{-3}$ & 399.8  & 0.0005
     & $6.56\times10^{-5}$ & $4.67\times10^{-3}$ & 543.1  & 0.0006
     & $7.60\times10^{-5}$ & $1.30\times10^{-3}$ & 941.6  & 0.0011 \\
& 64 & $3.00\times10^{-3}$ & $3.97\times10^{-2}$ & 526.2  & 0.0005
     & $2.97\times10^{-3}$ & $9.50\times10^{-3}$ & 884.1  & 0.0006
     & $1.78\times10^{1}$  & $9.95\times10^{-1}$ & 1563.2 & 0.0009 \\
\midrule
\multirow{3}{*}{11}
& 16 & $4.91\times10^{-5}$ & $5.07\times10^{-3}$ & 369.7  & 0.0005
     & $1.10\times10^{-4}$ & $2.51\times10^{-3}$ & 476.5  & 0.0006
     & $5.54\times10^{-5}$ & $2.04\times10^{-3}$ & 763.4  & 0.0009 \\
& 32 & $1.38\times10^{-4}$ & $2.56\times10^{-3}$ & 426.3  & 0.0005
     & $2.63\times10^{-5}$ & $4.37\times10^{-3}$ & 611.4  & 0.0006
     & $5.47\times10^{-5}$ & $1.74\times10^{-3}$ & 1008.3 & 0.0010 \\
& 64 & $3.75\times10^{-5}$ & $4.72\times10^{-3}$ & 622.3  & 0.0006
     & $1.61\times10^{-3}$ & $2.58\times10^{-2}$ & 1015.7 & 0.0007
     & $4.29\times10^{-1}$ & $1.26\times10^{-1}$ & 1847.8 & 0.0011 \\
\bottomrule
\end{tabular}}}
\end{table}

\section{Reproducibility Study: Multi-Run Stability of WE-PINNs
on the Sod Shock Tube Problem}
\label{app:multirun}

This appendix assesses the reproducibility and stability of WE-PINNs on the Sod shock tube problem by repeating the training independently $10$ times under two distinct initialization strategies: a constant initialization, where all runs use the same fixed network parameters and the same random seeds for all stochastic sampling steps, and a random initialization, where each run uses a different random seed for the network weights and sampling sequence. The per-run relative errors in the $L^1$, $L^2$, and $L^\infty$ norms are reported for all variables $\rho$, $u$, $p$, and $E$, together with the mean and standard deviation across the $10$ runs.

Under constant initialization (Table~\ref{tab:multirun_constant}), all $10$ runs produce identical results to machine precision, with standard deviations of order $10^{-17}$ or smaller. In this setting, both the initial network parameters and the random seeds controlling the sampled control volumes, entropy indices, and TVD points are fixed across runs, so the same stochastic sequence is replayed during training. The near-zero standard deviations therefore confirm determinism under fully controlled randomness, up to floating-point roundoff, and validate the reproducibility of the method under controlled conditions.

Under random initialization (Table~\ref{tab:multirun_random}), the $10$ runs produce slightly different results, with standard deviations of the order of $10^{-5}$ in $L^1$ for $\rho$, $p$ and $E$, and $10^{-4}$ in $L^1$ for $u$. These variations are small relative to the mean errors and reflect the inherent sensitivity of neural network training to the random initial weights, which is common to all PINNs-based methods. Importantly, the mean errors under random initialization are consistent with those reported in Table~\ref{tab:euler_errors}, confirming that the results are not specific to a particularly favorable initialization. The largest variability is observed for the energy field $E$ in the $L^\infty$ norm (standard deviation $\approx 3.8\times10^{-3}$), which reflects the sensitivity of the pointwise error near the shock and contact discontinuity. Taken together, these results demonstrate that WE-PINNs is robust to initialization and produces consistently accurate solutions across all runs.

\begin{table}[htbp]
\centering
\caption{Reproducibility study under \textbf{constant initialization}: relative $L^1$, $L^2$, and $L^\infty$ errors for all four variables of the Sod shock tube problem, averaged over the three evaluation times $t=0,0.2,0.4$, over $10$ independent training runs with fixed network initialization and fixed sampling seeds. Average and standard deviation (Std) are reported. The vanishing standard deviations confirm full determinism under controlled randomness.}

\label{tab:multirun_constant}
\renewcommand{\arraystretch}{1.1}
\makebox[\textwidth][c]{%
\resizebox{\textwidth}{!}{%
\begin{tabular}{l ccc ccc ccc ccc}
\toprule
\multirow{2}{*}{\textbf{Run}} &
\multicolumn{3}{c}{\textbf{Density $\rho$}} &
\multicolumn{3}{c}{\textbf{Velocity $u$}} &
\multicolumn{3}{c}{\textbf{Pressure $p$}} &
\multicolumn{3}{c}{\textbf{Energy $E$}} \\
\cmidrule(lr){2-4} \cmidrule(lr){5-7} \cmidrule(lr){8-10} \cmidrule(lr){11-13}
& $E_{L^1}$ & $E_{L^2}$ & $E_{L^\infty}$
& $E_{L^1}$ & $E_{L^2}$ & $E_{L^\infty}$
& $E_{L^1}$ & $E_{L^2}$ & $E_{L^\infty}$
& $E_{L^1}$ & $E_{L^2}$ & $E_{L^\infty}$ \\
\midrule
Run  1 & $6.453\times10^{-4}$ & $3.360\times10^{-3}$ & $4.993\times10^{-2}$ & $4.455\times10^{-3}$ & $2.004\times10^{-2}$ & $4.563\times10^{-2}$ & $4.558\times10^{-4}$ & $2.827\times10^{-3}$ & $5.730\times10^{-2}$ & $4.423\times10^{-4}$ & $2.984\times10^{-3}$ & $6.225\times10^{-2}$ \\
Run  2 & $6.453\times10^{-4}$ & $3.360\times10^{-3}$ & $4.993\times10^{-2}$ & $4.455\times10^{-3}$ & $2.004\times10^{-2}$ & $4.563\times10^{-2}$ & $4.558\times10^{-4}$ & $2.827\times10^{-3}$ & $5.730\times10^{-2}$ & $4.423\times10^{-4}$ & $2.984\times10^{-3}$ & $6.225\times10^{-2}$ \\
Run  3 & $6.453\times10^{-4}$ & $3.360\times10^{-3}$ & $4.993\times10^{-2}$ & $4.455\times10^{-3}$ & $2.004\times10^{-2}$ & $4.563\times10^{-2}$ & $4.558\times10^{-4}$ & $2.827\times10^{-3}$ & $5.730\times10^{-2}$ & $4.423\times10^{-4}$ & $2.984\times10^{-3}$ & $6.225\times10^{-2}$ \\
Run  4 & $6.453\times10^{-4}$ & $3.360\times10^{-3}$ & $4.993\times10^{-2}$ & $4.455\times10^{-3}$ & $2.004\times10^{-2}$ & $4.563\times10^{-2}$ & $4.558\times10^{-4}$ & $2.827\times10^{-3}$ & $5.730\times10^{-2}$ & $4.423\times10^{-4}$ & $2.984\times10^{-3}$ & $6.225\times10^{-2}$ \\
Run  5 & $6.453\times10^{-4}$ & $3.360\times10^{-3}$ & $4.993\times10^{-2}$ & $4.455\times10^{-3}$ & $2.004\times10^{-2}$ & $4.563\times10^{-2}$ & $4.558\times10^{-4}$ & $2.827\times10^{-3}$ & $5.730\times10^{-2}$ & $4.423\times10^{-4}$ & $2.984\times10^{-3}$ & $6.225\times10^{-2}$ \\
Run  6 & $6.453\times10^{-4}$ & $3.360\times10^{-3}$ & $4.993\times10^{-2}$ & $4.455\times10^{-3}$ & $2.004\times10^{-2}$ & $4.563\times10^{-2}$ & $4.558\times10^{-4}$ & $2.827\times10^{-3}$ & $5.730\times10^{-2}$ & $4.423\times10^{-4}$ & $2.984\times10^{-3}$ & $6.225\times10^{-2}$ \\
Run  7 & $6.453\times10^{-4}$ & $3.360\times10^{-3}$ & $4.993\times10^{-2}$ & $4.455\times10^{-3}$ & $2.004\times10^{-2}$ & $4.563\times10^{-2}$ & $4.558\times10^{-4}$ & $2.827\times10^{-3}$ & $5.730\times10^{-2}$ & $4.423\times10^{-4}$ & $2.984\times10^{-3}$ & $6.225\times10^{-2}$ \\
Run  8 & $6.453\times10^{-4}$ & $3.360\times10^{-3}$ & $4.993\times10^{-2}$ & $4.455\times10^{-3}$ & $2.004\times10^{-2}$ & $4.563\times10^{-2}$ & $4.558\times10^{-4}$ & $2.827\times10^{-3}$ & $5.730\times10^{-2}$ & $4.423\times10^{-4}$ & $2.984\times10^{-3}$ & $6.225\times10^{-2}$ \\
Run  9 & $6.453\times10^{-4}$ & $3.360\times10^{-3}$ & $4.993\times10^{-2}$ & $4.455\times10^{-3}$ & $2.004\times10^{-2}$ & $4.563\times10^{-2}$ & $4.558\times10^{-4}$ & $2.827\times10^{-3}$ & $5.730\times10^{-2}$ & $4.423\times10^{-4}$ & $2.984\times10^{-3}$ & $6.225\times10^{-2}$ \\
Run 10 & $6.453\times10^{-4}$ & $3.360\times10^{-3}$ & $4.993\times10^{-2}$ & $4.455\times10^{-3}$ & $2.004\times10^{-2}$ & $4.563\times10^{-2}$ & $4.558\times10^{-4}$ & $2.827\times10^{-3}$ & $5.730\times10^{-2}$ & $4.423\times10^{-4}$ & $2.984\times10^{-3}$ & $6.225\times10^{-2}$ \\
\midrule
Average   & $6.453\times10^{-4}$ & $3.360\times10^{-3}$ & $4.993\times10^{-2}$ & $4.455\times10^{-3}$ & $2.004\times10^{-2}$ & $4.563\times10^{-2}$ & $4.558\times10^{-4}$ & $2.827\times10^{-3}$ & $5.730\times10^{-2}$ & $4.423\times10^{-4}$ & $2.984\times10^{-3}$ & $6.225\times10^{-2}$ \\
\textbf{Std} &$\mathbf{4.571\times10^{-21}}$ & $\mathbf{1.829\times10^{-19}}$ & $\mathbf{1.463\times10^{-17}}$ & $\mathbf{0.000}$ & $\mathbf{0.000}$ & $\mathbf{0.000}$ & $\mathbf{2.286\times10^{-22}}$ & $\mathbf{1.829\times10^{-18}}$ & $\mathbf{0.000}$ & $\mathbf{2.286\times10^{-19}}$ & $\mathbf{0.000}$ & $\mathbf{0.000}$ \\
\bottomrule
\end{tabular}%
}}
\end{table}


\begin{table}[htbp]
\centering
\caption{Reproducibility study under \textbf{random initialization}: relative $L^1$, $L^2$, and $L^\infty$ errors for all four variables of the Sod shock tube problem, averaged over the three evaluation times $t=0,0.2,0.4$, over $10$  independent training runs with different random seeds. Average and standard deviation (Std) are reported. The small standard deviations confirm that WE-PINNs produces consistently accurate solutions across initializations.}
\label{tab:multirun_random}
\renewcommand{\arraystretch}{1.1}
\makebox[\textwidth][c]{%
\resizebox{\textwidth}{!}{%
\begin{tabular}{l ccc ccc ccc ccc}
\toprule
\multirow{2}{*}{\textbf{Run}} &
\multicolumn{3}{c}{\textbf{Density $\rho$}} &
\multicolumn{3}{c}{\textbf{Velocity $u$}} &
\multicolumn{3}{c}{\textbf{Pressure $p$}} &
\multicolumn{3}{c}{\textbf{Energy $E$}} \\
\cmidrule(lr){2-4} \cmidrule(lr){5-7} \cmidrule(lr){8-10} \cmidrule(lr){11-13}
& $E_{L^1}$ & $E_{L^2}$ & $E_{L^\infty}$
& $E_{L^1}$ & $E_{L^2}$ & $E_{L^\infty}$
& $E_{L^1}$ & $E_{L^2}$ & $E_{L^\infty}$
& $E_{L^1}$ & $E_{L^2}$ & $E_{L^\infty}$ \\
\midrule
Run  1 & $6.453\times10^{-4}$ & $3.360\times10^{-3}$ & $4.993\times10^{-2}$ & $4.455\times10^{-3}$ & $2.004\times10^{-2}$ & $4.563\times10^{-2}$ & $4.558\times10^{-4}$ & $2.827\times10^{-3}$ & $5.730\times10^{-2}$ & $4.423\times10^{-4}$ & $2.984\times10^{-3}$ & $6.225\times10^{-2}$ \\
Run  2 & $8.217\times10^{-4}$ & $3.654\times10^{-3}$ & $4.739\times10^{-2}$ & $6.309\times10^{-3}$ & $2.429\times10^{-2}$ & $5.040\times10^{-2}$ & $7.391\times10^{-4}$ & $3.836\times10^{-3}$ & $5.612\times10^{-2}$ & $7.048\times10^{-4}$ & $3.849\times10^{-3}$ & $5.949\times10^{-2}$ \\
Run  3 & $8.385\times10^{-4}$ & $3.841\times10^{-3}$ & $5.164\times10^{-2}$ & $5.172\times10^{-3}$ & $2.149\times10^{-2}$ & $4.815\times10^{-2}$ & $5.349\times10^{-4}$ & $3.260\times10^{-3}$ & $6.406\times10^{-2}$ & $5.176\times10^{-4}$ & $3.417\times10^{-3}$ & $6.974\times10^{-2}$ \\
Run  4 & $8.904\times10^{-4}$ & $3.825\times10^{-3}$ & $5.082\times10^{-2}$ & $5.719\times10^{-3}$ & $2.259\times10^{-2}$ & $4.736\times10^{-2}$ & $6.459\times10^{-4}$ & $3.530\times10^{-3}$ & $6.125\times10^{-2}$ & $6.281\times10^{-4}$ & $3.631\times10^{-3}$ & $6.656\times10^{-2}$ \\
Run  5 & $7.282\times10^{-4}$ & $3.452\times10^{-3}$ & $4.946\times10^{-2}$ & $4.363\times10^{-3}$ & $2.070\times10^{-2}$ & $4.685\times10^{-2}$ & $5.223\times10^{-4}$ & $3.205\times10^{-3}$ & $6.155\times10^{-2}$ & $5.102\times10^{-4}$ & $3.350\times10^{-3}$ & $6.688\times10^{-2}$ \\
Run  6 & $9.029\times10^{-4}$ & $3.844\times10^{-3}$ & $5.224\times10^{-2}$ & $6.624\times10^{-3}$ & $2.238\times10^{-2}$ & $4.946\times10^{-2}$ & $5.709\times10^{-4}$ & $3.413\times10^{-3}$ & $6.781\times10^{-2}$ & $5.454\times10^{-4}$ & $3.599\times10^{-3}$ & $7.387\times10^{-2}$ \\
Run  7 & $8.655\times10^{-4}$ & $4.050\times10^{-3}$ & $5.242\times10^{-2}$ & $5.733\times10^{-3}$ & $2.201\times10^{-2}$ & $4.704\times10^{-2}$ & $6.457\times10^{-4}$ & $3.880\times10^{-3}$ & $6.111\times10^{-2}$ & $6.278\times10^{-4}$ & $3.972\times10^{-3}$ & $6.648\times10^{-2}$ \\
Run  8 & $8.461\times10^{-4}$ & $3.842\times10^{-3}$ & $5.399\times10^{-2}$ & $5.265\times10^{-3}$ & $2.147\times10^{-2}$ & $4.738\times10^{-2}$ & $5.512\times10^{-4}$ & $3.096\times10^{-3}$ & $6.122\times10^{-2}$ & $5.328\times10^{-4}$ & $3.250\times10^{-3}$ & $6.661\times10^{-2}$ \\
Run  9 & $7.100\times10^{-4}$ & $3.773\times10^{-3}$ & $5.246\times10^{-2}$ & $4.565\times10^{-3}$ & $2.372\times10^{-2}$ & $5.102\times10^{-2}$ & $5.095\times10^{-4}$ & $3.350\times10^{-3}$ & $6.117\times10^{-2}$ & $5.017\times10^{-4}$ & $3.484\times10^{-3}$ & $6.619\times10^{-2}$ \\
Run 10 & $6.702\times10^{-4}$ & $3.601\times10^{-3}$ & $5.131\times10^{-2}$ & $4.797\times10^{-3}$ & $2.212\times10^{-2}$ & $4.734\times10^{-2}$ & $4.957\times10^{-4}$ & $3.124\times10^{-3}$ & $6.094\times10^{-2}$ & $4.829\times10^{-4}$ & $3.269\times10^{-3}$ & $6.620\times10^{-2}$ \\
\midrule
Average   & $7.919\times10^{-4}$ & $3.724\times10^{-3}$ & $5.117\times10^{-2}$ & $5.300\times10^{-3}$ & $2.208\times10^{-2}$ & $4.806\times10^{-2}$ & $5.671\times10^{-4}$ & $3.352\times10^{-3}$ & $6.125\times10^{-2}$ & $5.493\times10^{-4}$ & $3.481\times10^{-3}$ & $6.643\times10^{-2}$ \\
\textbf{Std} & $\mathbf{9.458\times10^{-5}}$ & $\mathbf{2.075\times10^{-4}}$ & $\mathbf{1.871\times10^{-3}}$ & $\mathbf{7.843\times10^{-4}}$ & $\mathbf{1.279\times10^{-3}}$ & $\mathbf{1.702\times10^{-3}}$ & $\mathbf{8.570\times10^{-5}}$ & $\mathbf{3.286\times10^{-4}}$ & $\mathbf{3.214\times10^{-3}}$ & $\mathbf{7.993\times10^{-5}}$ & $\mathbf{2.943\times10^{-4}}$ & $\mathbf{3.835\times10^{-3}}$ \\
\bottomrule
\end{tabular}%
}}
\end{table}

\end{document}